\documentclass[reqno]{amsart}

\usepackage{diagrams}
\usepackage{amssymb}
\usepackage{mathrsfs}
\usepackage{cite}
\usepackage{tikz-cd}
\usepackage{MnSymbol}
\usepackage[all,cmtip]{xy}
\usepackage{mathtools}
\usepackage{dirtytalk}

\DeclareMathOperator\C{\mathbb C}
\DeclareMathOperator\Z{\mathbb Z}
\DeclareMathOperator\R{\mathbb R}
\DeclareMathOperator\Q{\mathbb Q}
\DeclareMathOperator\N{\mathbb N}

\newtheorem{theorem}{Theorem}[section]
\newtheorem{lemma}[theorem]{Lemma}
\newtheorem{cor}[theorem]{Corollary}
\newtheorem{conj}[theorem]{Conjecture}
\newtheorem{prop}[theorem]{Proposition}
\theoremstyle{definition}
\newtheorem{definition}[theorem]{Definition}
\newtheorem{example}[theorem]{Example}

\theoremstyle{remark}
\newtheorem{remark}[theorem]{Remark}

\newcommand{\dontprint}[1]\relax

\newcommand{\ga}{\gamma}
\newcommand{\G}{{\mathbb G}}

\newcommand{\tot}{\operatorname{tot}}

\newcommand{\Om}{\Omega}
\renewcommand{\th}{\theta}
\newcommand{\hra}{\hookrightarrow}

\renewcommand{\P}{{\mathbb P}}
\newcommand{\A}{{\mathbb A}}
\newcommand{\wt}{\widetilde}
\newcommand{\ot}{\otimes}
\newcommand{\coker}{\operatorname{coker}}

\newcommand{\ul}{\underline}
\newcommand{\und}{\underline}
\newcommand{\Hom}{\operatorname{Hom}}
\newcommand{\Ext}{\operatorname{Ext}}

\newcommand{\VV}{{\mathcal V}}

\newcommand{\LL}{{\mathcal L}}

\newcommand{\OO}{{\mathcal O}}
\newcommand{\PP}{{\mathcal P}}
\newcommand{\UU}{{\mathcal U}}

\newcommand{\si}{\sigma}
\newcommand{\de}{\delta}
\newcommand{\sub}{\subset}
\newcommand{\Spec}{\operatorname{Spec}}
\newcommand{\Res}{\operatorname{Res}}
\newcommand{\ov}{\overline}

\newcommand{\om}{\omega}
\newcommand{\De}{\Delta}
\newcommand{\la}{\lambda}
\renewcommand{\a}{\alpha}
\renewcommand{\b}{\beta}
\newcommand{\Def}{\operatorname{Def}}
\newcommand{\Spf}{\operatorname{Spf}}

\newcommand{\tr}{\operatorname{tr}}
\newcommand{\rk}{\operatorname{rk}}
\newcommand{\id}{\operatorname{id}}
\newcommand{\End}{\operatorname{End}}
\newcommand{\lan}{\langle}
\newcommand{\ran}{\rangle}
\newcommand{\eps}{\epsilon}
\newcommand{\cplx}{{\mathrm{Cplx}}}
\newcommand{\perf}{\mathbf{Perf}}
\newcommand{\uperf}{\ul{\mathrm{Perf}}}

\newcommand{\gr}{\mathbf{gr}}

\newcommand{\vect}{\mathrm{Vect}}
\newcommand{\cV}{\mathcal{V}}
\newcommand{\ol}{\overline}

\newcommand{\cU}{\mathcal{U}}
\newcommand{\cM}{\mathcal{M}}
\newcommand{\cA}{\mathcal{A}}
\newcommand{\cT}{\mathcal{T}}
\newcommand{\cO}{\mathcal{O}}
\newcommand{\cI}{\mathcal{I}}
\newcommand{\ep}{\epsilon}

\newcommand{\ska}{{\mathrm{sAlg}}_k}
\newcommand{\Map}{{\mathrm{Map}}}

\newcommand{\sX}{\mathscr{X}}
\newcommand{\sB}{\mathscr{B}}
\newcommand{\sY}{\mathscr{Y}}
\newcommand{\sZ}{\mathscr{Z}}

\newcommand{\sM}{\mathscr{M}}
\newcommand{\sN}{\mathscr{N}}
\newcommand{\comod}{\mathbf{coMod}}

\newcommand{\QC}{\mathrm{Qcoh}}
\newcommand{\chain}{\mathbf{Ch}}
\newcommand{\uchain}{\underline{Ch}}
\newcommand{\bC}{\mathbf{C}^\bullet}

\newcommand{\am}{\overrightarrow{m}}
\newcommand{\aone}{\overrightarrow{1}}
\newcommand{\amu}{\overrightarrow{\mu}}
\newcommand{\fS}{\mathfrak{S}}
\newcommand{\HH}{\mathbb {H}}

\newcommand{\Cone}{\operatorname{Cone}}

\numberwithin{equation}{section}

\title{Bosonization of Feigin-Odesskii Poisson varieties}
\author{Zheng Hua}\date{}
\address{Department of Mathematics, the University of Hong Kong, Hong Kong SAR, China}
\email{huazheng@maths.hku.hk}
\author{Alexander Polishchuk}
\address{University of Oregon and National Research University Higher School of Economics, Russian Federation}
\email{apolish@uoregon.edu}

\begin{document}
\maketitle

\begin{abstract}
The derived moduli stack of complexes of vector bundles on a Gorenstein Calabi-Yau curve admits a 0-shifted Poisson structure. Feigin-Odesskii Poisson varieties are examples of such moduli spaces over complex elliptic curves. Using moduli stack of chains we construct an auxiliary Poisson varieties with a Poisson morphism from it to a Feigin-Odesskii variety. We call it  the \emph{bosonization} of Feigin-Odesskii variety. As an application, we show that the Feigin-Odesskii Poisson brackets on projective spaces (associated with stable bundles of arbitrary rank on elliptic curves) admit no infinitesimal symmetries.
\end{abstract}


\section{Introduction}
\subsection{Background and motivation}
In the seminal paper \cite{PTVV}, Pantev, Toen, Vaqui\'e and Vezzosi introduced shifted symplectic structures
in the framework  of derived algebraic geometry (see \cite{Lu}, \cite{HAGII}). Later, in a joint paper with Calaque \cite{CPTVV},
they introduced shifted Poisson structures on derived stacks. 
By introducing an appropriate shift, one can sometimes translate a problem for a Poisson stack to a problem on a symplectic stack, which in general is easier to deal with. One important result of this flavor is: if $f: \sX\to \sY$ is a Lagrangian morphism, where $\sY$ is equipped with an $n$-shifted symplectic structure, then $\sX$ is equipped with a natural $(n-1)$-shifted Poisson structure (cf. \cite[Theorem 4.22]{MS2}). The opposite of this theorem is expected to hold when $\sY$ is taken to be a formal stack. This result can be viewed as a derived version of the construction of symplectic groupoids associated to Poisson manifolds.

Following this idea, we have constructed and studied a large class of shifted Poisson stacks that arise as certain moduli stacks on Calabi-Yau varieties \cite{HP1, HP2, HP3}. Let $X$ be a Gorenstein Calabi-Yau $d$-fold, i.e. $X$ is a connected projective $d$-fold such that $\omega_X\cong \cO_X$. Let $\R\ul\cplx(X)$ be the derived moduli stack of bounded chain complex of vector bundles on $X$ up to chain isomorphisms, $\R\ul\vect^\gr(X)$ the derived moduli stack of $\Z$-graded vector bundles on $X$, and $\R\uperf(X)$ the derived moduli stack of perfect complexes on $X$. Let $p_X:\R\ul\cplx(X)\to \R\ul\vect^\gr(X)$ be the morphism that forgets the differential and $q_X: \R\ul\cplx(X)\to \R\uperf(X)$ the morphism sends a chain complex of vector bundles $C^\bullet$ to the class of $C^\bullet$ in $\R\uperf(X)$. We prove in \cite{HP1} that these maps define a Lagrangian correspondence with respect to the $(2-d)$-shifted symplectic structure constructed by Pantev, Toen, Vaqui\'e and Vezzosi \cite{PTVV}. As a consequence, $\R\ul\cplx(X)$ admits a $(1-d)$-shifted Poisson structure.

The link to classical Poisson geometry occurs when $d=1$. We obtain various Poisson structures that appear in representation theory by choosing different Calabi-Yau curves and components of the moduli stacks. This includes for example Mukai-Bottacin Poisson structure on Hilbert scheme of points on Fano surfaces, Drinfeld's standard Poisson structure on partial flag varieties and Feigin-Odesskii Poisson brackets on projective spaces. Shifted symplectic geometry provides a unified method to analyze symplectic leaves of these Poisson varieties.
 An important result in \cite{PTVV} says that if $\sX_1\to \sY$ and $\sX_2\to \sY$ are two Lagrangian morphisms, where $\sY$ is $n$-shifted symplectic, 
then the derived fiber product $\sX_1\times^h_\sY \sX_2$ is $(n-1)$-shifted symplectic. In the case $d=1$ we showed in \cite{HP1} that these assumptions are satisfied for 
$\sX_1=\R\ul\cplx(X)$, $\sY=\R\ul\vect^\gr(X)\times \R\uperf(X)$ and the residue gerbe $\sX_2$ associated with a point of $\sY$.
Therefore, we obtain a natural 0-shifted symplectic foliation of $\R\ul\cplx(X)$.

One of our principal motivations is the study of the Feigin-Odesskii Poisson brackets on projective spaces. Let $X$ be a Calabi-Yau curve and $\xi$ a simple vector bundle on $X$. Let $\sN(\xi)$ be the moduli stack of (non-splitting) complexes $\cO_X\to V$ such that $V/\cO_X\cong \xi$. One can show that the $0$-shifted Poisson structure on $\R\ul\cplx(X)$ induces one on $\sN(\xi)$.
Furthermore, it descends to a Poisson structure on the coarse moduli scheme $N(\xi)$ of $\sN(\xi)$, which is identified with the projective space $\P\Ext^1(\xi,\cO_X)$. 
When $X$ is a complex elliptic curve, the Poisson structure on $N(\xi)$ was first studied by Feigin and Odesskii in mid 90s as semi-classical limits of the elliptic algebras. The study of (algebraic) quantization of $\sN(\xi)$ is the first step towards the study of quantization of the moduli stack $\R\ul\cplx(X)$.  The importance of this problem lies in the following folklore conjecture.
\begin{conj}
Algebraic quantizations of $N(\xi)$ are Artin-Schelter regular algebras.
\end{conj}
In mid 90s, Tate and Van den Bergh proved that the quantization of $N(\xi)$ is regular when $X$ is an elliptic curve and $\xi$ is a line bundle \cite{TV96}. Since then there has been limited progress on the geometric approach of construction of regular algebras. The major obstacle to the development of this subject is the lack of general regularity criteria. We believe that for algebras arising as quantizations of $N(\xi)$, regularity should be checked via constructing another class of regular algebras that are faithfully flat extensions of our target algebras. We follow the terminology of Feigin and Odesskii to call this class of algebras \emph{Bosonization}. In some special cases of elliptic algebras, Bosonizations have been constructed (see \cite{Ode02, TV96}). One contribution of this paper is that we develop a geometric theory of bosonization for all Feigin-Odesskii Poisson brackets on projective spaces.

\subsection{Results of this paper}
Our first result is a generalization of the main result of \cite{HP1}. Let $X$ be a Gorenstein Calabi-Yau $d$-fold and let $\R\eps\uperf(X)$ be the derived moduli stack of complexes of perfect complexes on $X$. Objects of $\R\eps\uperf(X)$  are bounded chain complexes
\[
V_1\to V_2\to\ldots \to V_n
\] where $V_i$ are perfect complexes on $X$ and equivalences are defined by strict chain isomorphisms. We refer the readers to Section \ref{sec:prelim} for the precise construction of this stack via graded mixed objects. The moduli stack $\R\ul\cplx(X)$ is an open substack of $\R\eps\uperf(X)$ where $V_i$ are required to have amplitude $[0,0]$ (i.e., no nonzero cohomology). Let $\R\uperf^\gr_b(X)$ be the moduli stack of bounded $\Z$-graded objects in the category $\perf(X)$ of perfect complexes over $X$. We have stack morphisms $p_X: \R\eps\uperf(X)\to \R\uperf^\gr_b(X)$ and $q_X: \R\eps\uperf(X)\to \R\uperf(X)$.
\begin{theorem}[c.f. Theorem \ref{thm:Lag}]
Let $X/k$ be a Gorenstein Calabi-Yau $d$-fold. Then the stack map
\[
(p_X,q_X): \R\eps\uperf(X)\to \R\uperf^\gr_b(X)\times\R\uperf(X)
\]  is a Lagrangian correspondence. In particular, $\R\eps\uperf(X)$ admits a $(1-d)$-shifted Poisson structure.
\end{theorem}

For the most part we can copy the proof for $\R\ul\cplx(X)$ in \cite{HP1}, except that we need to check that $\R\eps\uperf(X)$ is a locally geometric stack locally of finite presentation. 
This is proved in Theorem \ref{thm:epsmodulialgebraic}. We also give a second proof of Theorem \ref{thm:Lag} in Appendix A using the equivalence between boundary structures and Lagrangian morphisms established by Calaque \cite{Ca14} (the idea of this proof was suggested to us by Pavel Safronov). 
Namely, we show that $\R\eps\uperf(X)$ is the mapping stack from $[\A^1\big/\G_m]\times X$ to $\R\uperf$ and there exists a boundary structure on the decomposition $[\A^1\big/\G_m]=\G_m\big/\G_m\sqcup B\G_m$. The Lagrangian structure on $(p_X,q_X)$ is induced by the boundary structure and the orbit inclusion maps $\G_m\big/\G_m\to [\A^1\big/\G_m]$  and $B\G_m\to [\A^1\big/\G_m]$. 

The new interpretation of the Lagrangian correspondence leads to an interesting corollary. Let $X$ be a Calabi-Yau curve, and let $F$ be the stack autoequivalence of $\R\eps\uperf(X)$ induced by an autoequivalence of $\perf(X)$. Then the semi-classical part of the 0-shifted Poisson structure on $\R\eps\uperf(X)$ is preserved by $F$ up to a nonzero scalar (see Appendix B). Here by the {\it semi-classical part} of a shifted Poisson structure we mean its weight 2 component (see Section \ref{sec:Poieps}).

The second part of the paper is devoted to the study of Feigin-Odesskii Poisson varieties $N(\xi)=\P\Ext^1(\xi,\OO)$ and their bosonizations. Let $X$ be a complex elliptic curve, and let $\xi$ be a rank $k$ stable vector bundle $\xi$ of degree $n>k$ (note that $gcd(n,k)=1$). Recall that the stack $\sN(\xi)$ parameterizes non-splitting complexes $\cO_X\to V$ such that $V/\cO_X\cong \xi$, equivalently nonzero morphisms $\xi\to \cO[1]$ whose mapping cone is a vector bundle (Lemma \ref{lem:exacttriangle}). Let
\[
\frac{n}{k}=n_1-\frac{1}{n_2-\frac{1}{\ldots -\frac{1}{n_p}}}.
\] be the unique representation of $n/k$ as a generalized continued fraction such that $n_i\geq 2$. We denote the successive convergents of this continued fraction as 
$$\mu_{p-1}=\frac{n(p-1)}{k(p-1)}=n_1, \ \mu_{p-2}=\frac{n(p-2)}{k(p-2)}=n_1-\frac{1}{n_2}, \ldots, \ \mu_0=\frac{n(0)}{k(0)}=\frac{n}{k},$$
and we set in addition $n(p)=1$, $k(p)=0$, $\mu_p=\infty$
(see 
Section \ref{sec:contfrac} for details). Let 
\[
\am=(m_1,\ldots,m_p)
\] be any sequence of positive integers. 
We denote by $\sB(\xi,\am)$ the moduli stack of sequences
\[
\xi=\xi_0\to \xi_1^{ss}\to\ldots\to \xi_p^{ss}\to \cO_X[1]
\]  where for $i=1,\ldots,p$, $\xi_i^{ss}$ is a semi stable bundle of rank $m_i\cdot k(i)$ and degree $m_i\cdot n(i)$ (resp., a torsion sheaf of length $m_p$ for $i=p$), and all the morphisms are nonzero. The iterated composition defines a rational morphism 
\[\xymatrix{
\beta(\xi,\am): \sB(\xi,\am)\ar@{-->}[r] & \sN(\xi).}
\]
We prove that
\begin{theorem}
The moduli stack $\sB(\xi,\am)$ admits a 0-shifted Poisson structure such that $\beta(\xi,\am)$ is a semi-classical Poisson morphism on its defining domain. 
\end{theorem}
This is a special case of Proposition \ref{betapoisson}. A morphism between shifted Poisson stacks is called {\it semi-classically Poisson} if the Poisson morphism condition holds for the weight 2 component (see Section \ref{sec:Poieps} for the precise definition). In particular a semi-classically Poisson morphism descends to an ordinary Poisson morphism between coarse moduli schemes. We study the following two special cases in more detail:
\begin{enumerate}
\item[$(1)$] $\am=\aone:=(1,\ldots,1)$;
\item[$(2)$] $k=1$.
\end{enumerate}

In the case (1), $\xi_i^{ss}$ are stable bundles. We prove that the Poisson structure on the coarse moduli scheme $B(\xi)$ of $\sB(\xi):=\sB(\xi,\aone)$ is zero (Theorem \ref{thm:zero}) and the image of 
the map $B(\xi)\to N(\xi)$, induced by $\beta(\xi,\aone)$, is isomorphic to a fiber of the addition map $X^{p+1}/G\to X$, where $G$ is a subgroup of the symmetric group (Theorem \ref{thm:ni>2} and Theorem \ref{thm:general ni}). 

We expect the image of $\beta(\xi,\aone)$ to be a connected component  of maximal dimension of the vanishing locus of the Feigin-Odesskii Poisson structure on the projective space  $N(\xi)$. 
We prove that it is an irreducible component of this vanishing locus.  
As a consequence, we show that the Poisson structure on $N(\xi)$ admits no infinitesimal symmetries, i.e., every Poisson vector field on $N(\xi)$ vanishes (see Theorem \ref{thm:nosym}). 
This was known only in the case when $k=1$ and $n$ odd (see \cite[Prop.\ 15.1]{Polishchuk}, \cite[Thm.\ 4.6]{Pym-Schedler}).

In the case (2), we have $p=k=1$ and $\am=m$. In this case we compute the Poisson bracket on $\sB(\xi,m)$ explicitly (see Propositions \ref{prop:poi-p=1} and 
\ref{Pois-br-bos-coord-prop}). 
We recover the Poisson bracket in \cite[Section 2.5]{Ode02}. This shows that $\sB(\xi,m)$ is the semi-classical limit of the Bosonization $B_{m,n}(\eta)$ of the elliptic algebra $Q_{n,1}(X,\eta)$ 
defined in \cite[Section 2.3]{Ode02}. 

\subsection{Organization of the paper}
Section \ref{sec:prelim} provides the derived geometry preliminaries for the construction of $\R\eps\uperf(X)$. Section \ref{sec:poisson str} is devoted to the construction of a shifted Poisson structure on $\R\eps\uperf(X)$ following the approach of \cite{HP1} (the second approach is given in Appendix A). 

In the rest of the paper we work with the weight 2 truncation of this shifted Poisson structure.
Since the 0-shifted Poisson stacks that appears in our applications are gerbes over smooth varieties, this truncation determines the Poisson geometry of the underlying coarse moduli schemes. Moreover, the weight 2 truncation of the shifted Poisson structure on $\R\eps\uperf(X)$ is determined by an explicit chain map (see Section \ref{sec:Poieps}), which allows us to compute Poisson bracket on the coarse moduli scheme effectively. In Section \ref{sec:boson}, we construct the moduli stack $\sB(\xi,\am) $ and prove that $\beta(\xi,\am): \sB(\xi,\am)\to \sN(\xi)$ is a semi-classically Poisson morphism (based on the techniques developed in Section \ref{sec:chains} and \ref{sec:chaind=1} about the moduli stack of chains). 

In Section \ref{sec:m=1} we specialize to the case $\am=\aone=(1,\ldots,1)$ and study the geometry of $\sB(\xi,\aone)$ and of the map $\beta(\xi,\aone)$. 
In Section \ref{van-vec-field-sec} we apply the results of this study to prove the vanishing of Poisson vector fields on $N(\xi)$. 
In Section \ref{sec:p=1} we study the case when $\xi$ is a line bundle. In Appendix B, we prove that the weight 2 truncation of the 0-shifted Poisson structure on $\R\eps\uperf(X)$ is preserved up to a nonzero scalar under autoequivalences of $\perf(X)$. The proof uses the content of Appendix A. Although this result is not used explicitly in this paper, it plays a fundamental role in the analysis of Poisson geometry of $\R\eps\uperf(X)$ since it relates different components of $\R\eps\uperf(X)$. In Appendix C we rigidify the bosonization moduli stack $\sB(\xi,\aone)$ 
and compute the determinant line bundle of the universal family. 
We expect this computation to play a role in quantization of $\sB(\xi,\am)$. 

\subsection{Acknowledgment}
We thank Pavel Safronov for sharing the idea of proving Theorem \ref{thm:Lag} using the boundary structure on $[\A^1\big/\G_m]$.
Z.H. is partially supported by the GRF grant no. 17303420 of University Grants Committee of Hong Kong SAR, China. 
A.P. is partially supported by the NSF grant DMS-2001224, 
and within the framework of the HSE University Basic Research Program and by the Russian Academic Excellence Project `5-100'.

\section{Graded mixed perfect complexes and its moduli stack}\label{sec:prelim}

\subsection{Graded mixed objects}
We follow the notation of \cite[Section 1]{CPTVV}. Let $k$ be a Noetherian commutative $\Q$-algebra. Let $C(k)$ be the category of unbounded dg-$k$-modules. Equip $C(k)$ with the standard model structure with weak equivalences being quasi-isomorphisms and fibrations being epimorphisms of cochain complexes. Let $M$ be a symmetric monoidal model category with a $C(k)$-enrichment.

Let $B=k[t,t^{-1}]\ot_k k[\ep]$ where $k[\ep]$ is the graded symmetric algebra with $|\ep|=-1$. $B$ is equipped with a Hopf algebra structure such that 
\[
\Delta_B(t)=t\ot t\hspace {2cm} \Delta_B(\ep)=\ep\ot 1+t\ot \ep
\] and the counit
\[
\varepsilon_B(t)=1,\hspace{2cm} \varepsilon_B(\ep)=0.
\]
 
A \emph{graded mixed object} in category $M$ is a $B$-comodule in $M$. More explicitly, one such object is a $\Z$-family of $\{E(p)\}_{p\in\Z}$ of objects in $M$ together with morphisms in $M$
\[
\ep=\{\ep_p: E(p)\to E(p+1)[1]\}_{p\in\Z}
\] where $[1]$ is defined by the $C(k)$-enrichment of $M$, and $\ep^2=0$. We write $(E,\ep)$ for the family together with the differential. A morphism 
\[
f: (E,\ep)\to (F,\ep)
\] is a family of maps $\{f(p): E(p)\to F(p)\}_{p\in\Z}$ in $M$ that commutes with $\ep$. We call a graded mixed object in $M$ \emph{bounded} if $E(p)=0$ except for finitely many $p$. Denote the category of graded mixed objects in $M$ by $\ep-M^\gr$. 

The category $M^\gr:=\prod_{p\in\Z} M$ is naturally a symmetric monoidal model category enriched in $C(k)$, inherited from $M$. There is a forgetful functor
\[
\ep-M^\gr\to M^\gr
\] forgetting the $k[\ep]$-structure. Equip $\ep-M^\gr$ with the symmetric monoidal model structure through the forgetful functor. Given a triangulated dg category $T$, following \cite{TV07} we denote the category of perfect (or compact) objects by $T_{pe}$. Suppose  $M$ is triangulated and admits arbitrary coproduct. An object of $\ep-M^\gr$ is called \emph{perfect or compact} if it is a compact object in $M^\gr$. Denote by $\ep_{pe}-M^\gr$ the subcategory of $\ep-M^\gr$ consisting of perfect objects.

\begin{example}
Assume that $M=C(k)$. An object of $\ep-C(k)^\gr$ is called a graded mixed complexes of dg-$k$-modules. Such an object is nothing but a double complex of $k$-modules whose horizontal differential are $\ep$. And a morphism is just a morphism of double complexes. It is perfect if and only if $E(p)$ is a perfect dg-$k$-module for all $p\in\Z$ and $E(p)=0$ except for finitely many $p$.
\end{example}

\begin{example}
Assume that $M=\QC(X)$, the category of quasi-coherent complexes on a projective $k$-scheme $X$. An object of $\ep-M^\gr$ is called a graded mixed complexes of quasi-coherent complexes on $X$.   It is perfect if and only if $E(p)$ is a perfect complex on $X$ for all $p\in\Z$ and $E(p)=0$ except for finitely many $p$. If we further assume that $E(p)$ has perfect amplitude $[0,0]$ for all $p$ then such an object is simply a bounded chain complex of vector bundles. Note that two such chains are isomorphic in $\ep-M^\gr$ means they are \emph{chain isomorphic} (not quasi-isomorphic!). Moduli stack of such complexes were studied in \cite{HP1}.
\end{example}

There are (at least) two ways to enrich $\ep-M^\gr$. Let $E, F$ be two mixed graded objects.
\begin{enumerate}
\item[(1)] We define the external hom by
\[
\Hom^{\N}_\ep(E,F):=\prod_{p\in\Z}\Big(\Hom^{\N}_\ep(E,F)(p)\Big)
\] where 
\[
\Hom^{\N}_\ep(E,F)(p)=\prod_{q\in\N}\Hom_{M}(E(q),F(q+p)).
\] The differential 
\[
\ep(p): \Hom^{\N}_\ep(E,F)(p)\to \Hom^{\N}_\ep(E,F)(p+1)[1]
\] is defined by the adjoint action of $\ep$ on $E$ and $F$.
\item[(2)] We define the external hom by
\[
\Hom^{\Z}_\ep(E,F):=\prod_{p\in\Z}\Big(\Hom^{\Z}_\ep(E,F)(p)\Big)
\] where 
\[
\Hom^{\Z}_\ep(E,F)(p)=\prod_{q\in\Z}\Hom_{M}(E(q),F(q+p)).
\] The differential 
\[
\ep(p): \Hom^{\Z}_\ep(E,F)(p)\to \Hom^{\Z}_\ep(E,F)(p+1)[1]
\] is defined by the adjoint action of $\ep$ on $E$ and $F$.
\end{enumerate}
If we compare these two enrichments for $M=C(k)$, then we find that for the first one the enriched hom space is a double complex concentrating in the right half-plane while for the second one the hom space is a double complex spreading out the whole plane. In other words, both Hom spaces are $\eps-C(k)^\gr$-enriched but the first one concentrates in nonnegative weights.  There is a  $\eps-C(k)^\gr$ morphism  
\[
\Hom^{\N}_\ep(E,F)\to \Hom^{\Z}_\ep(E,F).
\]
If we take the total complex then they provides two distinct $C(k)$-enrichments.
In this paper, we will consider the first enrichment. In \cite{CPTVV} the authors consider the second. Different enrichment will induce different derived structure on the moduli stack of objects. Note that the forgetful functor $\ep-M^\gr\to M^\gr$ is $C(k)$-enriched if we use the first enrichment on the source and the $C(k)$-enrichment that encodes only weight zero morphisms on the target.

\subsection{Moduli stack of graded mixed objects}
The goal is to prove the following theorem.
\begin{theorem}\label{thm:epsmodulialgebraic}
Let $X$ be a (flat) projective $k$-scheme. Let $M=\QC(X)$ be the category of quasi-coherent complexes on $X$. Then perfect mixed graded objects in $M$ form a locally geometric stack locally of finite presentation. 
\end{theorem}

From now on, we set $M=\QC(X)$. 
Denote by $\ska$ the category of simplicial commutative $k$-algebras.
 For $A\in \ska$, consider the simplicial presheaf 
\[
A\mapsto N\Big(\ep_{pe}-\QC(X\times \Spec A)^\gr\Big),
\]
where $N(-)$ stands for nerve. Denote by $\R\ul{\ep_{pe}-\QC(X)^\gr}$ the $D^-$-stack corresponds to the above simplicial presheaf.  
 
Let $\A^1:=\Spec k[s]$ be the ordinary affine line. The multiplicative group $\G_m=\Spec k[t,t^{-1}]$ acts on $\A^1$ by the comultiplication
\[
s\mapsto s\ot t.
\]
Denote by $[\A^1\big/\G_m]$ the corresponding global quotient stack.

\begin{lemma}\label{lem:1}
There is an equivalence of stacks
\[
\R\ul{\ep_{pe}-\QC(X)^\gr}\simeq \ul{\Map}_{st}([\A^1\big/\G_m]\times X,\R\uperf)
\] where $\ul{\Map}_{st}$ is the internal hom of the category of derived stacks.
\end{lemma}
\begin{proof}
Since $\QC(\Spec k)=C(k)$,  for $A\in \ska$
\begin{align*}
\R\ul{\ep_{pe}-\QC(X)^\gr}(A)&\simeq  \ul{\Map}_{st}(X\times \Spec A,\R\ul{\ep_{pe}-C(k)^\gr}).
\end{align*}
Recall that there is an equivalence  of dg categories
\[
\ep-M^\gr\simeq B-\comod_M
\]
where $B$ is the Hopf algebra defined in the previous section and $B-\comod_M$ is the category of dg $B$-comodules in $M$. We refer to \cite[Section 1.1.1]{CPTVV} for the details of this equivalence. When $M=C(k)$, there is a quasi-equivalence of dg categories
\[
B-\comod_{C(k)}\simeq {\mathrm{Rep}}\Big(\G_m\ltimes \G_a[-1]\Big)
\] where $\G_m\ltimes \G_a[-1]$ is a group stack (See \cite[Remark 1.1.2]{CPTVV}). We may identify category of representations of the group stack and the category of quasi-coherent sheaves over its classifying stack:
\[
{\mathrm{Rep}}\Big(\G_m\ltimes \G_a[-1]\Big)]\simeq \QC\Big( B(\G_m\ltimes \G_a[-1])\Big).
\]
As algebraic stacks, $B(\G_m\ltimes \G_a[-1])$ is equivalent with the global quotient stack $[\A^1\big/\G_m]$. 

Therefore we have a quasi-equivalence of categories
\[
\ep_{pe}-C(k)^\gr\simeq L_{pe}\Big([\A^1\big/ \G_m]\Big)=\perf\Big([\A^1\big/ \G_m]\Big)
\] \footnote{Later, we give an explicit equivalence in Prop \ref{q=genericfiber}.}
and equivalence of $D^-$-stacks
\[
\ul{\Map}_{st}(X\times \Spec A,\R\ul{\ep_{pe}-C(k)^\gr})\simeq \R\uperf\Big([\A^1\big/ \G_m]\times X\times \Spec A\Big)
\]

By \cite[Prop 3.4]{TV07}, the moduli functor $\cM_{-}$ from the opposite of the homotopy category of dg categories to the homotopy category of $D^-$-stacks has a left adjoint $L_{pe}$, i.e. for any $D^-$-stack $F$ and dg category $T$, we have a equivalence in the homotopy category of simplicial sets
\[
\Map_{dg-cat^{op}}(L_{pe}(F),T)\simeq \Map_{st}(F,\cM_T)
\] where $\Map_{st}$ is the external hom of $D^-$-stacks. Applying the adjunction, we get
\begin{align*}
&\R\uperf\Big([\A^1\big/ \G_m]\times X\times \Spec A\Big)\\
&\simeq \Map_{dg-cat}\Big({\bf{1}}, L_{pe}([\A^1\big/ \G_m]\times X\times \Spec A)\Big) \\
&\simeq \ul\Map_{st}([\A^1\big/ \G_m]\times X\times \Spec A,\R\uperf)\\
&\simeq \ul{\Map}_{st}([\A^1\big/\G_m]\times X,\R\uperf)(A)
\end{align*}
Here ${\bf{1}}$ is the dg category consisting of one object whose endomorphism is $k$. It is the initial object in the model category of dg categories.
\end{proof}

\begin{proof}[Proof of Theorem \ref{thm:epsmodulialgebraic}] 
Let $\sX$ and $\sY$ be derived stacks over $k$. Suppose that $\sX$ is \emph{formally proper} (see \cite[Definition 1.1.3]{HL-P19}) and of finite Tor amplitude over $k$ and $\sY$ is locally geometric and locally of finite presentation over $k$. Halpern-Leistner and Preygel proved that the derived mapping stack $\ul{\Map}_{st}(\sX,\sY)$ is locally geometric and locally of finite presentation over $k$ (\cite[Theorem 5.1.1]{HL-P19}). The coarse moduli space morphism $[\A^1/\G_m]\to \Spec k$ is a good quotient in the sense of \cite{Alper09}.  Using the good quotient property, one can show that the global quotient stack $[\A^1\big/\G_m]$ is formally proper (see \cite[Proposition 4.2.5]{HL-P19}). It is well known that $\R\uperf$ is locally geometric and locally of finite presentation over $k$ (cf. \cite[Proposition 3.7]{TV07}). Then the claim follows from Lemma \ref{lem:1} and the above mentioned theorem of Hapern-Leistner and Preygel.
\end{proof}

\section{Poisson structure on moduli stack of graded mixed perfect complexes}\label{sec:poisson str}
\subsection{Poisson structure on $\R\eps\uperf$}
For simplicity we write
\[
\R\eps\uperf=\R\ul{\eps_{pe}-C(k)^\gr}
\] for the $D^-$-stack of graded mixed objects in the category of perfect dg-$k$-modules. By Lemma \ref{lem:1}, it is equivalent to the mapping stack $\ul{\Map}_{st}([\A^1\big/\G_m],\R\uperf)$. The main result of this section is the following.

\begin{theorem}\label{thm:mixmoduliPoisson}
The $D^-$-stack $\R\eps\uperf$ admits a $1$-shifted Poisson structure.
\end{theorem} 

We refer the readers to Sections 1 and 2 of \cite{MS2} for the definitions of an {\it $n$-shifted Poisson structure} on derived stacks, an {\it $n$-shifted coisotropic structure} on a stack morphism and an {$n$-shifted Poisson morphism} (by definition, the structure of a Poisson morphism is given by a coisotropic structure on its graph). 

Compared with the notions of a shifted symplectic structure and of a Lagrangian morphism, the definition of a shifted Poisson structure is  more complicated mainly due to the lack of functoriality of polyvector fields. 
The tools developed in \cite{MS2} allow to construct shifted Poisson structures indirectly by constructing a Lagrangian morphism into a shifted symplectic stack 
(this was also the approach used in \cite[Section 3.1]{HP1}). This is based on the following result.

\begin{theorem}\label{thm:MS}
Suppose $\sX, \sY$ are locally geometric stacks locally of finite presentation.
Let $f:\sX\to \sY$ be a stack morphism. Suppose that $\sY$ is equipped with an $n$-shifted symplectic form $\omega$ and $f$ is Lagrangian. Then $\sX$ is equipped with a canonical $(n-1)$-shifted Poisson structure.
\end{theorem}
\begin{proof}
A Lagrangian structure on $(\sX,\sY,f,\omega)$ is equivalent to a nondegenerate coisotropic structure (\cite[Therorem 4.22]{MS2}), and
a coisotropic structure on $f$ induces a canonical $(n-1)$-shifted Poisson structure on $\sX$ (\cite[Section 2.1]{MS2}),  
\end{proof}

The following result, usually called the Lagrangian (resp. coisotropic) intersection theorem, is very useful for the construction of shifted Poisson stacks from given ones.
\begin{theorem}\cite[Theorem 2.9]{PTVV}\cite[Theorem 3.6]{MS2}\label{thm:fiberprod}
Let $f_1: \sX_1\to \sY$ and $f_2: \sX_2\to \sY$ be two Lagrangian (resp. coisotropic) morphisms with respect to the $n$-shifted symplectic (resp. Poisson) structure on $\sY$.  Then  $\sX_1\times^h_{\sY,f_1,f_2}\sX_2$ is equipped with a $(n-1)$-shifted symplectic (resp. Poisson) structure, and the projections $\sX_1\times^h_{\sY,f_1,f_2}\sX_2\to \sX_1$ and $\sX_1\times^h_{\sY,f_1,f_2}\sX_2\to \sX_2$ are Poisson morphisms.
\end{theorem}


There are two ways to prove Theorem \ref{thm:mixmoduliPoisson}. The first one which we sketch below is a copy of the proof of \cite[Theorem 3.13]{HP1}
 modulo Theorem \ref{thm:epsmodulialgebraic}. The second proof, suggested to us by Safronov, is conceptionally neat. It uses Lemma \ref{lem:1} and the equivalence between Lagrangian correspondences and boundary structures established by Calaque \cite{Ca14}. We include this proof in the Appendix.

\begin{proof}[Sketch of proof of Theorem \ref{thm:mixmoduliPoisson}]
Recall that $\R\uperf^\gr$ is the infinite product $\prod_{\Z}\R\uperf$ parameterizing graded objects in the category of perfect dg-$k$-modules. And denote by $\R\uperf^{\gr}_b$ the substack that consists of graded objects $(E(i))_{i\in\Z}$ such that $E(i)=0$ except for finitely many $i$. The forgetful functor $\ep-C(k)^\gr\to C(k)^\gr$ induces a stack map
\[
p: \R\eps\uperf\to \R\uperf^\gr.
\] It factors through the substack $\R\uperf^\gr_b\subset \R\uperf^\gr$. Recall that an object of $\ep-C(k)^\gr$ is a double complex with horizontal differential given by $\ep$. Taking the total complex defines a functor 
\[
\ep-C(k)^\gr\to C(k),
\] sending   perfect graded mixed objects in $C(k)$ to perfect dg-$k$-modules. We denote the corresponding stack map by
\[
q: \R\eps\uperf\to \R\uperf.
\]
If we restrict to objects $\Big((E(i))_{i\in\Z},\ep\Big)$ such that $E(i)$ are of perfect amplitude $[0,0]$ then we get a substack $\R\ul\cplx$ consisting of bounded complexes of vector bundles over $\Spec k$ (up to chain isomorphisms).

By the work of Pantev-To\"en-Vaqui\'e-Vezzosi, $\R\uperf$ admits a canonical $2$-shifted symplectic structure $\omega$. The product form $(\omega,\omega,\ldots)$ defines a $2$-shifted symplectic structure on $\R\uperf^\gr_b$. According to Theorem \ref{thm:epsmodulialgebraic}, $\R\eps\uperf$ is locally geometric and locally of finite type. By the above theorem of Melani and Safronov, it suffices to show that with respect to these two symplectic structures, the stack map 
\[
(p,q): \R\eps\uperf\to \R\uperf^\gr_b\times\R\uperf
\] is a Lagrangian correspondence. 
When restricted to $\R\cplx\subset \R\eps\uperf$, this is proved in \cite[Theorem 3.13]{HP1}. The proof extends to $\R\eps\uperf$ in an obvious way.
\end{proof}

\subsection{Poisson structure on $\R\eps\uperf(X)$ for a Gorenstein CY curve $X$}\label{sec:Poieps}
Using transgression, we may transfer a symplectic structure (resp. a Lagrangian structure) to the mapping stack with Calabi-Yau source. The symplectic case was first proved in \cite[Theorem 2.5]{PTVV}. The Lagrangian case can be proved using similar argument (cf. \cite[Theorem 2.10]{Ca14}. Recall a connected projective $k$-scheme $X$ is called \emph{Gorenstein Calabi-Yau} if the dualizing complex $\omega_X$ is invertible and $\omega_X\cong\cO_X$.

A stack morphism $\sX\to \sY_1\times \sY_2$, where $\sY_1$ and $\sY_2$ are equipped with $n$-shifted symplectic structures, 
is called a Lagrangian correspondence if it is a Lagrangian morphism when $\sY_2$ is equipped with the negative of its symplectic structure.  

\begin{theorem}\label{thm:Lag}
Let $X/k$ be a Gorenstein Calabi-Yau $d$-fold. Then the stack map
\[
(p_X,q_X): \R\eps\uperf(X):=\ul\Map_{st}(X,\R\eps\uperf)\to \R\uperf^\gr_b(X)\times\R\uperf(X)
\] induced by $(p,q): \R\eps\uperf\to \R\uperf^\gr_b\times \R\uperf$, is a Lagrangian correspondence. In particular, $\R\eps\uperf(X)$ admits a $(1-d)$-shifted Poisson structure.
\end{theorem}
\begin{proof}
By \cite[Lemma 3.4]{HP3}, $X$ admits an $\cO$-orientation of degree $d$ in the sense of \cite[Definition 2.4]{PTVV}. Applying Theorem \ref{thm:epsmodulialgebraic}, Theorem \ref{thm:mixmoduliPoisson} and \cite[Theorem 2.5]{PTVV}, we prove the desired statement.
\end{proof}

There is a useful variation of the Theorem \ref{thm:Lag}. Let $S\subset \Z$ be a subset and $\ol{S}$ be its complement. Denote by $p^{S}$ the map from $\R\eps\uperf(X)$ to $\prod_{i\in S} \R\uperf(X)$ sending $E$ to $E(i)$ with $i\in S$, and by $q^{\ol{S}}:=(p^{\ol{S}},q)$ be the map to 
\[
\Big(\prod_{i\in \ol{S}} \R\uperf(X)\Big)\times \R\uperf(X).
\]

\begin{theorem}\label{thm:basechange}
Equip $\Big(\prod_{i\in \ol{S}} \R\uperf(X)\Big)_b\times \R\uperf(X)$ with the product symplectic structure. The following holds.
\begin{enumerate}
\item[$(a)$]
The stack map 
\[
(p^{S},q^{\ol{S}}):\R\eps\uperf(X)\to \Big(\prod_{i\in S} \R\uperf(X)\Big)_b\times\Big(\Big(\prod_{i\in \ol{S}} \R\uperf(X)\Big)_b\times \R\uperf(X)\Big)
\] is a Lagrangian correspondence. Here the low index $b$ refers to the substack consisting of objects with only finitely many nonzero components.
\item[$(b)$] Let $f: Z\to  \Big(\prod_{i\in S} \R\uperf(X)\Big)_b$  be  a Lagrangian morphism. Then the base change 
\[
\xymatrix{
Z_\eps \ar[rr]^{f_\eps}\ar[d] &&\R\eps\uperf(X)\ar[d]^{p^S}\ar[r]^{q^{\ol{S}}} & \Big(\prod_{i\in \ol{S}} \R\uperf(X)\Big)_b\times \R\uperf(X)\\
Z\ar[rr]^f & &  \Big(\prod_{i\in S} \R\uperf(X)\Big)_b
}
\]  is a Poisson morphism, where the Poisson structure on $Z_\eps$ is defined by the (Lagrangian) composition $q^{\ol{S}}\circ f_\eps$.
\item[$(c)$] Let $g: W\to  \Big(\prod_{i\in \ol{S}} \R\uperf(X)\Big)_b \times \R\uperf(X)$ be  a Lagrangian morphism. Then the base change 
\[
\xymatrix{
W_\eps \ar[rr]^{g_\eps}\ar[d] &&\R\eps\uperf(X)\ar[d]^{q^{\ol{S}}}\ar[r]^{p^{S}} & \Big(\prod_{i\in S} \R\uperf(X)\Big)_b\\
W\ar[rr]^g& &  \Big(\prod_{i\in \ol{S}} \R\uperf(X)\Big)_b\times \R\uperf(X)
}
\]  is a Poisson morphism, where the Poisson structure on $W_\eps$ is defined by the (Lagrangian) composition $p^{S}\circ g_\eps$.
\end{enumerate}
\end{theorem}
\begin{proof}
We may copy the proof of \cite[Corollary 3.20]{HP1}, which uses Theorem \ref{thm:fiberprod} and the fact that the composition of two (shifted) Lagrangian correspondence is a Lagrangian correspondence (cf. \cite[Theorem 4.4]{Ca14}).
\end{proof}

Let $f:\sX\to \sY$ be a stack morphism. Suppose that $\sY$ is equipped with an $d$-shifted symplectic form $\omega$ and $f$ is Lagrangian. We then have a diagram
\begin{align}\label{dig:bivector}
\xymatrix{
T_\sX\ar[r]\ar@{-->}[rrd]_{\Theta_{\omega,h} \hspace{1cm}} & f^*T_{\sY}\ar[r]\ar[d]^{f^*\omega} & T_f \\
L_\sX[d]& f^*L_\sY[d]\ar[l] &L_f[d]\ar[l] & L_\sX[d-1]\ar[l]
}
\end{align}
Recall that an isotropic structure on $f$ is a homotopy $h: f^*\omega \sim 0$, which determines a dash arrow $\Theta_{\omega,h}$ from $T_\sX$ to the relative cotangent complex $L_f[d]$. The Lagrangian property says that $\Theta_{\omega,h}$ a quasi-isomorphism. Then we have a morphism in the derived category
\[
\pi_{\omega,h}: L_\sX[d-1]\to L_f[d]\to T_\sX
\] by composing $(\Theta_{\omega,h})^{-1}: L_f[d]\to T_{\sX}$ with the canonical map $L_\sX[d-1]\to L_f[d]$. We call $\pi_{\omega,h}$ (sometimes we write $\pi$ for simplicity) the underlying bivector of the $(d-1)$-shifted Poisson structure. This is the weight 2 component of the shifted Poisson structure (see \cite[Section 4.1]{MS2}). 

Let $\sX_1, \sX_2$ be two stacks.  Suppose that $\sX_1, \sX_2$ are equppied with $d$-shifted Poisson structure $\Pi_1, \Pi_2$. A stack morphism $f:\sX_1\to \sX_2$ is called \emph{semi-classically Poisson}  if we have a commutative diagram
\[
\xymatrix{
L_{\sX_1}\ar[r]^{\Pi_1} & T_{\sX_1}[-d]\ar[d] \\
f^*L_{\sX_2}\ar[u]\ar[r]^{f^*\Pi_2} & f^*T_{\sX_2}[-d]
}
\] in the derived category. Suppose $\sX_2$ is equipped with a $d$-shifted Poisson structure. Then $f:\sX_1\to \sX_2$ is called \emph{semi-classically coisotropic}  if the morphism 
\[
L_f\to f^*L_{\sX_2}\to f^*T_{\sX_2}[-d]\to T_f[-d]
\]
is zero in the derived category.
 By \cite[Theorem 4.22]{MS2}, a Lagrangian morphism is equivalent with a non-degenerate coisotropic morphism.
\begin{remark} 
Semi-classical Poisson morphism and semi-classical coisotropic morphism can be viewed as weight 2 truncation of the (derived) Poisson morphism and coisotropic morphism. A  0-shifted Poisson structure  is  a chain $\pi_2,\pi_3,\ldots, \pi_i,\ldots$ with $\pi_i\in \bigwedge^i T_{\sX_2}$ of degree $2-i+d$. The condition on the bivector $\pi_2$ is $[\pi_2,\pi_2]=d\pi_3$ where $d$ is the internal differential on $T_{\sX_2}$.  A coisotropic structure on $f$ gives a chain $f_2,f_3,\ldots,f_i,\ldots$ such that $f_i\in \bigwedge^i T_f$ of degree $1-i+d$. The weight 2 part satisfies the condition that the composition 
\[
L_f\to f^*L_{\sX_2}\rTo{f^*\pi_2} f^*T_{\sX_2}[-d]\to T_f[-d]
\] is homotopic to zero by $f_2$ (see \cite[Theorem 2.7]{MS2}). Therefore a coisotropic structure on $f$ determines a semi-classical coisotropic morphism. 
Poisson morphism is defined via putting a coisotropic structure on the graph (see \cite[Theorem 2.9]{MS2}). And we can check similarly that a Poisson morphism determines a semi-classical Poisson morphism by taking the weight 2 truncation.
\end{remark}

\begin{prop} \label{poisson+coisotropic}
Let $f:\sX\to\sY$ and $g:\sY\to\sZ$ be stack morphisms where $\sY$ and $\sX$ are $d$-shifted Poisson.
Suppose that $f$ is semi-classically coisotropic and $g$ is semi-classically Poisson then $g\circ f:\sX\to\sZ$ is semi-classically coisotropic. 
\end{prop}
\begin{proof}
Since the natural map $T_\sX\to (g\circ f)^* T_\sZ$ factors through the natural maps $T_\sX\to f^*T_\sY\to f^*g^*T_\sZ$, we have a commutative diagram of exact triangles
\[
\xymatrix{
 L_f\ar[r] &f^*L_\sY\ar[r] &L_\sX\\
L_{g\circ f}\ar[r]\ar[u] & f^*g^*L_\sZ\ar[r]\ar[u] & L_\sX\ar[u]^=
} \hspace{1cm}
\xymatrix{
 T_\sX\ar[r]\ar[d]^= &f^*T_\sY\ar[r]\ar[d] &T_f\ar[d]\\
T_\sX\ar[r] & f^*g^*T_\sZ\ar[r] & T_{g\circ f}
} 
\] dual to each other.
 Since $g$ is Poisson we have commutative diagram
\[
\xymatrix{
f^*L_{\sY}\ar[r]^{f^*\Pi_{\sY}} &f^* T_{\sY}[-d]\ar[d] \\
f^*g^*L_{\sZ}\ar[u]\ar[r]^{f^*g^*\Pi_{\sZ}} & f^*g^*T_{\sZ}[-d]
}
\] Composing these three diagram, we obtain a diagram
\[
\xymatrix{
L_f\ar[r] &f^*L_{\sY}\ar[r]^{f^*\Pi_{\sY}} &f^* T_{\sY}[-d]\ar[d]\ar[r]&T_f[-d]\ar[d] \\
L_{g\circ f}\ar[r]\ar[u] &f^*g^*L_{\sZ}\ar[u]\ar[r]^{f^*g^*\Pi_{\sZ}} & f^*g^*T_{\sZ}[-d]\ar[r]& T_{g\circ f}[-d]
}
\]
To check $g\circ f$ is semi-classically  coisotropic is same as checking the composing of the maps in the lower row is zero, while $f$ is semi-classically coisotropic says the composition of maps in the upper row is zero. Then the claim follows from the commutativity of the diagram.
\end{proof}

\begin{prop}\label{prop:composition of Lag}
Let $(q_0,p_1): \sX_{01}\to \sX_{0}\times \sX_{1}$ and $(p'_1,q_2): \sX_{12}\to \sX_{1}\times \sX_{2}$ be two Lagrangian correspondences with respect to the 1-shifted symplectic structures on $\sX_{0},\sX_{1},\sX_{2}$. Let us define $\sX_{02}$ from the cartesian diagram
\[\xymatrix{ \sX_{02}\ar[r]^{f_1}\ar[d]^{f'_1} & \sX_{12}\ar[d]^{p'_1}\\
\sX_{01}\ar[r]^{p_1} & \sX_{1}
}\] 
Then
\begin{enumerate}
\item[(1)] $(q_0f'_1, q_2f_1):\sX_{02}\to \sX_{0}\times \sX_{2}$ is a Lagrangian correspondence, in particular $\sX_{02}$ admits a 0-shifted Poisson structure.
\item[(2)] The morphism $(f'_1,f_1):\sX_{02}\to \sX_{01}\times \sX_{12}$ is semi-classically Poisson.
\end{enumerate}
\end{prop}
\begin{proof}
A proof of (1) can be found in \cite[Theorem 4.4]{Ca14}. To prove (2), we need to recall the construction of the weight 2 part of the Poisson structure on $\sX_{02}$.  In the rest of the proof, by a \emph{commutative diagram} we mean commute up to homotopy. Since we work in the semi classical setting, there is no need to specify the explicit homotopy.

By \cite[Theorem 4.22]{MS2}, the space of Lagrangian structures is equivalent with the space of non degenerate coisotropic structures. By \cite[Theorem 4.12]{MS2}, the space of symplectic structures is equivalent with the space of non degenerate Poisson structures. For $i=0,1,2$ we denote by $\omega_i$ the weight 2 part of the $(-1)$-shifted Poisson structure on $\sX_i$ and by $\pi_i$ the corresponding non degenerate Poisson structures. For simplicity, we denote by $T_i$ the tangent complex $T_{\sX_i}$ and $L_i$ for the cotangent complex, and $T_{ij}, L_{ij}$ for the tangent and cotangent complexes of $\sX_{ij}$. 

Consider a diagram
\[
\xymatrix{
&T_{02}\ar[r]^{((f'_1)_*,(f_1)_*)}\ar[d]_{((q_0f'_1)_*, (q_2f_2)_*))}  & T_{01}\oplus T_{12}\ar[d]_{((q_0)_*,(p_1)_*,(p'_1)_*,(q_2)_*)} \ar[r]^{((p_1)_*,(p'_1)_*)} & T_1\ar[d]^{\Delta_*}\\
&T_0\oplus T_2\ar[r]\ar[d]_{(\omega_0,-\omega_2)} & (T_0\oplus T_1)\oplus (T_1\oplus T_2)\ar[r]\ar[d]_{(\omega_0,-\omega_1,\omega_1,-\omega_2)}& T_1\oplus 
T_1\ar[d]^{(-\omega_1,\omega_1)}\\ 
&\Big(L_0\oplus L_2\Big)[1]\ar[d]_{((q_0f'_1)^*, (q_2f_2)^*))} & \Big(L_0\oplus L_1\Big)\oplus \Big(L_1\oplus L_2\Big)[1]\ar[l]\ar[d]_{((q_0)^*,(p_1)^*,(p'_1)^*,(q_2)^*)}  & L_1\oplus L_1[1]\ar[l]\ar[d]^{\Delta^*}\\
&L_{02}[1] & \Big(L_{01}\oplus L_{12}\Big)[1]\ar[l]_{((f'_1)^*,(f_1)^*)}  & L_1[1]  \ar[l]_{((p_1)^*,(p'_1)^*)}\\
}
\]
The first and fourth rows are the exact triangles of tangent and cotangent complexes of fiber product.
The second and third rows are the exact triangle of tangent and cotangent complexes for the projection map 
\[
\sX_0\times \sX_1\times \sX_1\times \sX_2\to \sX_1\times\sX_1.
\]
The two top squares and two bottom squares both commutes, and the left middle square also commutes. All rows are exact. The composition of all columns are zero due to the isotropic condition.

Let $L_{i,j}^{ij}$ be the relative cotangent complex that fits in the exact triangle
\[
L_{i,j}^{ij}\to L_{i}\oplus L_j\to L_{ij}\to L_{i,j}^{ij}[1].
\]  We have an exact triangle
\[\xymatrix{
L^{02}_{0,2}[1] & \Big(L_{0,1}^{01}\oplus L_{1,2}^{12}\Big)[1]\ar[l] & L_1[1]\ar[l]}
\] by taking the mapping cone of the vertical arrows from the third row to the fourth row.

The isotropic condition on $X_{01}, X_{12}$ implies there is a commutative diagram
\[\xymatrix{
T_{01}\oplus T_{12}\ar[rd]\ar[r]^{(\theta_{01},\theta_{12})} & \Big(L_{0,1}^{01}\oplus L_{1,2}^{12}\Big)[1]\ar[d]\\
& \Big(L_0\oplus L_1\Big)\oplus \Big(L_1\oplus L_2\Big)[1]
}
\]
The isotropic structure on $X_{02}$ is given by the composition
\[\xymatrix{
\theta_{02}: T_{02}\ar[r] & T_{01}\oplus T_{12}\ar[r]^{(\theta_{01},\theta_{12})} & \Big(L_{0,1}^{01}\oplus L_{1,2}^{12}\Big)[1]\ar[r] & L^{02}_{0,2}[1] }.
\] To show this is quasi isomorphism we use the fact that $\Delta^*$ splits.

Denote the quasi inverse of $\theta_{02}$ by $h_{02}$. Then we have the weight 2 part of the Poisson structure on $\sX_{02}$ is 
\[\xymatrix{
\pi_{02}: L_{02}\ar[r] & L_{0,2}^{02}[1]\ar[r]^{h_{02}} &T_{02}
}
\] Similar we have 
\[\xymatrix{
(\pi_{01},\pi_{12}): L_{01}\oplus L_{12}\ar[r] & \Big(L_{0,1}^{01}\oplus L_{1,2}^{12}\Big)[1]\ar[r]^{(h_{01},h_{12})} & T_{01}\oplus T_{12}
}
\]

Consider the diagram
\[
\xymatrix{
T_{02}\ar[r] &T_{01}\oplus T_{12}\\
L^{02}_{0,2}[1]\ar[u]^{h_{02}} & \Big(L_{0,1}^{01}\oplus L_{1,2}^{12}\Big)[1]\ar[l]\ar[u]_{(h_{01},h_{12})}\\
L_{02}\ar[u] & L_{01}\oplus L_{12}\ar[u]\ar[l]_{((f'_1)^*,(f_1)^*)}
}
\]
The top diagram commutes by the construction of $\theta_{02}$ and the bottom diagram commutes by functoriality. Then property (2) follows.
\end{proof}

Let $X$ be a Gorenstein CY variety. The Poisson bivector $\pi$ on $\R\eps\uperf(X)$ is induced by an explicit chain map (cf. \cite[Theorem 4.7]{HP1}, \cite[Section 2]{Pol98}, \cite[Section 9]{NS06})). Let us recall the construction of the chain map. Let $V^\bullet$ be a bounded chain complex of perfect complexes on $X$. We may view $V^\bullet$ as a double complex with finitely many columns consisting of perfect complexes. The column index is the weight and the horizontal differential is given by the mixed structure. 
Denote by $\ul\End^{\geq 0}(V^\bullet):=\bigoplus_{n\geq 0}\ul\End^n(V^\bullet)$ where
\[
\ul\End^n(V^\bullet)=\bigoplus_{i\in\Z} \ul\Hom(V^i,V^{i+n}).
\] Here $\ul\Hom$ is the internal Hom of perfect complexes and the differential $\partial$ on  $\ul\End^{\geq 0}(V^\bullet)$ is induced by the horizontal differential on $V^\bullet$, i.e. the mixed structure. Similarly we can define the complex $\ul\End^{\leq 0}(V^\bullet)$.
There is a natural isomorphism
\[
\ul\End^{\leq 0}(V^\bullet)\cong \Big(\ul\End^{\geq 0}(V^\bullet)\Big)^\vee.
\]
The Poisson bivector on $\R\eps\uperf(X)$ is induced by the chain map
\[
\xymatrix{
\ldots\ar[r] & \ul\End^{-1}(V^\bullet)\ar[r]^{\partial^\vee}\ar[d]^{\partial^\vee} & \ul\End^0(V^\bullet)\ar[d]^0\\
& \ul\End^0(V^\bullet)\ar[r]^\partial & \ul\End^1(V^\bullet)\ar[r] & \ldots
}
\]
The Poisson bivector $\pi$ from cotangent to tangent complex of $\R\eps\uperf(X)$ can be therefore computed by choosing a Cech resolution of $X$.

\subsection{Chains}\label{sec:chains}

For integers $a\leq b$, let $\eps-M^{[a,b]}$ be the subcategory of $\eps-M^\gr$ consisting of mixed graded objects $\Big((E(i))_{i\in [a,b]},\eps\Big)$, i.e. $E(i)=0$ unless $i\in [a,b]$.  A (finite) \emph{chain} of objects  in $M$ is
\[
E(a)\to E(a+1)\to \ldots \to\ldots\to E(b)
\]
with   $E(i)\to E(i+1)$ being a morphism in $M$. Note that we do not require compositions of consecutive arrows to vanish.

Now we take $M=\QC(X)$ for a projective $k$-scheme $X$. Denote by $\chain^{[a,b]}(X)$ the category of chain of perfect complexes on $X$
\[
V_a\to V_{a+1}\to \ldots \to\ldots\to V_b.
\] Let $\R\uchain^{[a,b]}(X)$ be the associated stack. There is a quasi-equivalence of stacks
\begin{align*}
&\R\uchain^{[a,b]}(X)\simeq \\
&\R\eps\uperf^{[a,a+1]}(X)\times^h_{\R\uperf(X)} \R\eps\uperf^{[a+1,a+2]}(X)\times^h\ldots \times^h_{\R\uperf(X)}  \R\eps\uperf^{[b-1,b]}(X)
\end{align*} with respect to the morphism $p^{S}: \R\eps\uperf^{[a,b]}(X)\to \R\uperf(X)$ with $S=\{i\}$ for $i\in [a,b]$. We have a stack map
\begin{equation}\label{Ch-beta-map-eq}
\beta: \R\uchain^{[a,b]}(X)\to \R\uchain^{[a,a+1]}(X)
\end{equation}
defined by sending our chain to the chain $V_a\to V_b$, in degrees $[a,a+1]$, with the morphism being the composition. 

\begin{theorem}\label{thm:Poissonchain}
Assume that $X$ is a Gorenstein Calabi-Yau curve. Then
$\R\uchain^{[a,b]}(X)$ admits a $0$-shifted Poisson structure such that the canonical projections to $\R\eps\uperf^{[i,i+1]}(X)$ 
and the composition map $\beta$ are semi-classically Poisson.
\end{theorem}
\begin{proof}
We define the $0$-shifted Poisson structure on $\R\uchain^{[a,b]}(X)$ by equipping the map
$$\R\uchain^{[a,b]}(X)\to \prod_{i=a}^{b-1} \Big(\R\eps\uperf^{[i,i]} \times \R\uperf(X)\Big) \times \R\eps\uperf^{[b,b]}$$
with a Lagrangian structure. We do this by induction in the length of $[a,b]$. 
Consider the diagram
\[
\xymatrix{
& \R\uchain^{[a,b]}(X)\ar[rd]^{f_1}\ar[ld]_{f'_1} \\
\R\eps\uperf^{[a,a+1]}(X)\ar[rd]^{p_1}\ar[d]^{q_0} & & \R\uchain^{[a+1,b]}(X)\ar[d]^{q_2}\ar[ld]_{p'_1}\\ 
\R\eps\uperf^{[a,a]}(X)\times\R\uperf(X) & \R\eps\uperf^{[a+1,a+1]}(X) & \prod_{i=a+2}^b\Big(\R\eps\uperf^{[i,i]}(X)\times\R\uperf(X) \Big)
}
\]
By the induction assumption, the two lower roof diagrams are Lagrangian correspondences. Hence, their composition is also Lagrangian.
On the other hand, by Proposition \ref{prop:composition of Lag}, $f'_1$ and $f_1$ are semi classically Poisson morphisms.

Next, we will show that the composition map \eqref{Ch-beta-map-eq} is semi-classically Poisson, for $a=0,b=2$ first.  Let $v$ be a chain of perfect $\OO_X$-modules
\[\xymatrix{
V_0\ar[r]^{d} & V_1\ar[r]^{d} & V_2
}
\] The tangent complex of $\R\uchain^{[0,2]}(X)$ at $v$ is
\[
T_v\R\uchain^{[0,2]}(X) \simeq \bC\Big(\bigoplus_{i=0}^2 \ul\End(V_i)\to \ul\Hom(V_0,V_1)\oplus \ul\Hom(V_1,V_2)\Big)[1]
\] with differential induced by the chain map of sheaves
\[
(a_{00} ,a_{11} ,a_{22} )\mapsto(da_{00}-a_{11}d,da_{11}-a_{22}d) \hspace{2cm} \text{for $a_{ii}\in\ul{\End}(V_i)$.}
\] Here $\ul\Hom$ is the internal hom in $\perf(X)$ and $\bC$ stands for the Cech cochain complex that computes the hypercohomology of a complex of coherent sheaves.
Denote by $\mu(v)$ the composition $d^2: V_0\to V_2$. It has tangent complex
\[
T_{\mu(v)} \R\eps\uperf^{[0,1]}(X)\simeq \bC\Big( \ul\End(V_0)\oplus\ul\End(V_2) \to \ul\Hom(V_0,V_2)\Big)[1]
\] with differential induced by the chain map
\[
(a_0,a_2)\mapsto d^2 a_{00}-a_{22}d^2.
\]
The tangent map $T_v\R\uchain^{[0,2]}(X) \to \mu^*T_{\mu(v)} \R\eps\uperf^{[0,1]}(X)$ is induced by the chain map
\[
(a_{00},a_{11},a_{22})\mapsto (a_{00},a_{22}),\hspace{2cm} (a_{01},a_{12})\mapsto da_{01}+a_{12}d
\]
where $a_{i,i+1} \in \ul\Hom(V_i,V_{i+1})$.  
It is straightforward to check this is a chain map.

By the $1$-Calabi-Yau property, 
\[
L_v\R\uchain^{[0,2]}(X) \simeq \bC\Big(\ul\Hom(V_2,V_1)\oplus \ul\Hom(V_1,V_0)\to \bigoplus_{i=0}^2 \ul\End(V_i)\Big)
\] with differential induced by the chain map
\[
(a_{10},a_{21})\mapsto (-a_{10}d,da_{10}-a_{21}d, da_{21}),
\]
and 
\[
L_{\mu(v)} \R\eps\uperf^{[0,1]}(X)\simeq \bC\Big(  \ul\Hom(V_2,V_0)\to\ul\End(V_0)\oplus\ul\End(V_2)\Big)
\] with differential induced by the chain map
\[
a_{20}\mapsto (a_{20}d^2,d^2 a_{20}).
\]
The cotangent map $\mu^*L_{\mu(v)} \R\eps\uperf^{[0,1]}(X)\to L_v\R\uchain^{[0,2]}(X) $ is induced by the chain map
\[
(a_{00},a_{22})\mapsto (a_{00},0,a_{22}),\hspace{2cm} a_{20}\mapsto (da_{20},a_{20}d).
\]
For simplicity, we write $E_{ij}:=\ul\Hom(V_i,V_j)$. 

\begin{lemma}\label{biv-Ch}
The Poisson bivector $\pi$ on $\R\uchain^{[0,2]}(X)$ is induced by the chain map of complexes of perfect complexes
$$\ol{\phi}: [E_{10}\oplus E_{21}\to E_{00}\oplus E_{11}\oplus E_{22}]\to [E_{00}\oplus E_{11}\oplus E_{22}\to E_{01}\oplus E_{12}],$$
\[
\ol{\phi}: (a_{10},a_{21})\mapsto (a_{10}d,da_{10},da_{21}), \hspace{1cm} (b_{00},b_{11},b_{22})\mapsto (0,db_{11}).
\]
\end{lemma}
\begin{proof}
The product Poisson structure on $\R\uchain^{[0,1]}(X)\times \R\uchain^{[1,2]}(X)$ is induced by the chain map 
\[
\xymatrix{
E_{10}+ E_{21}\ar[rr]^{\partial^\vee:\left(\begin{array}{ccc}a_{10} &a_{21}\end{array}\right)  \mapsto \left(\begin{array}{cc} a_{10}d & da_{10} \\ a_{21}d & da_{21}\end{array}\right)}  \ar[dd]_{\left(\begin{array}{ccc}a_{10}\\a_{21}\end{array}\right)  \mapsto \left(\begin{array}{cc} a_{10}d & da_{10} \\ a_{21}d & da_{21}\end{array}\right)} && (E_{00}+E_{11})+(E_{11}+E_{22})\ar[dd]^0 \\ 
& & & \\
(E_{00}+E_{11})+(E_{11}+E_{22})\ar[rr]_{\partial: \left(\begin{array}{cc} b_{00} & b_{11}\\ b_{11}^\prime & b_{22} \end{array}\right)\mapsto \left(\begin{array}{cc}db_{00}-b_{11}d\\ db^\prime_{11}-b_{22}d\end{array}\right) }&& E_{01}+E_{12}
}
\]
Here we write elements of $(E_{00}+E_{11})+(E_{11}+E_{22})$ as a matrix where the first row are elements of $(E_{00}+E_{11})$ and the second row are elements of $(E_{11}+E_{22})$.
We denote by $\phi$ the left vertical map. To compute the Poisson bivector on $\R\uchain^{[0,2]}(X)$, we consider the diagram
\begin{equation}\label{CD:poissonch}
\xymatrix{
E_{10}+ E_{21}\ar[rr] ^{\tau\partial^\vee} && E_{00}+E_{11}+E_{22} \ar@/^2.0pc/[ddd]^{\ol{\phi}} \\ 
E_{10}+ E_{21}\ar[rr]^{\partial^\vee} \ar[u]^=  \ar[dd]_{\phi}\ar@/_2.0pc/[ddd]_{\ol{\phi}} && (E_{00}+E_{11})+(E_{11}+E_{22})\ar[dd]^0\ar[u]_{\tau}\ar[lldd]^h \\ 
& & & \\
(E_{00}+E_{11})+(E_{11}+E_{22})\ar[rr]_{\partial}&& E_{01}+E_{12}\\
E_{00}+E_{11}+E_{22}\ar[rr]_{\partial\Delta}\ar[u]^\Delta& & E_{01}+E_{12}\ar[u]_{=}
}
\end{equation}
where 
\[\Delta(b_{00},b_{11},b_{22})=\left(\begin{array}{cc} b_{00} & b_{11}\\ b_{11} & b_{22}\end{array}\right),
\hspace{1cm}
\tau\left(\begin{array}{cc} b_{00} & b_{11}\\ b^\prime_{11} & b_{22}\end{array}\right)=\left(\begin{array}{ccc} b_{00}\\ b_{11}-b^\prime_{11} \\ b_{22}\end{array}\right).
\] 
We look for a homotopy $h$ of the middle square such that 
\[
\phi+h\partial^\vee=\Delta\ol{\phi},\hspace{2cm} \partial h=\ol{\phi}\tau.
\]
The map $h: (E_{00}+E_{11})+
(E_{11}+E_{22})\to E_{11}^{\oplus 2}$ defined by 
\[
h(\left(\begin{array}{cc} b_{00} & b_{11}\\ b^\prime_{11} & b_{22}\end{array}\right))=\left(\begin{array}{cc} 0 & 0\\ b_{11}-b'_{11} & 0\end{array}\right)
\] satisfies the desired property.  Then we obtain the chain map in the claim.
\end{proof}

By an abuse of notation we denote by $\mu$ the cotangent map induced by the composition map $\mu$.
It is straight forward to check for the diagram
\begin{equation}\label{diag:1}
\xymatrix{
E_{20}\ar[d]_\mu\ar[r] & E_{00}+E_{22}\ar[d]\\
E_{21}+E_{10}\ar[r]\ar[d]_{\ol{\phi}}& E_{00}+E_{11}+E_{22}\ar[d]^{\ol{\phi}}\\
E_{00}+E_{11}+E_{22}\ar[r]\ar[d] & E_{01}+E_{12}\ar[d]^{\mu^\vee}\\
E_{00}+E_{22}\ar[r] & E_{02}
}
\end{equation}
the composition of the left column is $a_{20}\mapsto (a_{20}d^2,d^2a_{20})$ and the composition of the right column is zero. Then it follows that the diagram
\[\xymatrix{
\beta^*L_{\mu(v)}\R\eps\uperf^{[0,1]}(X)\ar[r]^{\beta^*\pi} \ar[d]& \beta^*T_{\mu(v)}\R\eps\uperf^{[0,1]}(X) \\
L_v\R\uchain^{[0,2]}(X)\ar[r]^{\pi} &T_v\R\uchain^{[0,2]}(X)\ar[u]
}
\] commutes.

For general $n$ the Poisson bivector on $\R\uchain^{[0,n]}(X)$ is induced by the chain map:
\[\xymatrix{
\bigoplus_{i=0}^{n-1} E_{i+1,i}\ar[r]^{\partial^\vee}\ar[d] & \bigoplus_{i=0}^n E_{ii}\ar[d]\\
\bigoplus_{i=0}^n E_{ii}\ar[r]^{\partial} & \bigoplus_{i=0}^{n-1} E_{i,i+1}
}
\]
The left vertical map is
\[
(a_{10},a_{21},\ldots,a_{n,n-1})\mapsto (a_{10}d, da_{10},da_{12},\ldots, da_{n,n-1})
\] and
the right vertical map is
\[
(b_{00},\ldots,b_{nn})\mapsto (0, db_{11},\ldots, db_{nn}).
\] One can verify that the diagram similar to (\ref{diag:1}) commutes.
\end{proof}

It is easy to check that the chain map
\[
(a_{10},a_{21})\mapsto (a_{10}d,a_{21}d,da_{21}), \hspace{1cm} (b_{00},b_{11},b_{22})\mapsto (b
_{11}d, 0).
\] is homotopy equivalent to the chain map given in Lemma \ref{biv-Ch}, which gives another chain representative of the Poisson bivector on $\R\uchain(X)$.

An intermediate step of the above proof turns out to be useful later. We write it as a corollary.
\begin{cor}\label{cor:fiberprod}
The natural map $\R\uchain^{[a,b]}(X)\to \prod_{i=a}^{b-1} \R\uchain^{[i,i+1]}(X)$ is semi-classically Poisson, where the target is equipped with the product Poisson structure. 
\end{cor}
\begin{proof}
In diagram (\ref{CD:poissonch}), the middle square gives the chain map for the Poisson bivector on the target, while the outer square gives the chain map for the Poisson bivector on the source.  
\end{proof}


Let $v$ be a perfect complex on $X$. We denote by $[v]$ the stacky point (i.e. residue gerbe) on $\R\uperf(X)$ that corresponds to $v$. The canonical map 
\[
j_v: [v]\to \R\uperf(X)
\] is Lagrangian (\cite[Theorem 3.18]{HP1}). It follows from Theorem \ref{thm:basechange} that the base change
\[
[v]\times^h_{j_v, q} \R\eps\uperf(X)\to \R\eps\uperf(X)
\] is Poisson.

Let $a\leq b$ be two integers.
Denote by
\[
t^{[a,b]}: \R\eps\uperf(X)\to \R\eps\uperf^{[a,b]}(X)
\] be the stack map by taking the $[a,b]$-truncation of a mixed complexes. 

\begin{lemma}
$t^{[a,b]}$ is a semi-classical Poisson morphism.
\end{lemma}
\begin{proof}
We adopt the notation used in the proof of Theorem \ref{thm:Poissonchain}.
Let $v=\{\ldots\to V_i\to^d V_{i+1}\to \ldots\}$ be a perfect mixed graded complex. Recall that the Poisson bivector is induced by the chain map 
\[
\partial^\vee:=d\circ ?+ ?\circ d: \bigoplus_{i} E_{i+1,i}\to  \bigoplus_{i} E_{ii}, \hspace{1cm} 0: \bigoplus_{i} E_{ii}\to \bigoplus_{i} E_{i,i+1}.
\]
If $i\in [a,b-1]$ then the image of $\partial^\vee$ lies in the components with $i\in [a,b]$, i.e. $\partial^\vee$ is compatible with truncation. Therefore $t^{[a,b]}$ is Poisson.
\end{proof}

Let $f: U\to V$ be a morphism in $\perf(X)$. We may view it as a point in $\R\eps\uperf^{[0,1]}(X)$. Let 
\[
U\to V\to C(f)\to U[1]
\]  be the exact triangle with $C(f)$ being the mapping cone. Then $V\to C(f)$ is a point in $\R\eps\uperf^{[1,2]}(X)$. Denote by $\R\eps\uperf^{[0,2],e}(X)$ the stack of exact mixed graded complexes $V_0\to V_1\to V_2$.
\begin{lemma}\label{lem:exacttriangle}
The truncation maps
\[
t^{[0,1]}: \R\eps\uperf^{[0,2],e}(X)\to \R\eps\uperf^{[0,1]}(X), \hspace{1cm}  t^{[1,2]}: \R\eps\uperf^{[0,2],e}(X)\to \R\eps\uperf^{[1,2]}(X)
\] are semi-classically Poisson and are  quasi-equivalences of stacks.
\end{lemma}
\begin{proof}
Note that $\R\eps\uperf^{[0,2],e}(X)$ can be identified with the homotopy fiber product $[0]\times^h_{j_0, q} \R\eps\uperf^{[0,2]}(X)$ where $j_0: [0]\subset \R\uperf(X)$ is the substack consisting of acyclic prefect complexes. Since the cotangent and tangent complex at an acyclic perfect complex is acyclic, $j_0$ is Lagrangian. By Theorem \ref{thm:fiberprod} the natural morphism
\[
\R\eps\uperf^{[0,2],e}(X)\to \R\eps\uperf^{[0,2]}(X)
\] is Poisson and the composition with the truncation map is also Poisson. By exactness, they are quasi-equivalences of stacks. \footnote{We do not use $X$ is 1CY. It holds for any $d$.}
\end{proof}

\begin{example}
The above lemma gives an alternative moduli interpretation of Feigin-Odesskii Poisson moduli space which will be used later. Recall that Feigin-Odesskii's Poisson moduli space is the moduli stack of non-splitting complexes $\cO_X\to V$ such that $V$ is a vector bundle and $V/\cO_X$ is isomorphic to a fixed stable vector bundle $\xi$ of rank $k$ and degree $n$. By the lemma, in the homotopy category of stacks it is equivalent with the moduli stack of mixed complexes $\xi_{n,k}\to \cO_X[1]$ and the equivalence is compatible with Poisson structure, at least semi-classically.
\end{example}
 
 \section{Bosonization of Feigin-Odesskii Poisson varieties}\label{sec:boson}
\subsection{Continued fraction and stable sheaves on elliptic curves}\label{sec:contfrac}
Fix coprime integers $n>k>0$. By the Euclidean algorithm, there exists unique positive integers $n_1,\ldots,n_p\geq 2$ such that $n/k$ is represented as negative continued fraction
\[
\frac{n}{k}=n_1-\frac{1}{n_2-\frac{1}{\ldots -\frac{1}{n_p}}}.
\]
We denote by $D(n_1,\ldots,n_p)$ the determinant of the tridiagonal matrix
\[\left(
\begin{array}{ccccc}
n_1 & -1 & 0 & \ldots &\\
-1 & n_2 & -1 & 0 &\ldots \\
& & \ldots & &\\
 \ldots &0& -1 & n_{p-1} & -1\\
&\ldots & 0 & -1 & n_p
\end{array}\right)
\]
The successive convergents of the above continued fraction are 
$$\mu_{p-1}=\frac{n(p-1)}{k(p-1)}=n_1, \ \mu_{p-2}=\frac{n(p-2)}{k(p-2)}=n_1-\frac{1}{n_2}, \ldots, \ \mu_0=\frac{n(0)}{k(0)}=\frac{n}{k},$$
where 
\[
n(i):=D(n_1,\ldots,n_{p-i}),\hspace{1cm} k(i)=D(n_2,\ldots,n_{p-i}).
\]
In particular,
\[
n=D(n_1,\ldots,n_p),~~~k=D(n_2,\ldots,n_p)
\] 
(when $p=1$ we have $n=n_1$ and $k=1$).
As a convention, we set $n(p)=1, k(p)=0$, $\mu(p)=\infty$. 
Note that 
\[
n(0)=n>n(1)>\ldots>n(p-1)>n(p)=1,\hspace{1cm} k(0)=k>k(1)>\ldots>k(p-1)>k(p)=0,
\] and 
\begin{equation}\label{ni-ki-det1-eq}
n(i+1)k(i)-n(i)k(i+1)=1,
\end{equation}
in particular, $n(i)$ and $k(i)$ are coprime.

Now let $X$ be a complex elliptic curve. Recall for perfect complexes $V, W$ on $X$, one has
\[
\chi(V,W)=\sum (-1)^i\dim \Ext^i(V,W)=\deg(W)\rk(V)-\deg(V)\rk(W).
\]

\begin{lemma}\label{Si-stable-lem}
Let $\xi_i$ and $\xi_{i+1}$ be stable sheaves on $X$ of slopes $\mu_i$ and $\mu_{i+1}$ (where $0\le i\le p-1$). Then the space $\Hom(\xi_i,\xi_{i+1})$ is $1$-dimensional.
Furthermore, the unique (up to rescaling) nonzero morphism $\xi_i\to \xi_{i+1}$
fits into an exact sequence of sheaves
$$0\to S_i\to \xi_i\to \xi_{i+1}\to 0$$
where $S_i$ is a stable bundle of degree $d(i)>0$ and rank $r(i)>0$ that are given by
\[
r(i)=k(i)-k(i+1)=D(n_2,\ldots,n_{p-i}-1) \hspace{2cm}
d(i)=n(i)-n(i+1)=D(n_1,\ldots,n_{p-i}-1),  \hspace{2cm}
\] 
(for $p=1$, we have $r(0)=1$, $d(0)=n-1$).
Let $s(i)=\frac{d(i)}{r(i)}$ denote the slope of $S_i$, and let $s(p)=1$. Then 
\[
\frac{n}{k}>s(0)\geq s(1)\geq \ldots\geq s(p)>0,
\] 
and the equality $s(i-1)=s(i)$ holds if and only if $n_i=2$.
\end{lemma}

\begin{proof}
The relation \eqref{ni-ki-det1-eq} implies that $\chi(\xi_i,\xi_{i+1})=1$. Furthermore, since $\mu(\xi_i)<\mu(\xi_{i+1})$, we have $\Ext^1(\xi_i,\xi_{i+1})=0$, so
$\dim\Hom(\xi_i,\xi_{i+1})=1$.
By \cite[Cor.\ 14.11]{Pol-AV}, a nonzero map $\xi_i\to \xi_{i+1}$ is subjective. Furthermore, its kernel $S_i$ can be identified the spherical twist of $\xi_{i+1}$
with respect to $\xi_i$, hence, $S_i$ is simple. As $\rk(S_i)=k_i-k_{i+1}>0$, we deduce that $S_i$ is a stable bundle.

The determinantal formulas for $r(i)$ and $d(i)$ follow from the Laplace expansions 
\begin{align}\label{Laplace}
n(i)=n_{p-i}\cdot n(i+1)-n(i+2), \hspace{1cm} k(i)=n_{p-i}\cdot k(i+1)-k(i+2), 
\end{align}
\end{proof}

Fix a stable sheaf $\xi$ of rank $k$ and degree $n$. 
A \emph{stable descending chain for $\xi$} is a chain of nonzero morphisms 
\[
\xi=\xi_0\to \xi^{s}_1\to\ldots\to \xi^{s}_{r}\to \cO[1]
\] where $\xi_i^s$ are stable sheaves for $i=1,\ldots,r$ such that 
\[
\rk(\xi)>\rk(\xi^s_1)>\ldots>\rk(\xi^s_r).
\]
The rightmost map is required to be a nonzero morphism in derived category.  To keep the notation consistent, we set $\xi^s_{r+1}:=\OO[1]$.

We call a stable descending chain for $\xi$ 
\emph{complete} if  $r=p$ and
$\deg(\xi^{s}_i)=n(i)$ and $\rk(\xi^{s}_i)=k(i)$. 
In this case we have a unique (up to rescaling) morphism $\xi^{s}_i\to \xi^{s}_{i+1}$ for $i=0,\ldots,p$ (see Lemma \ref{Si-stable-lem}).
 
 A \emph{semi-stable descending chain for $\xi$} is a chain of nonzero maps between semi-stable sheaves 
\[
\xi=\xi_0\to \xi^{ss}_1\to\ldots\to \xi^{ss}_r\to \cO[1]
\] such that 
\[
\rk(\xi)>\frac{\rk(\xi^{ss}_1)}{m_1}>\ldots>\frac{\rk(\xi^{ss}_r)}{m_r}.
\] 
where $m_i:=\gcd(\rk(\xi^{ss}_i),\deg(\xi^{ss}_i))$. We call $\am=(m_1,\ldots,m_r)$ the \emph{multiplicity}  of the chain. 
A semi-stable descending chain is stable if and only if it has multiplicity $(1,\ldots,1)$. 

A semi-stable descending chain for $\xi$ is called \emph{complete}  if $r=p$, 
$\deg(\xi^{ss}_i)=m_i\cdot n(i)$ and $\rk(\xi^{ss}_i)=m_i\cdot k(i)$ for positive integers $m_1,\ldots, m_p$. 

Recall that the slope $\mu(F)$ of a semistable sheaf $F$ is defined by $\mu(F):=\frac{\deg(F)}{\rk(F)}$ (which is equal to $\infty$ when $F$ is a torsion sheaf).
For a descending chain as above we have the slope sequence $\amu:=(\mu(\xi^{ss}_1),\ldots,\mu(\xi^{ss}_r))$.

\subsection{Construction of bosonization}
We denote by $\sB(\xi,\amu, \am)$ the moduli stack of semi-stable descending chains of slope $\amu$ and multiplicity $\am$ for $\xi$. It is a substack of $\R\uchain(X)$. We set 
\[
\amu_{\max}=(n(1)/k(1),\ldots,n(p)/k(p)), \ \ \sB(\xi,\am):=\sB(\xi,\amu_{\max},\am),
\] 
so $\sB(\xi,\am)$ parameterizes {\it complete} semi-stable descending chains of multiplicity $\am$.

\begin{prop}
Fix a complex elliptic curve $X$. 
Given coprime integers $0<k<n$, let $\xi$ be a stable bundle of rank $k$ and degree $n$ on $X$. Then $\sB(\xi, \amu,\am)$ is a Poisson substack of $\R\uchain^{[0,r+1]}(X)$.
\end{prop}
\begin{proof}
Recall that $\R\uchain^{[0,r+1]}(X)$ is defined to be the homotopy fiber product 
\[
\R\eps\uperf^{[0,1]}(X)\times_{\R\uperf(X)}\ldots \times_{\R\uperf(X)}\R\eps\uperf^{[r,r+1]}(X)
\] which admits a Lagrangian morphism to $(r+1)$-fold product of $\R\uperf(X)$ by assigning a chain
\[
V_0\to \ldots \to V_r
\] the tuple 
\[
\Big( V_0, V_1/V_0,\ldots, V_r/V_{r-1}, V_r\Big).
\] Here by an abuse of notion we write $V_i/V_{i-1}$ for the underlying perfect complex of $V_{i-1}\to V_{i}$.
We obtain a Poisson open substack such that 
$V_i$ are (shifts of) sheaves with the prescribed rank and degree and the maps being nonzero. Finally to obtain $\sB(\xi, \amu,\am)$ we take the fiber product with respect to the above Lagrangian morphism and the canonical map
\[
([\xi], [\cO[1]])\to \R\uperf(X)\times \R\uperf(X)
\] into the first and last components of the $(r+1)$-fold product associate to the residue gerbe. By  part (b) of Theorem \ref{thm:basechange}, the base change
\[
\sB(\xi, \amu,\am)\to \R\uchain^{[0,r+1]}(X)
\] is Poisson.
\end{proof}
\begin{definition}
Let $X,n,k, \xi$ be defined as above. Denote by $\sN(\xi)$ the substack of $\R\eps\uperf(X)$ consisting of $\cO\to V$ such that $V/\cO\cong \xi$ and the extension morphism $\xi\to \cO[1]$ is nonzero. We will denote by $N(\xi)$ its coarse moduli space. Together with its Poisson structure, $N(\xi)$ is called the \emph{Feigin-Odesskii Poisson variety associated to the stable bundle $\xi$}.
\end{definition}

By Lemma \ref{lem:exacttriangle}, $\sN(\xi)$ is Poisson isomorphic to the stack of nonzero morphisms $\xi\to \cO[1]$.
From now on, we will not distinguish these two stacks.

Given 
\[
\amu=(\mu_1,\ldots,\mu_r),\hspace{2cm} \amu^\prime=(\mu^\prime_1,\ldots,\mu^\prime_s)
\]  and
\[
\am=(m_1,\ldots,m_r),\hspace{2cm} \am^\prime=(m^\prime_1,\ldots,m^\prime_s).
\]
we write $(\amu,\am)\succeq(\amu^\prime,\am^\prime)$ if 
there exists an order preserving injection $\iota: \{1,\ldots,s\}\to \{1,\ldots,r\}$ such that $\mu^\prime_i=\mu_{\iota(i)}$ and $m^\prime_i=m_{\iota(i)}$ for all $i=1,\ldots,s$. In particular this implies that $s\leq r$.

\begin{prop}\label{betapoisson}
Let $X,n,k,\xi,\amu,\am$ be defined as above. Suppose that $(\amu,\am)\succeq (\amu^\prime,\am^\prime)$. The iterated composition defines a rational morphism 
\[\xymatrix{
\beta(\xi, \amu,\amu^\prime, \am, \am^\prime): \sB(\xi, \amu,\am)\ar@{-->}[r]  &  \sB(\xi, \amu^\prime,\am^\prime),}
\]
regular on the Zariski open substack where the compositions are nonzero.
On the defining domain, $\beta(\xi, \amu,\amu^\prime, \am, \am^\prime)$ is a Poisson morphism.
\end{prop}
\begin{proof}
The vanishing condition of composition of morphisms is Zariski closed. Therefore the defining domain of  $\beta(\xi, \amu,\amu^\prime, \am, \am^\prime)$ is Zariski open in $ \sB(\xi, \amu,\am)$.
Since the composition of Poisson maps is a Poisson map, the last claim follows from Theorem \ref{thm:Poissonchain}.
\end{proof}

Thus, we get a hierarchy of Poisson moduli spaces with rational Poisson morphisms 
\[\xymatrix{
\ldots\ar@{-->}[r] & \sB(\xi, \amu,\am) \ar@{-->}[r] &  \sB(\xi, \amu^\prime,\am^\prime)\ar@{-->}[r] &\ldots.}
\]
When the data $\xi, \amu,\amu^\prime, \am, \am^\prime$ are fixed, we will write $\beta$ for $\beta(\xi, \amu,\amu^\prime, \am, \am^\prime)$. 

Applying Proposition \ref{betapoisson} to $\amu=\amu_{\max}$ and $\amu^\prime=\emptyset$ (and $\am^\prime=\emptyset$), so that
$\sB(\xi, \amu^\prime,\am^\prime)=\sN(\xi)$,
we get a morphism (after restricting to a Zariski open substack)
\[\xymatrix{
\beta:\sB(\xi,\am)\ar@{-->}[r] & \sN(\xi)}
\]
which we call the \emph{bosonization morphism} of $\sN(\xi)$ of multiplicity $\am$.

Let us denote by $\sM^{ss}(\amu,\am)$ the product over $i=1,\ldots,r$ of the moduli stacks of 
of semi-stable sheaves of rank $r_i$ and degree $d_i$ with $\mu_i=d_i/r_i$ and $\gcd(d_i,r_i)=m_i$. We have a forgetful map
\[
\pi(\xi, \amu,\am):  \sB(\xi, \amu,\am)\to\sM^{ss}(\amu,\am)
\] sending
$\xi\to \xi^{ss}_1\to \ldots \to \xi^{ss}_r\to\cO[1]$
to the tuple 
$(\xi^{ss}_1,\ldots, \xi^{ss}_{r})$.  

\begin{example}
When $k=1$,  $\xi$ is a line bundle of degree $n$, and $\am=m$. The moduli stack $\sB(\xi,m)$ parameterizes chains
\[
\xi\to \cO_D(D)\to \cO[1]
\] where $D$ is an effective divisor of degree $n$. It contains an open substack where $\cO_D(D)\to \cO[1]$ is the canonical map that corresponds to the extension $\cO(D)$.  We will compute the Poisson bracket on this substack explicitly in Section \ref{sec:p=1}.
\end{example}
 

\section{Geometry of $\sB(\xi)$}
\label{sec:m=1}

In this section we specialize to $\amu=\amu_{\max}$ and $\am=(1,\ldots,1)$, and we set
$$\sB(\xi):=\sB(\xi,(1,\ldots,1))=\sB(\xi,\amu_{\max},(1,\ldots,1)).$$

\subsection{Poisson structure}
\begin{lemma}\label{lem:Gm}
The stack
$\sB(\xi)$
is a $\G_m$-gerbe over $X^p$.
\end{lemma}
\begin{proof}
Recall that an object in $\sB(\xi)$ is a complete stable descending chain of nonzero morphisms
\begin{align}\label{a chain}
\xi=\xi_0\to \xi_1\to\ldots\to \xi_{p}\to \xi_{p+1}=\cO[1]
\end{align}  
where  $\dim\Hom(\xi_i,\xi_{i+1})=1$ for $i=0, \ldots, p$ (see Lemma \ref{Si-stable-lem}).
The coarse moduli space of $\xi_i$ is isomorphic to $X$ for $i=1,\ldots, p$. The forgetful map $\pi(\xi,\amu_{\max},\aone)$ composed with the coarse moduli functor is a map from 
$\sB(\xi)$ to $X^p$. Its fiber over
a point $([\xi_1],\ldots,[\xi_p])$ is equivalent to the substack of  the moduli stack of representations of $A_{p+2}$-quiver with the dimension vector $(1,\ldots,1)$ given by the condition that all morphisms are nonzero.
It is isomorphic to the global quotient stack $\Big[\G_m^{p+1}/\G_m^{p+2}\Big]$ with the action
\[
(t_0,\ldots,t_{p+1})\cdot (a_{01},\ldots,a_{p,p+1})=(t_0t_1^{-1}a_{01},\ldots,t_pt_{p+1}^{-1}a_{p,p+1}),
\]
which is equivalent to $B\G_m$.
\end{proof}

We say that a $d$-shifted Poisson structure on a stack $\sX$ \emph{vanishes semi-classically} if the Poisson bivector $L_\sX\to T_\sX[-d]$ induces zero maps on cohomology groups.  

\begin{theorem}\label{thm:zero}
The Poisson structure on $\sB(\xi)$ vanishes semi-classically.
\end{theorem}
\begin{proof}
First we claim that for $i=1,\ldots, p$ the restriction of the 0-shifted Poisson structure on the component of $\R\uchain^{[i,i+1]}$ consisting of the surjection $\phi_i:\xi_i^s\to \xi_{i+1}^s$ vanishes semi-classically. For simplicity we denote this complex by $\phi_i$. Recall that the tangent complex at $\phi_i$ is the Cech cochain of the complex 
\[
\partial_{\phi_i}: \ul\End(\xi_i^s)\oplus \ul\End(\xi_{i+1}^s)\to \ul\Hom(\xi^s_i,\xi^s_{i+1}).
\]
The Poisson bivector is induced by the chain map
\[\xymatrix{
 \ul\Hom(\xi^s_{i+1},\xi^s_{i})\ar[r]^{\partial^\vee_{\phi_i}}\ar[d]_{\partial^\vee_{\phi_i}} & \ul\End(\xi_i^s)\oplus \ul\End(\xi_{i+1}^s)\ar[d]^0\\
\ul\End(\xi_i^s)\oplus \ul\End(\xi_{i+1}^s)\ar[r]^{\partial_{\phi_i}} & \ul\Hom(\xi^s_i,\xi^s_{i+1}).
}
\] For simplicity
 we denote 
\[
E_{i+1,i}:=\ul\Hom(\xi^s_{i+1},\xi^s_{i}), ~~E_{i,i}:=\ul\End(\xi_{i+1}^s),~~E_{i,i+1}:=\ul\Hom(\xi^s_{i},\xi^s_{i+1})
\] and set
\[
Z_i=\{\partial^\vee:=\partial^\vee_{\phi_i}: E_{i+1,i}\to E_{ii}\oplus E_{i+1,i+1}\},~~~Z^\vee_i=\{\partial:=\partial_{\phi_i}: E_{ii}\oplus E_{i+1,i+1}\to E_{i,i+1}\}.
\]
Fix an open covering $U_+,U_-$ of $X$, the (Cech) cohomology of $Z_i$ and $Z_i^\vee$ can be computed by cohomology of the following complexes.
\[
\xymatrix{
\Gamma(U_+\sqcup U_- ,E_{i+1,i})\ar[rr]^{(d,\partial^\vee)}  &&\Gamma(U_+\cap U_- ,E_{i+1,i})\oplus \Gamma(U_+\sqcup U_-,E_{ii}\oplus E_{i+1,i+1})\\
&\ar[r]^{(\partial^\vee,d)} & \Gamma(U_+\cap U_-,E_{ii}\oplus E_{i+1,i+1})
}
\] and
\[
\xymatrix{
\Gamma(U_+\sqcup U_- ,E_{ii}\oplus E_{i+1,i+1})\ar[rr]^{(d,\partial)}  && \Gamma(U_+\cap U_-,E_{ii}\oplus E_{i+1,i+1})\oplus\Gamma(U_+\sqcup U_- ,E_{i,i+1}) \\
&\ar[r]^{(\partial,d)} & \Gamma(U_+\cap U_-,E_{i,i+1})
}
\] where $d$ is the Cech differential.
Since $\mu(\xi^s_{i+1})>\mu(\xi^s_i)$, 
\[
\HH^{-1}(Z_i)=\HH^1(Z^\vee_i)=0.
\]
The morphism $\HH^0(Z_i)\to \HH^0(Z^\vee_i)$ is zero since it is induced by the chain map $\left[\begin{array}{cc}\partial^\vee & d\\ 0 & 0 \end{array}\right]$.

By a similar argument the stack of morphisms $\xi\to \xi^s_1$ and the 0-shifted Poisson structure on the stack of morphisms $\xi^s_p\to \cO[1]$ vanishes semi-classically. These two cases are separated since $\xi$ and $\cO[1]$ are not deformed. We consider the natural morphism from $\sB(\xi)$ to the product of stacks, $\prod_{i=1}^{p-1} \R\uchain^{[i,i+1]}(X)$, by sending the chain $\xi\to \xi^{s}_1\to\ldots\to \xi^{s}_{p}\to \cO[1]$ to the tuple
\[
\Big(\xi\to \xi^s_1,\ldots,\xi^s_{i}\to \xi^s_{i+1},\ldots,\xi^s_p\to\cO[1]\Big).
\] Its tangent map and cotangent map is described by diagram (\ref{CD:poissonch}). 
By Corollary \ref{cor:fiberprod}, this is a semi-classically Poisson morphism. The Poisson structure on the target vanishes semi-classically by the above argument. According to the commutative diagram (\ref{CD:poissonch}), the tangent map of 
\[
\sB(\xi) \to \prod_{i=1}^{p-1} \R\uchain^{[i,i+1]}(X)
\]
induces an injection on cohomology and the cotangent map induces a surjection on cohomology. Therefore, the Poisson structure on $\sB(\xi)$ vanishes semi-classically.
\end{proof}

\subsection{Image of the bosonization map}

First, we will show that $\beta_\xi$ is a regular morphism. 

\begin{lemma}
The bosonization map $\xymatrix{\beta_\xi: \sB(\xi)\ar@{-->}[r] & \sN(\xi)}$ is regular everywhere. 
\end{lemma}
\begin{proof}
We consider a slightly different moduli stack $\sB_{n,k}$ whose coarse moduli space is denoted by $B_{n,k}$.  
The stack $\sB_{n,k}$ classifies collections of stable bundles $\xi=\xi_0,\ldots,\xi_p$ with the slope sequence $\mu_{\max}$, equipped with nonzero maps
\begin{align}\label{p-seq}
\xi\to \xi_1\to \ldots \to \xi_p\to \cO[1].
\end{align}
The difference between $\sB_{n,k}$ and $\sB(\xi)$ is that in $\sB_{n,k}$ we allow $\xi$ to deform. 

Let us denote by $\sN_{n,k}$ the moduli stack of nonzero extensions with $\xi$ being a stable bundle of degree $n$ and rank $k$ (but not fixed). Its coarse moduli space $N_{n,k}$ is a projective bundle over $X$. We have the rational bosonization map 
$$\xymatrix{\beta_{n,k}: \sB_{n,k}\ar@{-->}[r] & \sN_{n,k}}$$
sending a collection of maps \eqref{p-seq} to their composition.
It suffices to show $\beta_{n,k}$ is a morphism, i.e., the composition is nonzero.


By Lemma \ref{Si-stable-lem}, we have exact sequences 
\[
0\to S_i\to \xi_i\to \xi_{i+1}\to 0
\] 
for $i<p$, where $S_i$ is a stable bundle of slope $s(i)>0$.
Now the exact sequence
$$0=\Hom(S_i,\cO)\to \Ext^1(\xi_{i+1},\OO)\to \Ext^1(\xi_i,\OO)$$
shows that the map $\Ext^1(\xi_{i+1},\OO)\to \Ext^1(\xi_i,\OO)$ is injective. This implies our assertion.
\end{proof}

By Lemma \ref{lem:Gm}, the bosonization map $\beta_\xi: \sB(\xi)\to \sN(\xi)$ descends to a Poisson morphism of the coarse moduli schemes
$$\b_\xi:B(\xi)\cong X^p\to N(\xi)\cong \P\Ext^1(\xi,\cO_X).$$ 
Since the Poisson structure on the source is zero (Theorem \ref{thm:zero}) its image must lie in the vanishing locus of the Poison structure on $N(\xi)$. The rest of this section devotes to the study of the image of this morphism. 

\begin{theorem}\label{thm:ni>2}
When $n_i>2$ for all $i=1,\ldots, p$, the morphism $\beta_\xi: B(\xi)\to N(\xi)$ is a closed immersion.
\end{theorem}
\begin{proof}
We follow the setup of in the proof of the previous lemma. The extensions $V_i$ of $\xi_i$ by $\OO$, corresponding to the compositions $\xi_i\to \OO[1]$, fit into the diagram
\begin{equation}\label{pullback}
\xymatrix{
\xi\ar[r] & \xi_1\ar[r] &\ldots &\ar[r] &\xi_{p-1}\ar[r] &\xi_p\\
 V_0\ar[r]\ar[u]  & V_1\ar[r]\ar[u]  &\ldots & \ar[r]   &V_{p-1}\ar[u] \ar[r]   &V_p\ar[u]\\
\cO\ar[u]\ar[r]^{=} &  \cO\ar[u]\ar[r]^{=} &\ldots & \ar[r]^= &\cO\ar[u]\ar[r]^= &\cO\ar[u] 
}
\end{equation} 
It follows that we have exact sequences
$$0\to S_i\to V_i\to V_{i+1}\to 0.$$
Let us set $S_p=V_p$ (note that $\xi_p$ is a skyscraper at a point, so $S_p=V_p$ is a line bundle of degree $1$).
Since $n_i>2$ for all $i=1,\ldots,p$, we have
\[
\mu(S_0)>\mu(S_1)>\ldots>\mu(S_p)=1
\] 
(see Lemma \ref{Si-stable-lem})
Since $S_i$ are stable, it follows that $\Ext^1(S_i,S_j)=0$ for $i>j$, so
by descending induction on $i$ we deduce that 
$$V_i=\bigoplus_{j\ge i}^p S_j,$$ 
in particular, $V_0=\bigoplus_{i=1}^p S_i$.
Given $V_0$, the ordered sequence $S_0,\ldots,S_p$ is uniquely determined up to isomorphism as indecomposable factors of $V_0$. It further determines the ordered sequence $\xi_1,\ldots,\xi_p$ up to isomorphism. Since for $i=0,\ldots, p$ the morphism $\xi_i\to \xi_{i+1}$ is unique up to a nonzero scalar, $\beta_{n,k}$ is set theoretically injective.

To show the tangent map is injective, we let $A=\C[\eps]/(\eps^2)$ and $X_{A}:=X\times \Spec A$. Let 
\begin{align}\label{A-family}
\xi_0^A\to \xi_1^A\to \ldots\to \xi_i^A \to \xi^A_p\to \cO_{X_A}[1]
\end{align}  be an $A$-family where  $\xi^{A}_i$ is a  deformation of $\xi_i$ over $A$.  Similarly, we define $S_i^A$ as the kernel of $\xi_i^A\to \xi^A_{i+1}$. Since $S_j^A|_X$ are stable, $S^A_j$ are indecomposable. Then $\beta_{n,k}^A$ ($\beta_{n,k}$ on $A$-points) sends  (\ref{A-family})
to the iterated extension $\cO_{X_A}\to V^A_0$ over $A$. We claim that 
\[
V_0^A\cong \bigoplus_{j} S^A_j.
\] 
It suffices to show that for $i\neq j$, $\Ext^1_{X_A}(S_i^A,S^A_j)=0$. Since the ext group is a finitely generated module over $A$ and vanishes when restricted to the closed point, this follows from the Nakayama lemma. The factors $S_j^A$ are mutually non isomorphic since $S^A_j|_X$ have distinct slopes for distinct $j$. The vector bundle $V_0^A$ determines the ordered sequence of indecomposable factors $S_0^A,\ldots,S_p^A$ (by slope of $S^A_j|_X$), i.e. $\beta^A$ is injective. 
Thus, the morphism $\beta_{n,k}$ (and hence $\beta_\xi$) is a closed embedding. 
\end{proof}

For general $n_i$, the bosonization map is more complicated. To study it we need to fix some notations. Define a decomposition $\tau$ of the index set $\{0,\ldots,p\}=\bigsqcup_{j=1}^q I_j$ by the equivalence relation that $i\sim i+1$ if and only if $n_{i+1}=2$. For example if $n/k=27/8$ then $\{n_1,n_2,n_3,n_4\}=\{4,2,3,2\}$. In this case $\{0,\ldots,4\}=\{0\}\sqcup \{1,2\}\sqcup \{3,4\}$. Let $\mathfrak{S}_{p+1}$ be the symmetric group acting on $\{0,\ldots,p\}$ by permutation. Given a partition $\tau$, denote by $\fS^\tau_{p+1}$ the subgroup of permutation that preserves $\tau$. Let $\alpha: X^{p+1}\to X$ be the addition map with respect to the group law of $X$. 
By the fundamental theorem of symmetric functions $X^{p+1}/\fS_{p+1}^\tau$ is smooth. 
Since the addition map is $\fS_{p+1}$-invariant,  it factors through a smooth morphism
\[
\bar{\alpha}: X^{p+1}\Big\slash \fS^\tau_{p+1} \to X.
\]  

\begin{definition}
A semi-stable sheaf $F$ is called \emph{quasi indecomposable} if the Jordan-Holder factors of its distinct indecomposable summands are non isomorphic. 
Equivalently, $F=\bigoplus_i F_i$, where $F_i$ are indecomposable of the same slope, and $\Hom(F_i,F_j)=0$ for $i\neq j$.
\end{definition}

Indecomposable sheaves are quasi indecomposable. A semi-stable sheaf is quasi indecomposable if and only if there exists an autoequivalence that sends it to the structure sheaf of a zero dimensional subscheme.

\begin{lemma}\label{lem:qi-moduli}
(i) For any semi-stable sheaf $F$ of rank $r$ and degree $d$, one has $\dim\End(F)\ge gcd(r,d)$. One has $\dim\End(F)=gcd(r,d)$ if and only if $F$ is quasi indecomposable.
There is a unique quasi indecomposable representative in any given S-equivalence class of semi-stable sheaves.

(ii) For each given $r\ge 0$ and $d$, let us denote by $\cU^{qi}_{r,d}\sub \cU_{r,d}$ the open substack 
 in the stack $\cU_{r,d}$ of semistable sheaves of rank $r$ and degree $d$, corresponding to sheaves $F$ with $\dim\End(F)\le gcd(r,d)$.
Then $\cU^{qi}_{r,d}$ classifies quasi indecomposable sheaves of rank $r$ and degree $d$.
The coarse moduli scheme of $\cU^{qi}_{r,d}$ is naturally isomorphic to that of $\cU_{r,d}$, and is isomorphic to the symmetric power $X^{(gcd(r,d))}$.
\end{lemma}

\begin{proof}
(i) Using the Fourier-Mukai transforms, we reduce to the case $r=0$. Any torsion sheaf of degree $d$ has form $F=\OO_{n_1p_1}\oplus\ldots \OO_{n_kp_k}$, where
$n_1+\ldots +n_k=d$, for some points $p_1,\ldots,p_k$. 
Hence,
$$\dim\End(F)\ge \sum \dim\End(\OO_{n_i p_i})=\sum n_i=d,$$
with the equality if and only if all $p_i$ are distinct. The last assertion follows from the fact that the $S$-equivalence class of $F$ is determined by the divisor $\sum n_i p_i$.

\noindent
(ii) By semicontinuity, the condition $\dim \End(F)\le gcd(r,d)$ defines an open substack. By part (i), this condition picks out exactly quasi indecomposable bundles.
Hence, the coarse moduli space of $\cU^{qi}_{r,d}$, denote by $U^{qi}_{r,d}$
is an open subscheme in the coarse moduli space $U_{r,d}$ (which is known to be isomorphic to $X^{(gcd(r,d))}$.
The fact that it is the entire $U_{r,d}$ follows from (i).
\end{proof}

Set $d(I_j):=\sum_{i\in I_j} d(i)$, $r(I_j):=\sum_{i\in I_j}r(i)$ and $\cU^{qi}_{r(I_j), d(I_j)}$ to be the moduli stack of quasi indecomposable sheaves of rank $r(I_j)$ and degree $d(I_j)$. The coarse moduli space of
$\cU^{qi}_{r(I_1), d(I_1)}\times \cU^{qi}_{r(I_2), d(I_2)}\times\ldots\times \cU^{qi}_{r(I_q), d(I_q)}$ is isomorphic to $X^{p+1}\Big\slash \fS^\tau_{p+1}$. 

\begin{lemma}\label{lem:V0}
A point $\{\cO\to V_0\}$ of $\sN_{n,k}$ lies in the image of $\beta_{n,k}: \sB_{n,k}\to \sN_{n,k}$ if and only if $V_0\cong \bigoplus_{j=1}^q U_j$ with $U_j\in \cU^{qi}_{r(I_j),d(I_j)}$. 
\end{lemma}
\begin{proof}
Recall that if $n_i=2$ then $s_{i-1}=s_i$. Consider the diagram
\[
\xymatrix{
\xi_{i-1}\ar[rr] & &\xi_i \ar[ld]_{[1]}\ar[rr] && \xi_{i+1}\ar[ld]_{[1]}\\
& S_{i-1}\ar[lu] && S_i\ar[lu]\ar[ll]^{[1]}_{\circlearrowleft}
}
\] where the left and right triangles are exact  and the middle one is a composition. Let $\{0,\ldots,p\}=\sqcup_{j=1}^q I_j$ be the decomposition $\tau$  constructed above. It follows that for all $
\ell\in I_j$ for a fixed $j$, $S_\ell$ has the same slope. Set $U_j$ to be the kernel of $\xi_{i_j}\to \xi_{i_{j+1}}$. Then $U_j$ is the iterated extension
\[
\xymatrix{
U_j\ar[r]  & \ldots & \star\ar[rr] && \star\ar[ld]_{[1]}\ar[rr] && S_{i_{j+1}}\ar[ld]_{[1]}\\
&  S_{i_{j}}\ar[ul] & \ldots &  S_{i_{j+1}-2}\ar[ul] & &  S_{i_{j+1}-1}\ar[ul]
}
\] and $V_0$ is an iterated extension by $U_j$. Since slope of the stable factors of $U_j$ strictly decreases as $j$ increases, $V_0\cong \bigoplus_{j=1}^q U_j$. It remains to show that $U_j$ is quasi indecomposable.

Let  $U_j=\bigoplus_{k} U_{j,k}$ such that $U_{j,k}$ are indecomposable. It is clear that all stable factors of $U_{j,k}$ are isomorphic given a fixed $k$. We will show that stable factors associate to distinct $k$ are non isomorphic by contradiction. Without loss of generality, we assume $S_{i_j}=S$ and $i_j\in I_{j,1}$ and $S_\ell=S$ for $\ell\in I_{j,2}$.  Since
\begin{align}
U_j\to \xi_{i_j}\to \xi_{i_{j+1}}
\end{align} is exact, 
there is an injection $\Hom(S, U_j)\to \Hom(S,\xi_{i_j})$ where $\Hom(S,\xi_{i_j})$ is one dimensional. But $\Hom(S, U_j)$ is at least 2-dimensional by our assumption which leads to a contradiction.
The tuple $(S_{i_j},\ldots,S_{i_{j+1}-1})$ are exactly Jordan-Holder factors of $U_j$. On the other hand $U_j$ is a canonical representative in the $S$-equivalence class of semi-stable bundles with these Jordan-Holder factors. Since any sequence of stable factors $S_0,\ldots,S_p$ can be realized as sub quotients of a sequence
\[
\xi_0=\xi\to \xi_1\to\ldots\xi_p\to\cO[1]
\] we conclude that for any $U_j\in \cU^{qi}_{r(I_j),d(I_j)}$ a complex $\cO\to V_0=\bigoplus_{j=1}^q U_j$ is in the image of $\beta_{n,k}$.

\end{proof}

\begin{lemma}\label{lem:unique section}
Given a $V_0$ in the previous lemma, there exists a unique point $\phi:\cO\to V_0$ in $\sN_{n,k}$ up to automorphisms of $V_0$.
\end{lemma}
\begin{proof}
Since $\sN_{n,k}$ is a $\G_m$-gerbe over a scheme, the tangent complex at $\{\cO\to V_0\}$ has one dimensional cohomology group at degree $-1$. Therefore the linear map
\[\xymatrix{
\End(V_0)\ar[r]^{\cdot\phi} & \Hom(\cO, V_0)}
\] is injective. It is enough to prove that it is an isomorphism. Indeed, then for any other $\phi':\cO\to V_0$ such that the cokernel is isomorphic to $\xi$,
we would have an endomorphism $A:V_0\to V_0$ such that $A\phi=\phi'$. Then $A$ induces a morphism of exact sequences
\[\xymatrix{
0\ar[r] & \cO\ar[r]^\phi\ar[d]^\id & V_0\ar[r]\ar[d]^A& \xi\ar[r]\ar[d]^\alpha & 0\\
0\ar[r] & \cO\ar[r]^{\phi^\prime} & V_0\ar[r] & \xi\ar[r] & 0
}
\]
Since $\Hom(V_0,\cO)=0$, we necessarily have $\alpha\neq 0$. Hence, $\alpha$ is an automoprhism of $\xi$, so $A$ is an automorphism of $V_0$.


Thus, it suffices to check that $\dim_{\C} \End(V_0)=\dim_{\C} \Hom(\cO,V_0)$. By diagram (\ref{pullback}), 
\[
\dim_{\C} \Hom(\cO,V_0)=\dim_{\C} H^0(\xi)=n.
\] 
Recall that $\deg(S_i)=d(i)$ and  $\rk(S_i)=r(i)$. Since $V_0=\bigoplus U_j$ and $\Hom(U_{j_1},U_{j_2})=0$ if $j_1<j_2$, we have 
\[
\dim_{\C} \End(V_0)=p+1+\sum_{0\leq i_1<i_2\leq p} d(i_1) r(i_2)-d(i_2) r(i_1).
\]
Set $v_i=(k(i)+1,n(i))$ with $i=0,\ldots, p$ and let $\Gamma$ be the piecewise linear curve generated by  $v_{p+1}:=(0,0),v_p,\ldots,v_0=(k+1,n)$. Since $r(i)=k(i)-k(i-1)$ and $d(i)=n(i)-n(i-1)$, $(r(i),d(i))=v_i-v_{i-1}$. Moreover, $\Gamma$ is concave and lies inside the triangle $\Delta$ with vertices $\{(0,0), (1,0), (k+1,n)\}$. It is easy to compute that the area of the subregion of $\Delta$ below $\Gamma$ is $\frac{p+1}{2}$ and the area of the subregion above $\Gamma$ is $\frac{1}{2}\sum_{0\leq i_1<i_2\leq p} d(i_1) r(i_2)-d(i_2) r(i_1)$. Then we prove that $\dim_{\C} \End(V_0)=n$.
\end{proof}.

\begin{theorem}\label{thm:general ni}
Let $n_1,\ldots,n_p$ be integers greater or equal than two. The image of $\beta_{n,k}: B_{n,k}\to N_{n,k}$ is isomorphic to $X^{p+1}\Big\slash \fS^\tau_{p+1}$. 
The image of $\beta_\xi: B(\xi)\to N(\xi)$ is isomorphic to a fiber of the addition map $\bar{\alpha}: X^{p+1}\Big\slash \fS^\tau_{p+1}\to X$.
\end{theorem}
\begin{proof}
By Lemma \ref{lem:V0} and Lemma \ref{lem:qi-moduli}, the image of the stack morphism $\beta: \sB_{n,k}\to \sN_{n,k}$ is the sub moduli space consisting of complexes $\{\cO\to V_0\}$ such that $V_0/\cO\cong \xi_0=\xi$ and $V_0\cong \bigoplus U_j$ with $U_j$ being quasi indecomposable bundles whose stable factors are of rank $r(\ell)$ and degree $d(\ell)$ for $\ell\in I_j$. By Lemma \ref{lem:unique section}, the sub moduli stack defined above is a $\G_m$-gerbe over the coarse moduli scheme of $U_{r(I_1), d(I_1)}\times U_{r(I_2), d(I_2)}\times\ldots\times U_{r(I_q), d(I_q)}$ that is isomorphic to $X^{p+1}\Big\slash \fS^\tau_{p+1}$. 
Under such an isomorphism, the determinant map 
\[ (V_1,\ldots,V_q)\in U_{r(I_1), d(I_1)}\times U_{r(I_2), d(I_2)}\times\ldots\times U_{r(I_q), d(I_q)}\mapsto \det V_1\ot\ldots\ot\det V_q \] is identified with the addition map $\bar{\alpha}: X^{p+1}\Big\slash \fS^\tau_{p+1}\to X$.
\end{proof}

Later we will also need the following result about deformations of a generic point in the image of $\beta_{n,k}$.

\begin{lemma}\label{qi-formal-lem}
Let $V_0=\oplus_{j=1}^q U^0_j$ be in the image of $\beta_{n,k}$, where $U_j$ is a generic point of $\UU^{qi}_{r(I_j),d(I_j)}$. Suppose
$V$ is a family of vector bundles on $C$ over a reduced base $B$, such that $V_0=V_{b_0}$ for some point $b_0\in B$, such that
$\dim\End(V_b)=\dim \End(V_0)$ for all $b\in B$. Then the induced family over the formal neighborhood of $b_0$ in $B$ lies in the image of $\beta_{n,k}$.
\end{lemma}

\begin{proof}
The assumptions imply that $\pi_*\und{\End}(V)$ is a vector bundle of rank $\dim \End(V_0)$ over $B$, and that this remains true after any base change.

Let us consider the functor $\Def_f(V_0)$ of deformations of $V_0$ preserving the filtration $F^iV_0=\oplus_{j\le i} U^0_j$. We claim that the natural morphism $\Def_f(V_0)\to \Def(V_0)$
is smooth (here $\Def(V_0)$ is the functor of all deformations of $V_0$). Indeed, let $\und{\End}_f(V_0)\sub \und{\End}(V_0)$ denote the subsheaf of endomorphisms preserving the
filtration. Then the exact sequence
$$0\to \und{\End}_f(V_0)\to \und{\End}(V_0)\to \oplus_{i<j} \und{\Hom}(U^0_i,U^0_j)\to 0,$$
together with the vanishing of $\Ext^1(U^0_i,U^0_j)$ for $i<j$, imply that the map $H^1(\und{\End}_f(V_0))\to H^1(\und{\End}(V_0))$ is surjective, which proves our claim.

Let us replace $B$ by the formal neighborhood $\Spf(A)$ of $b_0$. Then we can equip $V$ with a filtration $F^iV$ by subbundles deforming the filtration $F^iV_0$.
Now we have an exact sequence of $A$-modules
$$0\to \End_f(V)\to \End(V)\rTo{\de} \oplus_{i<j} \Hom(U_i,U_j)$$
where $U_i=F^iV/F^{i-1}V$. We know that $\End(V)$ and $\Hom(U_i,U_j)$ are free $A$-modules (recall that $\Ext^1(U_i,U_j)=0$).
Furthermore, the reduction of $\de$ modulo maximal ideal is surjective (since $V_0$ splits), hence, $\de$ itself is surjective. Thus, $\End_f(V)$ is a free $A$-module, and
its reduction modulo maximal ideal is identified with $\End_f(V_0)$. 

Now let us consider the map of $A$-modules
$$\End_f(V)\rTo{\rho} \oplus_i \End(U_i).$$
Note that since $U^0_j$ is a generic point of $\UU^{qi}_{r(l_j),d(l_j)}$, it is a direct sum of pairwise non-isomorphic stable bundles of the same slope.
This implies that $U_i$ is also a family in $\UU^{qi}_{r(l_j),d(l_j)}$ and $\End(U_i)$ is a free $A$-module.
Since the reduction of $\rho$ modulo maximal ideal is an isomorphism $\End_f(V_0)\to \oplus_i \End(U^0_i)$, the map $\rho$ is also an isomorphism.

We claim that this implies that the filtration $(F^iV)$ splits, so $V\simeq \oplus_i U_i$ is in the image of $\beta_{n,k}$ (see Lemma \ref{lem:V0}). 
Indeed, we can use induction on $n$. Suppose we already know that $V=F^iV\oplus U_{i+1}\oplus\ldots$. Then from surjectivity of $\rho$ we deduce surjectivity of 
the similar map 
$$\End_f(F^iV)\to \oplus_{j\le i} \End(U_j).$$
Hence, there exists a filtered endomorphism $f:F^iV\to F^iV$ such that the induced map $f_i:U_i\to U_i$ is identity, while the induced maps $U_j\to U_j$ for $j<i$ are zero.
Since $\Hom(U_j,U_{j'})=0$ for $j'<j$, we get that $f(F^{i-1}V)=0$. Therefore, $f$ factors through a map $U_i\to F^iV$ which gives a splitting of the projection $F^iV\to U_i$.
\end{proof}

\section{Vanishing of Poisson vector fields with respect to $q_{n,k}$}\label{van-vec-field-sec}

\subsection{Vector fields on a projective space vanishing on a nondegenerate variety}

\begin{lemma}\label{nondeg-van-vec-f-lem}
Let $Z\sub \P^n$ be a connected Zariski closed subset not contained in any hyperplane. Then any global vector field $v$ on $\P^n$, vanishing at every point of $Z$,
is zero.
\end{lemma}

\begin{proof}
Let us equip $Z$ with the reduced subscheme structure and let $I_Z\sub \OO_{\P^n}$ be the corresponding ideal sheaf.
The Euler exact sequence
$$0\to \OO\to V\ot \OO(1)\to T_{\P^n}\to 0$$
induces a commutative diagram with exact rows
\begin{diagram}
H^0(\OO_{\P^n})&\rTo{}& V\ot H^0(\OO_{\P^n}(1))&\rTo{}& H^0(T_{\P^n})&\rTo{}& 0\\
\dTo{\simeq}&&\dTo{\a}&&\dTo{\b}\\
H^0(\OO_Z)&\rTo{}& V\ot H^0(\OO_Z(1))&\rTo{}& H^0(T_{\P^n}|_Z)
\end{diagram}
Note that the map $\a$ is injective since $Z$ is not contained in any hyperplane. Now an easy diagram chasing shows that $\b$ is also injective, as claimed.
\end{proof}

\subsection{Vector fields preserving a subvariety}

Here we work over the characteristic zero ground field $k$.

Recall that for a line bundle $L$ on a smooth variety $X$, the {\it Atiyah bundle} $\cA_L$ is the sheaf of infinitesimal symmetries of $L$. 
It fits into an exact sequence
\begin{equation}\label{At-ex-seq}
0\to \cO_X\to \cA_L\to \cT_X\to 0
\end{equation}
and the corresponding extension class is the first Chern class $c_1(L)\in H^1(X,\Om_X^1)$.

The following observation will be crucial for us.

\begin{lemma}\label{At-vec-field-lem} 
Let $X\hra \P^n$ be a locally closed embedding of a smooth variety, and 
let $L=\cO(1)|_X$. Then for any global vector field $v$ on $\P^n$ preserving $X$,
the corresponding vector field on $X$ 
lifts to a global section of the Atiyah bundle $\cA_L$.
\end{lemma}

\begin{proof}
Let us 
denote by $\cT_{\P^n,X}\sub \cT_{\P^n}$  
the subsheaf of vector fields $v$ preserving $X$, i.e., such that $v(\cI)=\cI$, where $\cI$ is the ideal of $X$.
We denote by $\cA_{\cO(1),X}$ the preimage of $\cT_{\P^n,X}$ in the Atiyah bundle of $\cO(1)$ on $\P^n$, so that we have an exact sequence
$$0\to \OO_X\to \cA_{\cO(1),X}\to \cT_{\P^n,X}\to 0$$
Note that since $H^1(\P^n,\OO)=0$, every global vector field on $\P^n$ preserving $X$ lifts to a global section of $\cA_{\cO(1),X}$.
By the naturality of the definition of the Atiyah bundle, we have a morphism 
$$\cA_{\cO(1),X}\to i_*\cA_{\cO(1)|_X}$$
compatible with the natural map $\cT_{\P^n,X}\to \cT_X$, where $i:X\to \P^n$ is the embedding.
Thus, a global section of $\cA_{\cO(1),X}$ lifting $v$ induces a global section of $\cA_L$ lifting the corresponding vector field on $X$.
\end{proof}

\begin{definition} Let $X$ be a smooth variety, $L$ a line bundle over $X$. We say that the pair $(X,L)$ is {\it infinitesimally rigid} if the map
$H^0(X,\cA_L)\to H^0(X,\cT_X)$ is zero, i.e., only trivial vector field on $X$ lifts to a global section of $\cA_L$.
\end{definition}

\begin{lemma}\label{inf-rigid-lem}
The pair $(X,L)$ is infinitesimally rigid if and only if the map
\begin{equation}\label{c1-map-eq}
H^0(X,\cT_X)\rTo{\cup c_1(L)} H^1(X,\cO_X)
\end{equation}
is injective. 
\end{lemma}

\begin{proof}
This follows immediately from 
the exact sequence the cohomology long exact sequence
$$H^0(X,\cA_L)\to H^0(X,\cT_X)\rTo{\cup c_1(L)} H^1(X,\cO_X)\to\ldots$$
associated with \eqref{At-ex-seq}.
\end{proof}

Using Lemma \ref{At-vec-field-lem}, we get the following general criterion.

\begin{prop}\label{inf-rigid-emb-prop} 
Let $X\sub \P^n$ be a smooth connected projective variety, 
not contained in any hyperplane. 
Assume that the pair $(X,\cO(1)|_X)$ is infinitesimally rigid.
Then there are no nonzero vector fields on $\P^n$ preserving $X$.
\end{prop}

\begin{proof} Set $L=\cO(1)|_X$.
By Lemma \ref{At-vec-field-lem}, a vector field $v$ on $\P^n$ preserving $X$ induces a global vector field $v_X$ on $X$ that lifts to a global section of the Atiyah bundle $\cA_L$.
Since $(X,L)$ is infinitesimally rigid, we deduce that $v_X=0$. Hence, the vector field $v$ vanishes at every point of $X$. By Lemma \ref{nondeg-van-vec-f-lem}, this implies that $v=0$.
\end{proof}

Next, we will give examples of infinitesimally rigid pairs.
We start with the following general observation.

\begin{lemma}\label{inf-rigid-pullback-lem} 
Let $f:Y\to X$ be a morphism of smooth varieties, such that $f^*:H^1(\OO_X)\to H^1(\OO_Y)$ is injective and the image of the restriction map
$H^0(T_X)\to H^0(f^*T_X)$ is injective and its image is contained in the image of the tangent map $H^0(T_Y)\to H^0(f^*T_X)$. Let $L$ be a line bundle on $X$,
such that $(Y,f^*L)$ is infinitesimally rigid, then $(X,L)$ is infinitesimally rigid.
\end{lemma}

\begin{proof} By Lemma \ref{inf-rigid-lem} and by injectivity of $H^1(\OO_X)\to H^1(\OO_Y)$, it is enough to check the injectivity of the composition
$$H^0(T_X)\to H^0(f^*T_X)\rTo{\cup f^*c_1(L)}\to H^1(\OO_Y).$$
Since the first arrow is injective and its image is contained in the image of $H^0(T_Y)$, the assertion follows from 
injectivity of the composition 
$$H^0(T_Y)\to H^0(f^*T_X)\rTo{\cup f^*c_1(L)}\to H^1(\OO_Y),$$
which holds by infinitesimal rigidity of $(Y,f^*L)$.
\end{proof}

\begin{definition}\label{X-variety-def}
For an elliptic curve $E$ and an $r$-tuple of positive numbers $(m_1,\ldots,m_r)$, let 
$$X(m_1,\ldots,m_r)\sub E^{(m_1)}\times\ldots \times E^{(m_r)}$$ 
denote the fiber over $0\in E$ of the map
$$E^{(m_1)}\times\ldots \times E^{(m_r)}\to E: (D_1,\ldots,D_r)\mapsto a(D_1)+\ldots+a(D_r),$$
where $a:E^{(m_i)}\to E$ is the addition map. 
\end{definition}

Let $Y_r\sub E^r$ denote the kernel of the addition map $E^r\to E$ (so $Y_r$ is isomorphic to $E^{r-1}$). Then we have a natural smooth morphism
$$p:X(m_1,\ldots,m_r)\to Y_r$$
whose fibers are products of projective spaces.
On the other hand, set
$$Z(m_1,\ldots,m_r):=\{(x_1,\ldots,x_r)\in E^r \ | m_1x_1+\ldots+m_rx_r=0\}.$$
This is a subgroup of $E^r$, and the component of zero, $Z(m_1,\ldots,m_r)^0$ is an abelian variety of dimension $r$.
Furthermore, we have a natural isogeny,
$$p:Z(m_1,\ldots,m_r)^0\to Y_r: (x_1,\ldots,x_r)\mapsto (m_1x_1,\ldots,m_rx_r).$$

We have a natural action of $Z(m_1,\ldots,m_r)^0$ on $X(m_1,\ldots,m_r)$, given by 
$$(x_1,\ldots,x_r):(D_1,\ldots,D_r)\mapsto (t_{x_1,*}D_1,\ldots,t_{x_r,*}D_r),$$
where $t_x:E\to E$ is the translation automorphism by $x\in E$. This action is compatible with the projections $p$ and the action of $Y_r$
on itself by translations.

\begin{lemma}\label{tangent-sym-lem}
Let $E$ be an elliptic curve. 
For an $r$-tuple of positive numbers $(m_1,\ldots,m_r)$, such that $r\ge 2$, consider the variety $X=X(m_1,\ldots,m_r)$ equipped with the action of $Z^0=Z(m_1,\ldots,m_r)^0$.
Then the the space $H^0(X,T_X)$ is $(r-1)$-dimensional and is spanned by the global vector fields coming from the action of $Z^0$.
\end{lemma}

\begin{proof}
Let us first consider the relative tangent map $T_a$ for the addition morphism $a:E^{(m)}\to E$. It is well known that $E^{(m)}\simeq \P V$,
where $V$ is a stable bundle over $E$ (namely, $V$ is the Fourier transform of a line bundle of degree $m$ on $E$).
Hence, we have $Ra_*\OO_{\P V}\simeq \OO_E$,
and the exact sequence
$$0\to \OO_{\P V}\to a^*V(1)\to T_p\to 0.$$
gives rise to an isomorphism on $E$,
\begin{equation}\label{T-rel-addition-map}
Ra_*T_a\simeq \und{\End}_0(V).
\end{equation}

Now we have
$$X\simeq \P(p_1^*V_1)\times_{Y_r}\times\ldots \times_{Y_r} \P(p_r^*V_r),$$
where $V_i$ is the stable bundle on $E$ such that $E^{(m_i)}\simeq \P(V_i)$, and $p_i:Y_r\to E$ is the projection to
the $i$th factor. Thus, we can identify $T_p$ with the direct sum of pull-backs of the relative tangent bundles to
$\P(p_i^*V_i)\to Y_r$. Hence, taking into account \eqref{T-rel-addition-map}, we get an isomorphism on $Y_r$,
$$Rp_*T_p\simeq \bigoplus_{i=1}^r p_i^*\und{\End}_0(V_i).$$

Since $r\ge 2$, each projection $p_i:Y_r\to E$ can be identified with the projection of the direct product $E\times E^{r-2}$ to the first factor.
Since $H^*(E,\und{\End}_0(V_i))=0$, this implies by Kunneth formula that $H^*(p_i^*\und{\End}_0(V_i))=0$.
Therefore, we deduce that $H^*(X,T_p)=0$.
Now the first exact sequence 
$$0\to T_p\to T_X\to p^*T_{Y_r}\to 0$$
gives an isomorphism 
$$H^0(T_X)\rTo{\sim} H^0(p^*T_{Y_r})\simeq H^0(T_{Y_r}).$$

Now the assertion follows from the fact that we have compatible actions of $Z^0$ on $X$ and $Y_r$,
and that the map $T_0Z^0\to H^0(T_{Y_r})$ induced by this action is an isomorphism.
\end{proof}

\begin{prop}\label{ab-sym-inf-rigid-prop} 
(i) Let $A$ be an abelian variety. Then for every ample line bundle $L$ on $A$, the pair $(A,L)$ is infinitesimally rigid.

\noindent
(ii) Consider the variety $X=X(m_1,\ldots,m_r)$ associated with an elliptic curve $E$ and an $r$-tuple $(m_1,\ldots,m_r)$, where $r\ge 2$
(see Definition \ref{X-variety-def}).
Then for every ample line bundle $L$ on $X$, the pair $(X,L)$ is infinitesimally rigid.
\end{prop}

\begin{proof} 
(i) By Lemma \ref{inf-rigid-lem}, it is enough to check that for any ample line bundle $L$ on $A$ the map \eqref{c1-map-eq} is injective.
Using the triviality of the tangent bundle to $A$, we get an identification
$H^1(A,\Om^1_A)\simeq H^1(A,\OO)\ot T^*$, where $T=T_0A$. It is well known that for ample $L$ the tensor $c_1(L)\in H^1(A,\OO)\ot T^*$ is nondegenerate.
Hence, the corresponding cup product map
$$H^0(A,T_A)\simeq T\to H^1(A,\OO)$$ 
is an isomorphism.

(ii) Since $p:X\to Y_r$ satisfies $Rp_*\OO_X\simeq \OO_{Y_r}$, the map $p^*:H^1(\OO_{Y_r})\to H^1(\OO_X)$ is an isomorphism.
Recall that we have an action of the abelian variety $Z^0$ on $X$.
Let $f:Z^0\to X$ given by the action on some point $x\in X$ (so its image is the corresponding orbit). 
Then the composed map $Z^0\to X\to Y_r$ is the composition of an isogeny with a translation,
so it induces an isomorphism $H^1(\OO_{Y_r})\to H^1(\OO_{Z^0})$. It follows that the restriction map $H^1(\OO_X)\to H^1(\OO_{Z^0})$ is an isomorphism.

On the other hand, Lemma \ref{tangent-sym-lem} implies that the map $H^0(X,T_X)\to H^0(Z^0,f^*T_X)$ is injective and its image is contained in the subspace
$H^0(Z^0,T_{Z^0})$. Note also that the map $f$ is finite, so for any ample line bundle $L$, the pull-back $f^*L$ is still ample. Now the assertion follows 
from Lemma \ref{inf-rigid-pullback-lem} and from part (i).
\end{proof}

\begin{cor}\label{no-vec-fields-cor} 
Let $X$ be a variety as in Proposition \ref{ab-sym-inf-rigid-prop}(ii). Then for any projective embedding
$X\sub \P^n$, such that $X$ is not contained in any hyperplane, there are no nonzero vector fields on $\P^n$,
preserving $X$.
\end{cor} 

\subsection{Application to our situation}


\begin{lemma}\label{nondeg-Xp-lem} 
Assume that $\mu(v)>1$. Then the image of the morphism $\beta_\xi:B(\xi)\to N(\xi)$
 is not contained in any hyperplane.
\end{lemma}

\begin{proof} 
Recall that $N(\xi)$ is isomorphic to $\P H^0(E,\xi)^*\simeq \P H^1(E,\xi^\vee)$.
Let us consider the group $H$ of pairs $(x\in E, \a:\xi\rTo{\sim} t_x^*\xi)$. It is well known that $H$ is a Heisenberg group,
which is an extension of $E_n$ (the group of $n$-torsion points in $E$) by $\G_m$, and the representation of $H$ on $H^0(E,\xi)^*\simeq H^1(E,\xi^\vee)$
is irreducible. Let us consider the induced action of $E_n$ on the projective space $\P H^1(E,\xi^\vee)$.
It is easy to see that the map $\beta_\xi: B(\xi)\to N(\xi)$ is equivariant with respect to the group $E_n$.
It follows that the image of $\beta_\xi$ in $\P H^1(E,\xi^\vee)$ is $E_n$-invariant. But any $E_n$-orbit contained in the image of $\beta_\xi$ linearly spans the ambient projective space, by irreducibility.
This implies the assertion.
\end{proof}

\begin{theorem} \label{thm:nosym}
Assume $n>k+1$. Then there are no nonzero Poisson vector fields for the Feigin-Odesskii bracket
$q_{n,k}$ on $\P^{n-1}$.
\end{theorem}

\begin{proof} Any Poisson vector field preserves every irreducible component of the vanishing locus of the Poisson structure. Lemma \ref{qi-formal-lem} implies that
the image of the morphism $\b_\xi$ is such a component. Also, by Lemma \ref{nondeg-Xp-lem}, it is not contained in any hyperplane, and by
Theorem \ref{thm:general ni}, it is isomorphic to the variety of the form $X(m_1,\ldots,m_r)$ with $r\ge 2$ (see Definition \ref{X-variety-def}). Now, by Corollary \ref{no-vec-fields-cor}, any vector field preserving the
image of $\b_\xi$ is zero.
\end{proof}

\section{Case $k=1$ with arbitrary $m$}\label{sec:p=1}

\subsection{A formula for the Poisson bracket}

Let us consider in more detail the case $k=p=1$. Then $\xi$ is a line bundle of degree $n$, and we can consider an open substack $\sB'(\xi,m)\subset \sB(\xi,m)$
consisting of 
$\xymatrix{\xi\ar[r]^{\phi} & \cO_D \ar[r]^{u} & \cO[1]}$
such that $\phi\neq 0$ and $u$ is a generator of $\Ext^1(\cO_D,\cO)$ as an $H^0(\cO_D)$-module.
Then we can identify $\sB'(\xi,m)$ with a $\G_m$-gerbe over a certain projective bundle over $X^{(m)}$. Namely, we can replace $\cO_D$ by an isomorphic sheaf $\cO_D(D)$.
Then every $u$ is obtained by applying an automorphism of $\cO_D$ to the canonical element $u_0\in \Ext^1(\cO_D(D),\cO)$, the class of the extension
$$0\to \cO\to \cO(D)\to \cO_D(D)\to 0.$$
Now $\phi$ is a nonzero element of $H^0(\xi^{-1}(D)|_D)$, and automorphisms of $\xi$ act on the latter space by rescaling.
Thus, the coarse moduli space $B'(\xi,m)$ of $\sB'(\xi,m)$ can be identified with $\P \cV$, where $\cV$ is a vector bundle over $X^{(m)}$ with the fiber $H^0(\xi^{-1}(D)|_D)$ over $D$.
Our goal is to calculate the corresponding classical Poisson bracket on $B'(\xi,m)\simeq \P \cV$. 

Let 
$$\tau_D:H^0(\OO_D(D))\to k$$
denote the canonical residue functional (corresponding to a choice of the trivialization of $\om_X$). 
We set
$$H_D:=\ker(\tau_D)\sub H^0(\OO_D(D)).$$
The exact sequence $0\to \OO\to \OO(D)\to \OO_D(D)\to 0$ induces an identification $H_D\simeq H^0(\OO(D))/k$. We consider the map
$$F_D:H_D\rTo{\sim} H^0(\OO(D))/k\to H^0(\OO(D)|_{2D})/k.$$
For each $\phi\in H^0(\xi^{-1}(D)|_D)$, we consider the induced map
$$\rho_\phi: H^0(\xi|_{2D})\to H^0(\xi|_D)\rTo{\phi} H^0(\OO(D)|_D).$$

\begin{prop}\label{prop:poi-p=1}
Assume $\phi$ induces an isomorphism $\OO_D\to \xi^{-1}(D)|_D$. Then 

\noindent
(i) we have a natural identification of the cotangent spaces to $B'(\xi,m)$ at the point corresponding
to $(D,\phi)$, where $\phi\in H^0(\xi^{-1}(D)|_D)$:
$$T_{(D,\phi)}B'(\xi,m)\simeq \rho_\phi^{-1}(H_D)\sub H^0(\xi|_{2D}).$$


\noindent
(ii) Under the identification of (i), the Poisson bracket on $B'(\xi,m)$ corresponds to the skew-symmetric form on $\rho_\phi^{-1}(H_D)$ given by
\begin{equation}\label{Pi-pairing-formula}
\lan x,x'\ran=\tau_D\bigl([-F_D(\rho_\phi(x))\cdot x'+F_D(\rho_\phi(x'))\cdot x]\cdot \phi\bigr).
\end{equation}
Here we view the product $F_D(\rho_\phi(x))\cdot x'$ as an element of $H^0(\xi(D)|_{2D})/\lan x'|_D\ran$ and the difference
$-F_D(\rho_\la(x))x'+F_D(\rho_\la(x'))x$ is in fact a well defined element of $H^0(\xi|_D)/\lan x|_D, x'|_D\ran$.
Since $\tau_D(\lan x|_D, x'|_D\ran\cdot\phi)=0$, the right-hand side of our formula is well defined.
\end{prop}

\begin{proof}
\noindent
{\bf Step 1. General formula for the contangent space}.
We adopt the notation of Section \ref{sec:chains}.
We start with the identification of the tangent space
$$TB'\simeq \ker(H^1\Big(\Cone[E_{00}\oplus E_{11}\oplus E_{22}\rTo{d_h} E_{01}\oplus E_{12}][-1]\Big)\to H^1(E_{00}\oplus E_{22})),$$
where the differential $d_h$ is $(a_{00},a_{11},a_{22})\mapsto (da_{00}-a_{11}d,da_{11}-a_{22}d)$,
and the dual identification of the cotangent space with
\begin{equation}\label{T*B-first-identification}
T^*B'\simeq \coker(H^0(E_{00}\oplus E_{22})\to H^1\Big(\Cone[E_{10}\oplus E_{21}\rTo{d_h^\vee} E_{00}\oplus E_{11}\oplus E_{22}][-1]\Big)),
\end{equation}
where $d_h^\vee(a_{10},a_{21})=(-a_{10}d,da_{10}-a_{21}d,da_{21})$.
Using a natural exact sequence of complexes, we can rewrite the latter identification as
\begin{equation}\label{alpha-map-eq}
T^*B'\simeq \ker(H^1(L^\bullet)\rTo{\a} H^1(E_{00}\oplus E_{22})),
\end{equation}
where 
$$L^\bullet:=\Big(\Cone[E_{10}\oplus E_{21}\rTo{\de_h} E_{11}][-1]\Big), \ \ \de_h(a_{10},a_{21})=da_{10}-a_{21}d,$$ 
and
the map $\a$ factors through the map $H^1(E_{10}\oplus E_{21})\to H^1(E_{00}\oplus E_{22})$, induced
by the morphism $E_{10}\oplus E_{21}\to E_{00}\oplus E_{22}: (x_{10},x_{21})\mapsto (x_{10}d,dx_{21})$.

\noindent
{\bf Step 2. Choice of resolutions}.
In our case we will choose certain complexes $V_0$, $V_1$, together with quasi-isomorphisms $V_0\to \xi$, $V_1\to \OO_D(D)$, and we set $V_2=\OO[1]$,
so that both morphisms $\phi:\xi\to \OO_D(D)$ and $u:\OO_D(D)\to \OO[1]$ are represented by chain maps $\phi:V_0\to V_1$, $u:V_1\to V_2$.
For this we pick an auxiliary positive divisor $E$, such that $E\cap D=\emptyset$, and lift $\phi$ to a morphism $\wt{\phi}:\xi\to \OO(E+D)$ (where we use
an isomorphism $\OO(E+D)|_D\simeq \OO_D(D)$). Then we consider a complex concentrated in degrees $[-1,0]$,
$$V_0:=[\xi(-D-E)\rTo{(-1,1)} \xi(-D)\oplus \xi(-E)].$$
Note that the natural embedding $\xi(-D)\to \xi$ and $\xi(-E)\to \xi$ induce a quasi-isomorphism $V_0\to \xi$. Next, we consider the standard resolution of $\OO_D(D)$,
$$V_1:=[\OO\to \OO(D)],$$
concentrated in degrees $[-1,0]$. We have a natural map $V_1\to \OO[1]$ with the unique nontrivial component $\id:\OO\to \OO$.
Also, $\phi$ is represented by the chain map $\phi_V:V_0\to V_1$ which is given by $\wt{\phi}$ on $\xi(-D-E)$ and $\xi(-E)$, and is zero on $\xi(-D)$.

\noindent
{\bf Step 3. Computation of $H^1(L^\bullet)$}.
We have natural quasi-isomorphisms 
$$E_{10}\to \und{\Hom}(V_1,\xi)=[\xi(-D)\rTo{-1} \xi],$$ 
$$E_{21}=\und{\Hom}(\OO[1],V_1)\to \und{\Hom}(\OO[1],\OO_D(D))=\OO_D(D)[-1], \text{ and }$$
$$E_{11}\to \und{\Hom}(V_1,\OO_D(D))=\OO_D\oplus \OO_D(D)[-1],$$
and the following diagram of chain maps of complexes, where the vertical arrows are quasi-isomorphisms: 
\begin{equation}\label{E21-E10-to-E11-eq}
\begin{diagram}
E_{10}\oplus E_{21}&\rTo{\de_h}& E_{11}\\
\dTo{}&&\dTo{}\\
[\xi(-D)\stackrel{-1}{\to} \xi]\oplus \OO_D(D)[-1]&\rTo{\nu}& \OO_D\oplus \OO_D(D)[-1],
\end{diagram}
\end{equation}
where $\nu$ has components $\phi:\xi(-D)\to \OO_D$, $\phi:\xi\to \OO_D(D)$ and $-\id:\OO_D(D)\to \OO_D(D)$.
Consider the chain map
$$[\xi(-2D)\rTo{1} \xi]\to [\xi(-D)\rTo{-1} \xi]\oplus \OO_D(D)[-1]$$
with the components $-1:\xi(-2D)\to \xi(-D)$, $\id:\xi\to \xi$, $\phi:\xi\to \OO_D(D)$. 
Then it induces a quasi-isomorphism of 
$[\xi(-2D)\to \xi]$ with the kernel of $\nu$, and hence an isomorphism 
$$H^0(\xi|_{2D})=H^1[\xi(-2D)\to \xi]\rTo{\sim} H^1\Big(\Cone[E_{10}\oplus E_{21}\to E_{11}][-1]\Big)=H^1(L^\bullet).$$

In fact, this isomorphism is induced by a chain map
\begin{equation}\label{xi-2D-E-chain-map}
[\xi(-2D-E)\rTo{1} \xi(-E)]\to L^\bullet
\end{equation}
with the components 
$$\xi(-E)\rTo{(0,\id,\wt{\phi})}(\xi(-D)\oplus \xi(-E))\oplus \OO(D)=V_0^0\oplus V_1^0\sub E_{10}\oplus E_{21}\sub L^\bullet,$$
$$\xi(-2D-E)\rTo{(-1,\wt{\phi},\wt{\phi})}\xi(-D-E)\oplus \OO\oplus \OO(-D)\sub V_0^0(-D)\oplus V_1^{-1}\oplus V_1^{-1}(-D)\sub
E_{10}\oplus E_{21}\oplus E_{11}[-1]\sub L^\bullet.$$

\noindent
{\bf Step 4. Computation of $\ker(\a)$}.
Next, we need to compute the kernel of the map $\a$ from this space to $H^1(E_{00}\oplus E_{22})$ (see \eqref{alpha-map-eq}).
We begin by observing that the trace morphisms $\tr:E_{ii}\to \OO$ induce isomorphisms on $H^1$, for $i=0,1,2$.
Hence, we can replace the map \eqref{alpha-map-eq} by the map induced on $H^1$ by the map of complexes
\begin{diagram}
E_{10}\oplus E_{21}&\rTo{}& E_{11}\\
\dTo{}\\
\OO\oplus \OO
\end{diagram}
given by $(x_{10},x_{21})\mapsto (\tr(x_{10}d),\tr(dx_{21}))$.
Now, since the differential in the top complex is given by $(x_{10},x_{21})\mapsto dx_{10}-x_{21}d$, the equality
$\tr(dx_{21})=\tr(x_{21}d)$ shows that the above map of complexes is homotopic to
$$(x_{10},x_{21})\mapsto (2\tr(x_{10}d),0).$$ 
Hence,
the kernel of $\a$ is identified with
the kernel of the map to $H^1(\OO)$, coming from the chain map 
$$E_{10}=\und{\Hom}(V_1,V_0)\to \OO$$
with components 
$$V_0^0(-D)\to \xi(-D-E)\rTo{\wt{\phi}} \OO, \ \ V_0^{-1}=\xi(-D-E)\rTo{-\wt{\phi}} \OO.$$ 

It follows that the composed chain map 
$[\xi(-2D-E)\to \xi(-E)]\to L^\bullet\to \OO$
is given by the component $-\wt{\phi}:\xi(-2D-E)\to \OO$.

Hence, the space $\ker(\a)$ can be identified with the kernel of the map induced on $H^1$ by the composition
$$[\xi(-2D-E)\to \xi(-E)]\to [\xi(-D-E)\to \xi(-E)]\to \OO.$$
We can rewrite this as a composition
$$[\xi(-2D-E)\to \xi(-E)]\to [\xi(-D-E)\to \xi(-E)]\to
[\OO\to \OO(D)]\to \OO,$$
where the middle chain map is given by $\wt{\phi}$.
At the level of the derived category, this corresponds to the composition
$$\xi|_{2D}[-1]\to \xi|_D[-1]\rTo{\phi} \OO_D(D)[-1]\rTo \OO,$$
which induces the map $\tau_D\circ\rho_\phi$ on $H^1$.
Thus, we get the claimed identification of the cotangent space $T^*B'=\ker(\a)$.

\noindent
{\bf Step 5. Identification of the tangent space}.

Recall that $TB'$ is identified with the kernel of the map
$$\ga:H^1C^\bullet\to H^1(E_{00}\oplus E_{22})\simeq H^1(\OO\oplus \OO),$$
where 
$$C^\bullet:=\Cone[E_{00}\oplus E_{11}\oplus E_{22}\rTo{d_h} E_{01}\oplus E_{12}][-1].$$

From the exact sequence of complexes
$$0\to \OO(D)\to V_1\to V_2\to 0,$$ 
we get the identification
$$E_{12}/E_{22}\rTo{\sim} \und{\Hom}(\OO(D),V_2)=V_2(-D).$$
As before, we will use the quasi-isomorphism $E_{11}\to \OO_D\oplus\OO_D(D)[-1]$. In addition,
we have natural quasi-isomorphisms $E_{00}\rTo{\tr}\OO$, $V_2(-D)\to [\OO\to \OO_D][1]$, and $E_{01}\to \und{\Hom}(V_0,\OO_D(D))\to \xi^{-1}(D)|_D$
induced by the projection 
$$\und{\Hom}(V_0^0,\OO_D(D))\to \und{\Hom}(\xi(-E),\OO_D(D))\simeq \xi^{-1}(D)|_D.$$ 

There is a homotopy $h$ from the composition $E_{00}\rTo{d\circ?} E_{01}\to \xi^{-1}(D)|_D$ to the composition $E_{00}\rTo{\tr}\OO\rTo{\phi} \xi^{-1}(D)|_D$ 
given by 
$$h:\und{\Hom}(V_0^{-1},V_0^0)=\OO(E)\oplus \OO(D)\to \OO(E)|_D=\OO_D\rTo{\phi} \xi^{-1}(D)|_D.$$
Next, the composition $E_{11}\rTo{-?\circ d} E_{01}\to \xi^{-1}(D)|_D$ coincides with the composition $E_{11}\to \OO_D\rTo{-\phi}\xi^{-1}(D)|_D$,
while the composition $E_{11}\rTo{d\circ ?} E_{12}/E_{22}\to V_2(-D)$ coincides with the composition $E_{11}\to V_1(-D)\to V_2(-D)$, so it
is homotopic to the composition $E_{11}\to \OO_D\to [\OO\to \OO_D][1]$, via the homotopy $h'$ with the component $E_{11}\to \und{\Hom}(V_1^0,V_1^0)=\OO\rTo{\id}\OO$.


Thus, we have quasi-isomorphisms
$$C^\bullet\to \Cone[E_{00}\oplus E_{11}\rTo{d_h} E_{01}\oplus E_{12}/E_{22}][-1] \to
\Cone[\OO\oplus (\OO_D\oplus \OO_D(D)[-1])\to \xi^{-1}(D)|_D\oplus [\OO\to\OO_D][1]][-1]\simeq \wt{C}^\bullet,$$
where
$$\wt{C}^\bullet:=\OO\cdot e_1\oplus \OO_D\oplus \OO\cdot e_2\rTo{\de}  \OO_D(D)\oplus \xi^{-1}(D)|_D\oplus \OO_D,$$
with 
$$\de=\left(\begin{matrix} 0 & 0 & 0 \\ \phi & \phi & 0 \\ 0 & -\id & -1 \end{matrix}\right).$$
The map $C^\bullet\to \wt{C}^\bullet$ is given by the components 
$$E_{00}\rTo{(\tr,-h)} \OO\cdot e_1\oplus \xi^{-1}(D)|_D, \ E_{01}[-1]\to \xi^{-1}(D)|_D[-1], \ E_{12}[-1]\to V_2(-D)[-1]=\OO(-D)\to \OO\cdot e_2,$$
$$E_{11}\to \und{\Hom}(V_1^0,V_1^0)\oplus \und{\Hom}(V_1^0,V_1^1)[-1]=\OO\oplus \OO(D)[-1]\rTo{(1,-\id),1} (\OO_D\oplus \OO\cdot e_2)\oplus \OO_D(D)[-1].$$

Furthermore, we claim that the map $\ga:H^1C^\bullet\to H^1(\OO\oplus \OO)$
is represented by the natural projection $\wt{C}^\bullet\to \OO\cdot e_1\oplus \OO\cdot e_2$.
Indeed, $\OO\cdot e_1$ corresponds to $E_{00}$, so the assertion follows from the fact that the composed map
$$\Cone[E_{22}\to E_{12}][-1]\to E_{22}\to \OO$$
can be identified with the projection $[\OO\cdot e_2\to \OO_D]\to \OO$.

Next, we observe that we have a quasi-isomorphism
$$\wt{C}^\bullet\to \ov{C}^\bullet:=[\OO\cdot e_1\oplus \OO\cdot e_2\rTo{\de'} \OO_D(D)\oplus \OO_D],$$
with
$$\de'=\left(\begin{matrix} 0 & 0 \\ 1 & -1 \end{matrix}\right),$$
given by the identity on both components of $\OO$, on $\OO_D[-1]$ and on $\OO_D(D)[-1]$, and by $\phi^{-1}:\xi^{-1}(D)|_D[-1]\to \OO_D[-1]$.

Finally, we have an obvious quasi-isomorphism
$$\OO(-D)\cdot e_1\oplus \OO\cdot (e_1+e_2)\oplus \OO_D(D)[-1] \hra \ov{C}^\bullet,$$
so that looking at the kernel of $\ga$, we get an identification
$$TB'\simeq \ker(H^1(\OO(-D))\to H^1(\OO))\oplus H^0(\OO_D(D)).$$

\noindent
{\bf Step 6. Identification of the Poisson structure: chain realization}.
Recall (see Lemma \ref{biv-Ch}) that the Poisson structure $T^*B'\to TB'$ is induced by the map of cones of the rows in the commutative square
\begin{equation}\label{Pois-str-chain-map-eq}
\begin{diagram}
E_{10}\oplus E_{21} &\rTo{} & E_{11}\\
\dTo{d_v} &&\dTo{d'_v}\\
E_{00}\oplus E_{11}\oplus E_{22}&\rTo{d_h}& E_{01}\oplus E_{12}
\end{diagram}
\end{equation}
where the vertical maps are
$$d_v(a_{10},a_{21})=(a_{10}d,da_{10},da_{21}), \ \ d'_v(b_{11})=(b_{11}d,0).$$
More precisely we use embedding 
$$T^*B'\sub H^1(L^\bullet)=H^1\Big(\Cone[E_{10}\oplus E_{21}\to E_{11}][-1]\Big), \ \ TB'\sub H^1(C^\bullet)=H^1\Big(\Cone[E_{00}\oplus E_{11}\oplus E_{22}\to E_{01}\oplus E_{12}][-1]\Big)$$
(in both cases we consider the kernel of the natural map to $H^1(E_{00}\oplus E_{22})$).
An important observation is that $d_v$ differs from $d_h^\vee$ in that it has zero component $E_{21}\to E_{11}$.

Recall that in Step 3 we constructed a quasi-isomorphism from $[\xi(-2D-E)\to \xi(-E)]$ to
 the shifted cone of the first row of \eqref{Pois-str-chain-map-eq}, $L^\bullet$.
Calculating the composition
\begin{equation}\label{Pois-str-chain-map-bis-eq} 
[\xi(-2D-E)\to \xi(-E)]\to L^\bullet\to C^\bullet\to \wt{C}^\bullet \to \ov{C}^\bullet, 
\end{equation}
we get the components 
$-\wt{\phi}:\xi(-2D-E)\to \OO(-D)\cdot e_1\sub\OO\cdot e_1$ and $-\phi:\xi(-E)\to \OO_D(D)$.

Thus, the induced map on $H^1$,
$$H^0(\xi|_{2D})\rTo{\sim} H^1(L^\bullet)\to H^1(\ov{C}^\bullet)\simeq H^1(\OO(-D)\cdot e_1\oplus \OO\cdot (e_1+e_2))\oplus H^0(\OO_D(D))$$
has components $H^0(\xi|_{2D})\to H^1(\xi(-2D-E))\rTo{-\wt{\phi}} H^1(\OO(-D))$ and $H^0(\xi|_{2D})\to H^0(\xi|_D)\rTo{-\phi} H^0(\OO_D(D))$,
where we use the connecting homomorphism corresponding to the exact sequence
$$0\to \xi(-2D-E)\to \xi(-E)\to \xi|_{2D}\to 0.$$

\noindent
{\bf Step 7. Dual point of view on $TB'$}.
Let us consider the perfect pairing
$$R(\cdot,\cdot):H^0(\xi^{-1}(2D)|_{2D})\ot H^0(\xi|_{2D}) \to k: R(s,t)=\Res_{2D}(s\cdot t).$$
It induces a perfect pairing between $T^*B'\sub H^0(\xi|_{2D})$ and $H^0(\xi^{-1}(2D)|_{2D})/\lan \phi\ran$, where
we view $\phi$ as an element of $H^0(\xi^{-1}(D)|_D)\sub H^0(\xi^{-1}(2D)|_{2D})$.

Let us compute the corresponding map 
$$\si:H^0(\xi^{-1}(2D)|_{2D})/\lan\phi\ran\rTo{\sim} TB'.$$ 
We claim that it comes from a natural map
$$H^0(\xi^{-1}(2D)|_{2D})\rTo{\wt{\phi}|_{2D}^{-1}} H^0(\OO(D)|_{2D})\rTo{\sim} H^1\Cone[E_{11}\to E_{01}\oplus E_{12}][-1]=H^1(L^\vee[-1])\to TB',$$
where the second arrow comes from the quasi-isomorphisms
$$\OO_{2D}(D)[-1]\simeq [\OO\rTo{(-1,\phi)}\OO(D)\oplus \xi^{-1}(D)|_D]\simeq 
\Cone[\ker(E_{11}\to E_{12})\to E_{01}][-1]\simeq L^\vee[-1],$$
where $\ker(E_{11}\to E_{12})\simeq [\OO\rTo{-1}\OO(D)]$.
We also claim that the composed arrow $H^0(\OO_{2D}(D))\to TB'\to H^1(\ov{C}^\bullet)$ is induced by the chain map
$$[\OO(-D)\rTo{-1} \OO(D)]\to \ov{C}^\bullet$$
with the components $1:\OO(-D)\to \OO\cdot e_1$ and $1:\OO(D)\to \OO_D(D)$.
In other words, the composition 
$$H^0(\OO(D)|_{2D})/\lan 1\ran \rTo{\wt{\phi}} H^0(\xi^{-1}(2D)|_{2D})/\lan \wt{\phi}\ran\rTo{\si} TB'\to H^1(\OO(-D)\oplus \OO_D(D)[-1])$$ 
is induced by the chain map
$$[\OO(-D)\rTo{1} \OO(D)]\to \OO(-D)\oplus \OO_D(D)[-1]$$ 
with components $-\id:\OO(-D)\to \OO(-D)$ and $1:\OO(D)[-1]\to \OO_D(D)[-1]$.
The two claims above follow immediately from the following two commutative diagrams of chain maps
\begin{diagram}
\Cone[[\OO\to \OO(D)]\to E_{01}][-1]&\rTo{}& [\OO\to \OO(D)\oplus \xi^{-1}(D)|_D]\\
\dTo{}&&\dTo{}\\
L^\vee[-1]&\rTo{}&\ov{C}
\end{diagram}
where the bottom arrow is the composition $L^\vee[-1]\to C\to \ov{C}$ and right vertical arrow has components $\id:\OO\to \OO\cdot e_1$, $1:\OO(D)[-1]\to \OO_D(D)[-1]$ and 
$\phi^{-1}:\xi^{-1}(D)|_D[-1]\to \OO_D[-1]$; and
\begin{diagram}
\Cone[[\OO\to \OO(D)]\to E_{01}][-1]&\rTo{}& [\OO\to \OO(D)\oplus \xi^{-1}(D)|_D]\\
\dTo{}&&\dTo{}\\
L^\vee[-1]&\rTo{}&\xi^{-1}(2D)|_{2D}
\end{diagram}
where the bottom arrow is obtained from the map dual to 
\eqref{xi-2D-E-chain-map}, and the right vertical arrow has components $\wt{\phi}:\OO(D)\to \xi^{-1}(2D)|_{2D}$ and the natural embedding $\xi^{-1}(D)|_D\to \xi^{-1}(2D)|_{2D}$. 

\noindent
{\bf Step 8. Final computation}.

Note that since the map $\tau_D$ is given by the residue at $D$, we can rewrite the formula \eqref{Pi-pairing-formula} as
$$\lan x,x'\ran=-R(\Phi(x),x')+R(\Phi(x'),x),$$
where $\Phi:T^*B'\to TB'$ is the composition of the map
$$T^*B'\to H^0(\xi^{-1}(2D)|_{2D})/\lan \phi\ran: x\mapsto \wt{\phi}\cdot F_D(\phi\cdot x|_D)$$
with the isomorphism $\si:H^0(\xi^{-1}(2D)|_{2D})/\lan \phi\ran\to TB'$.
In other words, the Poisson tensor is given by $-\Phi+\Phi^\vee$.

It remains to identify $-\Phi$ and $\Phi^\vee$ with two components of the map we calculated in Step 6.
In fact, we want to identify the map $\Phi$ with the map 
$$T^*B'\to H^0(\xi|_{2D})\to H^0(\xi|_D)\rTo{\phi}H^0(\OO_D(D))\sub TB'$$
and the map $\Phi^\vee$ with the map
$$T^*B'\to \ker(H^1(\OO(-D))\to H^1(\OO))\sub TB',$$
induced by the composition $H^0(\xi|_{2D})\to H^1(\xi(-2D-E))\rTo{-\wt{\phi}} H^1(\OO(-D))$.

For the first identification, we need to prove that the embedding $H_D\to H^0(\OO_D(D))\to TB'$ coincides with the composition
$$H_D\rTo{F_D}H^0(\OO(D)|_{2D})/\lan 1\ran\rTo{\wt{\phi}} H^0(\xi^{-1}(2D)|_{2D})/\lan \phi\ran.$$
Due to the description of the isomorphism $\si\wt{\phi}:H^0(\OO(D)|_{2D})/\lan 1\ran\to TB'$ in Step 7, this is equivalent to identifying the composition
$$H_D\rTo{F_D} H^1[\OO(-D)\to \OO(D)]\to H^1[\OO(-D)\rTo{0} \OO_D(D)]$$
with the natural embedding $H_D\to H^0(\OO_D(D))$. But this is clear from the definition of $F_D$, since the composition 
$$H^0(\OO(D))\to H^1[\OO(-D)\to \OO(D)]\to H^1[\OO(-D)\rTo{0} \OO_D(D)]$$
coincides with the natural map $H^0(\OO(D))\to H^0(\OO_D(D))$.
 
It is easy to see that under the dualities
$$H_D^\vee\simeq \ker(H^1(\OO(-D))\to H^1(\OO)), \ \ (H^0(\OO(D)|_{2D})/\lan 1\ran)^\vee\simeq \ker(H^1[\OO(-D)\to \OO(D)]\to H^1(\OO)),$$
the map $F_D^\vee$ dual to $F_D$ is induced by the natural map
$$H^1(\OO(-D)\to \OO(D))\to H^1(\OO(-D)).$$ 
Hence, $\Phi^\vee$ can be identified with the composition
\begin{align*}
&T^*B'\simeq \ker(H^1[\OO(-D)\to \OO(D)]\to H^1(\OO))\to \ker(H^1(\OO(-D))\to H^1(\OO))\simeq \\
&H^0(\OO_D)/\lan 1\ran\rTo{\phi} H^0(\xi^{-1}(D)|_D)/\lan \phi\ran\to H^0(\xi^{-1}(2D)|_{2D})/\lan \phi\ran\rTo{\si} TB'.
\end{align*} 
The result of Step 7 implies that the composition 
$$H^0(\OO_D)/\lan 1\ran \rTo{\phi} H^0(\xi^{-1}(D)|_D)/\lan \phi\ran\to H^0(\xi^{-1}(2D)|_{2D})/\lan \phi\ran\rTo{\si} TB'\to H^1(\OO(-D)\oplus \OO_D(D)[-1])$$
is induced by the chain map 
$$[\OO(-D)\to \OO]\to \OO(-D)\oplus \OO_D(D)[-1]$$
with the component $-\id:\OO(-D)\to \OO(-D)$.
This easily implies that the composition of $\Phi^\vee$ with the embedding $TB'\hra H^1(\OO(-D))\oplus H^0(\OO_D(D))$ is induced by negative of the map
$H^1(\OO(-D)\to \OO(D))\to H^1(\OO(-D))$, which leads to the claimed identification.
\end{proof}

 \subsection{Local coordinates}

Recall that $\VV$ denotes the vector bundle over $X^{(m)}$ with the fiber $H^0(\xi^{-1}(D)|_D)$ over $D$.
Given an open subset $U\sub X$ and a section $s\in \xi(U)$, we define a function $y(s)$ on the open subset 
$\tot(\VV|_{U^{(m)}})\sub \tot(\VV)$ by
$$y(s)(D,\phi)=\tau_D(\phi\cdot s).$$
It is easy to check that if $\phi$ is regular then the cotangent vector at $(D,\phi)$ associated with $y(s)$ depends only on
$s|_{2D}$ and this defines an isomorphism
$$H^0(\xi|_{2D})\rTo{\sim} T^*_{(D,\phi)}\tot(\VV).$$
Furthermore, under this identification, the cotangent map to the projection $\tot(\VV)\to X^{(m)}$ is given by
$$T^*_D X^{(m)}\simeq H^0(\OO_D)\rTo{\phi^{-1}} H^0(\xi(-D)|_D)\sub H^0(\xi|_{2D})\simeq  T^*_{(D,\phi)}\tot(\VV).$$
 
Now let us fix an open subset  $U_0\sub X$ and a nonvanishing section $s_0$ of $\xi|_{U_0}$.
We consider an \'etale open $\wt{\tot(\VV)}\to \tot(\VV)$ corresponding to a choice of distinct ordered points 
$p_1,\ldots,p_n\in U_0$ such that $D=p_1+\ldots+p_n$.

We denote by $x_i:\wt{\tot(\VV)}\to U_0$ the projection to $p_i$, and we define an additional set of $n$ functions,
$$y_i:\wt{\tot(\VV)}\to\A^1:(p_1,\ldots,p_n,\phi)\mapsto \Res_{p_i}(\phi s_0),$$
$i=1,\ldots,n$.
Then under the above identification of the cotangent space $T^*_{(D,\phi)}\tot(\VV)$ with $H^0(\xi|_{2D})$, we have
$$dx_i=\phi^{-1}|_{p_i}\in H^0(L(-p_i)|_{p_i}),$$
$$dy_i=s_0|_{2p_i}.$$
Note that the functions $x_i$ and $y_i/y_j$ descend to local functions on the projectivization $\P \VV$.

We have $\rho_\phi(dx_i)=0$, so 
$$\lan dx_i,dx_j\ran=0.$$
Hence, the Poisson bracket of $x_i$ with $x_j$ vanishes.

We have $\rho_\phi(dy_i)=(\phi s_0)|_{p_i}$, so 
$$\Res_{p_i}\rho_\phi(\frac{dy_i}{y_i})=1.$$
Thus, the differentials $dx_i$ and $dy_i/y_i-dy_j/y_j$ span (the pull-back of) the cotangent space to $\P\VV$.

Furthermore,
$$h_{ij}:=F(\rho_\phi(dy_i/y_i-dy_j/y_j)) \mod k$$
is a function in $H^0(\OO(p_i+p_j))$ that has the residue $1$ at $p_i$ and the residue $-1$ at $p_j$.

Therefore, we get
$$\lan \frac{dy_i}{y_i}-\frac{dy_j}{y_j}, dx_k\ran=\Res_{p_k} h_{ij}=\de_{ik}-\de_{jk}.$$
Note that
$$d(y_j/y_i)=-\frac{y_j}{y_i}\cdot(\frac{dy_i}{y_i}-\frac{dy_j}{y_j}).$$
Thus, the above formula can be rewritten as the following formula for the Poisson bracket:
$$\{\frac{y_j}{y_i},x_k\}=(\de_{jk}-\de_{ik})\frac{y_j}{y_i}.$$

Finally, we have
\begin{align*}
&\lan \frac{dy_i}{y_i}-\frac{dy_j}{y_j}, \frac{dy_i}{y_i}-\frac{dy_k}{y_k}\ran=
\tau_D\bigl([h_{ij}\cdot (s_0|_{2p_i}-s_0|_{2p_k})-h_{ik}\cdot (s_0|_{2p_i}-s_0|_{2p_j})]\cdot\phi\bigr)=\\
&(h_{ij}-h_{ik})(p_i)-h_{ij}(p_k)+h_{ik}(p_j).
\end{align*}
In other words,
$$\{\frac{y_j}{y_i},\frac{y_k}{y_i}\}=[(h_{ij}-h_{ik})(x_i)-h_{ij}(x_k)+h_{ik}(x_j)]\cdot \frac{y_j}{y_i}\cdot\frac{y_k}{y_i}.$$
Note that if use the uniformization of $X$ compatible with our choice of the global differential, then we can take
$$h_{ij}(x)=\zeta(x-x_i)-\zeta(x-x_j),$$
where $\zeta(x)$ is the Weierstrass zeta-function.
Then the above formula can be rewritten as
$$\{\frac{y_j}{y_i},\frac{y_k}{y_i}\}=2[\zeta(x_i-x_k)+\zeta(x_k-x_j)+\zeta(x_j-x_i)]\cdot \frac{y_j}{y_i}\cdot\frac{y_k}{y_i}.$$
Since $\zeta(x)$ is equal to the logarithmic derivative of the theta-function $\th(x)$ (with zero at $x=0$), we arrive at the following formulas for our Poisson structure.

\begin{prop}\label{Pois-br-bos-coord-prop}
Let us modify the variables $y_i$ by setting
$$y'_i:=y_i\cdot \prod_{j\neq i}\th(x_j-x_i).$$
Then in the variables $(x_i)$, $(y'_j/y'_i)$, the Poisson bracket is given by
$$\{x_i,x_j\}=0, \ \ \{\frac{y'_j}{y'_i},\frac{y'_k}{y'_i}\}=0,$$
$$\{\frac{y'_j}{y'_i},x_k\}=(\de_{jk}-\de_{ik})\frac{y'_j}{y'_i}.$$
\end{prop}

The formulas of Proposition \ref{Pois-br-bos-coord-prop} are compatible with those in \cite{Ode02}.

 
\section*{Appendix A: the second proof of Theorem \ref{thm:mixmoduliPoisson}}
There is an alternative way to interpret the stack morphism $p,q$ defined in Section \ref{sec:poisson str}. Since 
$\R\eps\uperf\simeq \ul\Map_{st}([\A^1\big/ \G_m], \R\uperf)$ by Lemma \ref{lem:1}, the inclusion of closed and open $\G_m$-orbits
\[
i: B\G_m\to [\A^1\big/ \G_m], \hspace{2cm} j:\Spec k\simeq \G_m\big/\G_m\to [\A^1\big/ \G_m]
\] induces two maps 
\[
\R\eps\uperf\to \Map_{st}(B\G_m,\R\uperf)=\R\uperf^\gr,\hspace{.5cm} \R\eps\uperf\to \Map_{st}(\G_m\big/\G_m,\R\uperf)\simeq \R\uperf.
\]
The first map coincides with $p$ obviously.  
\begin{prop}\label{q=genericfiber}
The inclusion of the open orbit $j: \G_m\big/ \G_m\to [\A^1\big/ \G_m]$ induces the $q$ map.
\end{prop}
\begin{proof}
Recall from the proof of Lemma \ref{lem:1}, we show the the equivalence of $\eps_{pe}-C(k)^\gr$ and $\perf([\A^1\big/\G_m])$ through the representation category of the group stack $\G_m\ltimes \G_a[-1]$. An alternative way to construct an equivalence is via the Beilinson-Gelfand-Gelfand correspondence. Set $k[\eps]$ to be the algebra of dual number  with $|\eps|=1$ and  $k[x]$ be the polynomial algebra with $|x|=-1$. Let $K:=k[x]\ot_k k[\eps]$ be the Koszul resolution of $k$, i.e. 
\[\xymatrix{
k[x]\cdot\eps \ar[r]^{\eps\mapsto x}  &k[x].}
\]
The functor $F:=K\ot_{k[\eps]}? : D^-(\gr-k[\eps])\simeq D^+(\gr-k[x])$ is an equivalence. The subcategory of compact objects $\perf(\gr-k[x])\subset D^+(\gr-k[x])$ is identified with the subcategory $D_{fd}(\gr-k[\eps])\subset D^-(\gr-k[\eps])$ consisting of objects with finite dimensional cohomology. Note that $\perf([\A^1\big/\G_m])$ is equivalent to $D^b(\gr-k[x])$. A graded mixed object in $C(k)$ is the same as a complexes of graded $k[\eps]$-module. The graded mixed object is perfect if and only if the corresponding complex of $k[\eps]$-modules has finite dimensional cohomology. Therefore the derived category of $\eps_{pe}-C(k)^\gr$ is equivalent with $D_{fd}(\gr-k[\eps])$. We consider the diagram
\[
\xymatrix{
D_{fd}(\gr-k[\eps])\ar[r]^F\ar[d]^{|-|} & \perf(\gr-k[x])\ar[d]^{\ot k[x,x^{-1}]}\\
\perf(k) & \perf(\gr-k[x,x^{-1}])\ar[l] _{(-)_0}
}
\] where $|-|$ is the functor taking total complex and $(-)_0$ sends a graded module 
\[
\ldots M_{-1}+M_0+M_1+\ldots
\] to $M_0$. Note that $(-)_0$ is an equivalence. To verify 
the diagram commutes it suffices to check on bounded complexes of free $k[\eps]$-modules of finite rank. 
Let $C$ be a complex
\[
\ldots \to k[\eps]^{\oplus n_i}\to k[\eps]^{\oplus n_{i+1}}\to\ldots
\]
Since $F(k[\eps])=K$, $F(C)$ is the total complex of the double complex
\[
\ldots \to K^{\oplus n_i}\to K^{\oplus n_{i+1}}\to\ldots
\] where the internal differential of $K$ is vertical. The diagram commutes since the composition of $(-)_0$ and $\ot k[x,x^{-1}]$ sends $K$ to the complex $\{\id: k\to k\}$.

Then $|-|$ induces $q$ map on the moduli stack and the functor $\ot k[x,x^{-1}]$ induces $j$.

\end{proof}

Now we briefly recall the notion of boundary structures following \cite[Section 2.2.3]{Ca14}. Let $\sX$ be a $\cO$-compact (see \cite[Definition 2.1]{PTVV}) stack that is equipped with a fundamental class $[\sX]$ of degree $d$, i.e. a morphism of complexes $\R\Gamma(\sX,\cO_\sX)\to k[-d]$ for some $n\in\Z$. 
Fix a morphism $f:\partial\sY\to \sY$ between $\cO$-compact stacks. 
A \emph{boundary structure of degree $d$} is a null homotopy $f_*[\partial\sY]\sim 0$ as a morphism from $\R\Gamma(\partial\sY,\cO_{\partial\sY})$ to $k[-d]$. Given a perfect complex on $\sY$, a boundary structure induces a morphism $\R\Gamma(\sY,E^\vee)\to \Gamma(\sY,E)^\perp$ where $\Gamma(\sY,E)^\perp$ is the mapping cone of the morphism
\[
\Gamma(\partial \sY, f^*E^\vee)\to \Gamma(\sY,E)^\vee[-d].
\]
We say that a boundary structure is \emph{non degenerate} if the above morphism is a quasi-isomorphism for any perfect complex $E$.

\begin{lemma}\label{lem: boundary}
The stack morphism $f=(i,j): B\G_m\bigsqcup \G_m\big/\G_m\to [\A^1\big/\G_m]$ admits a non-degenerate boundary structure of degree zero.
\end{lemma}
\begin{proof}
We need to construct two morphisms
\[
[I]: \R\Gamma(B\G_m, \cO_{B\G_m})\to k, \hspace{2cm} [J]: \R\Gamma(\G_m\big/\G_m, \cO_{\G_m\big/\G_m})=k\to k
\] such that $f_*[I\sqcup J]=f_*[I]+f_*[j]\sim 0$, i.e. the morphism 
\[\xymatrix{
\R\Gamma([\A^1\big/\G_m], \cO_{[\A^1\big/\G_m]})\ar[rrr]^{(i^*,j^*)}  &&& \R\Gamma(B\G_m, \cO_{B\G_m})\times \R\Gamma(\G_m\big/\G_m, \cO_{\G_m\big/\G_m})\\
&\ar[rr]^{([I],[J])} &&  k}
\] is homotopic to zero.

We use the equivalence between quasi-coherent sheaves on stacks and quasi-coherent sheaves on the corresponding simplicial schemes. The stack morphism $i, j$ can be realized as vertical maps of the simplicial schemes below
\[
\xymatrix{
\ldots \G_m\times\G_m\ar@<-.5ex>[r]\ar@<+.5ex>[r]\ar[d] & \G_m\ar[r]\ar[d] &\Spec k\ar[d]^i\\
\ldots \G_m\times\G_m\times \A^1\ar@<-.5ex>[r]\ar@<+.5ex>[r] & \G_m\times \A^1\ar[r] &\A^1\\
\ldots \G_m\times\G_m\times \A^1\setminus 0\ar@<-.5ex>[r]\ar@<+.5ex>[r]\ar[u] & \G_m\times \A^1\setminus 0\ar[r]\ar[u] &\A^1\setminus 0\ar[u]_j\\
}
\] The right most vertical maps are the canonical inclusions. 
We do not specify the other maps since they are not used. 
Use Dold-Kan correspondence we obtain $i^*, j^*$ as cochain maps
\[
\xymatrix{
k\ar[r] & k[\eps,\eps^{-1}]\ar[r] &\ldots\\
k[t]\ar[r]\ar[u]^{i^*}\ar[d]_{j^*} & k[t]\ot k[\eps,\eps^{-1}]\ar[r]\ar[u]\ar[d] &\ldots\\
k[t,t^{-1}]\ar[r] &k[t,t^{-1}]\ot k[\eps,\eps^{-1}]\ar[r] &\ldots\\
}
\] Set $[I]$ to be the chain map that is the 
identity map
$1: k\to k$ in degree zero and zero elsewhere  and set $[J]$ to be the canonical projection $k[t,t^{-1}]\to k$ times $-1$ in degree zero and zero elsewhere.
Obviously, $[I\sqcup J]:=([I],[J])$ is a fundamental class of $B\G_m\bigsqcup \G_m\big/\G_m$ such that $f_*[I\sqcup J]=0$.

To check that the boundary structure is non-degenerate it suffices to test for the perfect object $E=\cO_{[\A^1\big/\G_m]}$. The fundamental class $[I\sqcup J]$ induces a pairing on $\R\Gamma(B\G_m\bigsqcup \G_m\big/\G_m, \cO)$ whose restriction on degree zero part $k[t,t^{-1}]\times k$ is the bilinear form
\[
\Big(\sum_{i\in\Z} a_it^i, b_0\Big) \ot \Big(\sum_{i\in\Z} \alpha_it^i, \beta_0\Big)\mapsto -\sum_{i\neq 0} a_i\alpha_{-i} -(a_0\alpha_0-b_0\beta_0). 
\]
The non-degenerateness of the boundary structure then follows from the fact that the morphism $(j^*,i^*): k[t]\subset k[t,t^{-1}]\times k$ is maximal isotropic with respect to the pairing.
\end{proof}
\begin{proof}[Proof of Theorem \ref{thm:mixmoduliPoisson}]
By Lemma \ref{lem:1} and Theorem \ref{thm:epsmodulialgebraic}, $\R\eps\uperf\simeq \ul\Map_{st}([\A^1\big/ \G_m], \R\uperf)$ is a locally geometric stack locally of finite type. Fix a morphism $f:\partial\sY\to \sY$ between $\cO$-compact stacks, and $\sX$ a locally geometric stack locally of finite type equipped with a $d$-shifted symplectic structure. Suppose the mapping stacks  $\Map_{st}(\partial \sY, \sX)$ and $\Map_{st}(\sY, \sX)$ are both locally geometric and locally of finite type. Then every boundary structure on $f:\partial \sY\to Y$ defines an isotropic structure on $f^*: \Map_{st}(\sY, \sX)\to \Map_{st}(\partial \sY, \sX)$ and moreover non-degenerate boundary structure defines a Lagrangian structure (see \cite[Theorem 2.9]{Ca14}). By Lemma \ref{lem: boundary}, we have a non-degenerate boundary structure on the map $(i,j)$. And $\R\uperf$ is equipped with a 2-shifted symplectic structure constructed by Pantev, To\"en, Vaqui\'e and Vezzosi. It clearly extends to shifted structure on $\R\uperf^\gr_b$.  As a consequence of \cite[Theorem 2.9]{Ca14}, the stack morphism
\[
\R\eps\uperf \to \R\uperf^\gr_b\times \R\uperf
\] induced by $(i,j)$ is Lagrangian. Finally by Theorem \ref{thm:MS}, $\R\eps\uperf$ is 1-shifted Poisson.

\end{proof}

\section*{Appendix B: Invariance property of the Poisson bivector under autoequivalences}
Let $X$ be a reduced Gorenstein Calabi-Yau curve and $F$ be an autoequivalence of $\perf(X)$. By \cite[Corollary 9.13]{OL10}, $F$ can be represented as a kernel functor. Therefore it induces a quasi equivalence of stacks on $\R\uperf(X)$ and on $\R\eps\uperf(X)$.

\begin{theorem}
Set $\sX=\R\eps\uperf(X)$, $F:\sX\to \sX$ a stack quasi-equivalence induced by an autoequivalence and $\pi_h: L_\sX\to T_\sX$ to be the Poisson bivector induced by the 0-shifted Poisson structure. Then there exists a nonzero scalar $\lambda$ such that the diagram commutes in the derived category of $\sX$:
\[\xymatrix{
L_\sX\ar[r]^{\lambda\cdot \pi_h} & T_\sX\ar[d] \\
F^*L_\sX\ar[u]\ar[r]^{F^*\pi_h} & F^*T_\sX.
}\]
\end{theorem}

\begin{lemma} 
Set $\sY=\R\uperf(X)$, $F:\sY\to \sY$ a stack quasi-equivalence induced by an autoequivalence. Let $\omega_\beta\in $ be the closed (symplectic) form of degree 1 on $\sY$ induced by nonzero section $\beta\in H^0(X,\omega_X)$ and the canonical 2-shifted symplectic form on $\R\uperf$. Denote by $\theta_{\omega_\beta}$ the underlying 2-form, viewed as a chain map $\theta_{\omega_\beta}: T_\sY\to L_\sY[1]$. Then $F^*\theta_{\omega_\beta}= \theta_{\omega_{\lambda_F\beta}}$ on cohomology groups for some nonzero scalar $\lambda_F$.
\end{lemma}
\begin{proof}
For a reduced complex projective curve $X$, the first Hochschild homology $HH_1(X)$ is isomorphic to $H^0(\omega_X)\cong k$ (see \cite[Example 1.6]{Weibel97}). An autoequivalence $F$ induces an action on $HH_1$ therefore acts on $H^0(X,\omega_X)$ as scalar multiplication by $\lambda_F\neq 0$.

We compute the action of $F$ on $\beta$. Denote by $S$ the Serre functor. We refer the readers to \cite[Chapter 5]{Ballard09} for the definition and properties of Serre functor on an additive category. Let $X$ be a projective Gorenstein $k$-scheme.
By \cite[Lemma 6.6]{Ballard09}, $\perf(X)$ admits a Serre functor $S:=?\ot \omega_X[\dim X]$. 
Given two perfect complexes $A,B$, denote by
\[
\eta_{A,B}: \Hom(A,B)\to \Hom(B,SA)^\vee
\] the natural isomorphism in the definition of Serre functor. It satisfies the compatibility condition 
\[
\eta_{B,SA}^\vee\circ s_*=\eta_{A,B}
\] where $s_*: \Hom(A,B)\to\Hom(SA,SB)$ is a natural isomorphism.
There exists a natural isomorphism 
\[
\tau: F\circ S\to S\circ F
\] defined by
\[\xymatrix{
\Hom(A,FSB)\ar[r]^{(f^{-1})_*} & \Hom(F^{-1}A,SB)\ar[rr]^{\eta_{F^{-1}A,B}} & &\Hom(SB,SF^{-1}A)^\vee\\
\ar[r]^{s_*^\vee} & \Hom(B,F^{-1}A)^\vee\ar[r]^{(f^{-1})_*^\vee} &\Hom(FB,A)^\vee\ar[r]^{(\eta_{FB,A}^\vee)^{-1}} & \Hom(A,SFB)
}
\]
Note that $\beta$ defines a natural isomorphism from the shift functor to $S$. We set $\beta^F$ to be the image of $\beta$ under the induced action of $F$. Denote by $\beta_A$ the morphism from $A[1]$ to $SA$. Then
\[
(\beta^F)_A=(f^{-1})_*\circ \tau^{-1}\circ \beta_{FA}\circ f_*,
\] where $f_*: \Hom(A,B)\to\Hom(FA,FB)$ is a natural isomorphism, i.e. we have the commutative diagram
\[
\xymatrix{
\Hom(FB,FA[1])\ar[r]^{\beta_{FA}} & \Hom(FB,SFA)\ar[r]^{\tau^{-1}} & \Hom(FB,FSA)\ar[d]^{(f^{-1})_*}\\
\Hom(B,A[1])\ar[u]_{f_*}\ar[rr]^{(\beta^F)_A} &  &\Hom(B,SA)
}
\]
By a diagram chasing we obtain the following commutative diagram: 
\[\xymatrix{
\Hom(A,B)\ar[rr]^{f_*}\ar[d]^{\eta_{A,B}} && \Hom(FA,FB)\ar[d]^{\eta_{FA,FB}}\\
\Hom(B,SA)^\vee \ar[d]^{(\beta^F_A)^\vee} & & \Hom(FB,SFA)^\vee\ar[d]^{\beta_{FA}^\vee}\\
\Hom(B,A[1])^\vee & & \Hom(FB,FA[1])^\vee\ar[ll]^{f_*^\vee}
}
\]
Now set $A=B=E$ and $S=?\ot \omega_X[1]$ where $X$ is a Gorenstein CY curve. The composition of the left column is $\theta_{\omega_{\beta^F}}$ and the composition of the right column is $F^*\theta_{\omega_\beta}$. Since $\beta^F_A=\lambda_F\cdot \beta_A$, we prove the desired statement.
\end{proof}

\begin{remark}
The group of autoequivalences of a complex elliptic curve is well understood. One can check that for $F$ being the equivalence induced by the Kummer involution $\lambda_F=-1$. If the elliptic curve admits complex multiplication then the corresponding $\lambda_F$ will be a primitive root of unity.
\end{remark}

\begin{proof}[Proof of Theorem] 
Recall that  $\sX=\R\eps\uperf(X)=\Map_{st}([\A^1/\G_m], \sY)$, $\sY=\R\uperf(X)$ and $\sY^\gr=\Map_{st}(B\G_m,\sY)$. Set $f=(p,q): \sX\to \sY^\gr\times \sY$ for simplicity. Then the symplectic form $\omega_\beta$ together with the null homotopy $h: f^*\omega_\beta \sim 0$ determines a quasi-isomorphism $\Theta_{h,\beta}: T_\sX\to L_f$. The Poisson bivector $\pi_{h,\beta}: L_{\sX}\to T_\sX$ is defined (in the derived category) as the composition of $\Theta_{h,\beta}^{-1}$ and the canonical map $L_\sX\to L_f$. Here we add the lower index $\beta$ to emphasize that the bivector depends on the Calabi-Yau structure $\beta$. From Lemma \ref{lem: boundary}, we see that $h$ is determined by the boundary structure on $B\G_m\bigsqcup \G_m\big/\G_m\to [\A^1\big/\G_m]$, therefore is independent of $\beta$.
By the above lemma, $F^*$ acts on  $\R\eps\uperf(X)=\Map_{st}([\A^1\big/ \G_m],\sY)$ via its action on $\sY$. Therefore we have
\[
F^*\pi_{h,\beta}=\pi_{h,\lambda_F\beta}=\lambda_F^{-1} \pi_{h,\beta}.
\]
\end{proof}

\section*{Appendix C: Fine moduli space of bosonizations}
Recall that the stacky bosonization map $\beta_\xi:\sB(\xi)\to \sN(\xi)$ descends to an ordinary Poisson morphism $\beta_\xi: B(\xi)\to N(\xi)$ between the coarse moduli schemes. In order to study quantization of these Poisson structures one needs to know the determinant of the universal family, which can be extracted by rigidifying the moduli problem.

For a primitive vector $v=(r,d)$ with $r\geq 0$, let $M(v)$ denote the coarse moduli space of stable sheaf of rank $r$ and degree $d$ on $X$.
To normalize a universal sheaf over $M(v)\times X$, we need to rigidify the moduli problem.
One natural way of doing this is to fix a vector $v_1=(r_1,d_1)$ with $r_1\geq 0$ such that $\chi(v_1,v)=1$ and fix a stable sheaf $V_1$ with vector $v_1$.
Then in the rigidified moduli problem we consider stable sheaves $V$ with vector $v$ equipped with a trivialization of the $1$-dimensional space $\Hom(V_1,V)$.
This rigidification gives rise to the universal sheaf $\cU_{v_1}=\cU_{v_1}(v)$ on $M(v)\times X$.
The natural question is how $\cU_{v_1}$ depends on $v_1$.

Alternatively, we can fix a vector $v_2=(r_2,d_2)$ with $r_2\geq 0$ such that $\chi(v,v_2)=1$, fix a stable sheaf $V_2$ with vector $v_2$ and rigidify the moduli problem
by trivializing $\Hom(V,V_2)$. Let us denote by $\cU'_{v_2}=\cU'_{v_2}(v)$ the corresponding universal sheaf on $M(v)\times X$. 
The question is how $\cU_{v_1}$ and $\cU'_{v_2}$ are related. 

Note that by associating with a stable sheaf in $M(v)$ its determinant, we can define an isomorphism $M(v)\simeq X$ uniquely up to a translation. 
We call any such isomorphism {\it standard}.
In particular, it makes sense to talk about degrees of line bundles on $M(v)$.

\begin{lemma}\label{univ-bundle-lem1}
Let $V_{v_2}$ be a stable sheaf on $X$ with vector $v_2$ such that $\chi(v,v_2)=1$.
\begin{enumerate}
\item[$(i)$] The line bundle 
$$L(v_1,v_2):=p_{1*}\underline{\Hom}(\cU_{v_1},p_2^*V_{v_2})$$ 
on $M(v)$ has degree $-\chi(v_1,v_2)$. One has
$$\UU'_{v_2}\simeq \cU_{v_1}\ot p_1^*L(v_1,v_2).$$

\item[$(ii)$] One has 
$$\cU_{v'_1}\simeq \cU_{v_1}\ot p_1^*\Lambda,$$
with $\Lambda$ a line bundle of degree $m$, where $v'_1=v_1+mv$.
\end{enumerate}
\end{lemma}

\begin{proof}
(i) Using an appropriate autoequivalence we reduce to the case $v_1=(1,0)$, $v=(1,1)$, $v_2=(d-1,d)$. Then $M(v)=X$ and $\cU_{v_1}=\OO_{X^2}(\De)$.
By the Grothendieck-Riemann-Roch
\[
ch(L(v_1,v_2))={p_1}_*ch(p_2^*V_{v_2}(-\Delta)).
\]
So it has degree $-d=-\chi(v_1,v_2)$.
Now the bundle $\cU_{v_1}\ot p_1^*L(v_1,v_2)$ has the same rigidification as $\cU'_{v_2}$, so they are isomorphic.

\noindent
(ii) This immediately follows from (i).
\end{proof}
If $V_1$ and $V_1^\prime$ are stable vector bundles of vector $v_1$ then the universal sheaves differ by $p_1^*\Lambda$ for a unique line bundle $\Lambda$ of degree zero on $M(v)$.
Let us denote by $\PP=\OO(\De-p_0\times X-X\times p_0)$ the Poincar\'e line bundle on $X\times X$.

\begin{lemma}\label{poincare-moduli-lem}
For a pair of vectors with $\chi(v_1,v_2)=1$ the line bundle
$$\PP(v_1,v_2)=p_{12*}\underline{\Hom}(p_{13}^*\cU'_{v_2}(v_1),p_{23}^*\UU_{v_1}(v_2))$$
on $M(v_1)\times M(v_2)$ get identified with $\PP$ under some standard isomorphisms $M(v_i)\simeq X$.
Here $p_{ij}$ are projections from $M(v_1)\times M(v_2)\times X$ to pairwise products of factors.
\end{lemma}

\begin{proof} 
Using an appropriate autoequivalence we reduce to the case $v_1=(1,0)$, $v_2=(1,1)$.
Then taking $\OO$ and $\OO(p_0)$ as fixed line bundles of degrees $0$ and $1$, we have 
$$\cU_{v_1}(v_2)=\OO_{X^2}(\De), \ \ \cU'_{v_2}(v_1)=\OO_{X^2}(X\times p_0-\De).$$
Thus, we need to identify the line bundle
$$\PP(v_1,v_2):=p_{12*}(\OO_{X^3}(\De_{23}+\De_{13}-X^2\times p_0)).$$
The exact sequence of vector bundles 
$$0\to \PP(v_1,v_2)\to p_{12*}(\OO_{X^3}(\De_{23}+\De_{13}))\to \OO_{X^2}(X\times p_0+p_0\times X)\to 0$$
gives an isomorphism
$$\PP(v_1,v_2)\simeq (\det p_{12*}(\OO_{X^3}(\De_{23}+\De_{13})))(-X\times p_0-p_0\times X).$$
On the other hand, the exact sequence
$$0\to p_{12*}\OO_{X^3}(\De_{23})\to p_{12*}(\OO_{X^3}(\De_{23}+\De_{13}))\to \OO_{X^2}(\De)\to 0$$
gives an isomorphism
$$\det p_{12*}(\OO_{X^3}(\De_{23}+\De_{13}))\simeq \OO_{X^2}(\De),$$
so we deduce an isomorphism
$$\PP(v_1,v_2)\simeq \PP.$$
\end{proof}

Set $v_i=(k(i),n(i))$ for $i=0,\ldots,p$. As a convention we set $v_{p+1}=(-1,0)$. 
We have an isomorphism
\[
\pi=(\pi_1,\ldots,\pi_p): B(\xi)\to M(v_{1})\times\ldots \times M(v_p)
\] sending 
\[
\xi_0=\xi\to \xi_1\to\ldots \to \xi_p\to\xi_{p+1}=\cO[1]
\] to $(\xi_1,\ldots,\xi_p)$. Fix a collection of stable sheaves $V_{v_1},\ldots,V_{v_p}$ of vector $v_1,\ldots,v_p$. Then we identify $M(v_i)$ with the fine moduli space of $\{V_{v_{i+1}}\to V\}$ for stable $V$ with vector $v_i$, which leads to a universal sheaf $\cU_{v_{i+1}}(v_i)$ on $M(v_i)\times X$. For different choices of $V_{v_i}$, the corresponding universal sheaf differ by tensoring with the pullback of a degree zero line bundle over $M(v_i)$. 

\begin{theorem}
Fix standard isomorphisms $M(v_i)\cong X$ for $i=1,\ldots, p$. The pull back of $\cO(1)$ on $\P\Ext^1(\xi,\cO_X)$ by the bosonization map $\beta_\xi$ is isomorphic to the line bundle $L_1\boxtimes L_2\ldots\boxtimes L_p(-\sum_{i=1}^{p-1} \Delta_{i,i+1})$ where $L_i$ are line bundles on $M(v_i)\cong X$ with degree $n_{p-i}+1$ for $i=1,p$ and degree $n_{p-i}+2$ for $1<i<p$.
\end{theorem}
\begin{proof}
We adopt the notation in the proof of Lemma \ref{poincare-moduli-lem}. First we assume that $p>1$. Fixing $\xi_i$, the universal sheaf of the moduli space of sequences $\{\xi_i\to \xi_{i+1}^\prime\to \xi_{i+2}^\prime\}$ (with $\xi_{i+1}^\prime$ and $\xi_{i+2}^\prime$ varying) is given by 
\[
\LL_{i+1,i+2}:=(p_{1,\ldots,p})_*\underline\Hom(p_{i+1,p+1}^*\cU_{v_i}(v_{i+1}), p_{i+2,p+1}^*\cU_{v_{i+1}}(v_{i+2})).
\]
For $i=p-1$, we have $\xi_{p-1}\to \xi_p^\prime\to \xi_{p+1}=\cO[1]$ where only $\xi_p^\prime$ varies. Set
\[
\LL_{p,p+1}=(p_{1,\ldots,p})_*\underline\Hom(p_{p,p+1}^*\cU_{v_{p-1}}(v_{p}), p_{p,p+1}^*\cU^\prime_{v_{p+1}}(v_{p})).
\]
By Lemma \ref{univ-bundle-lem1}, 
\[
\cU_{v_i}(v_{i+1})=\cU^\prime_{v_{i+2}}(v_{i+1})\ot p_{i+1}^*\Lambda_{i+1}
\] for a line bundle $\Lambda_{i+1}$ of degree $\chi(v_i,v_{i+2})$.
By the Laplace formula (\ref{Laplace}),
\[
\chi(v_i,v_{i+2})=n_{p-i}.
\]
The bosonization map $\beta$ sends a chain to the composition of all arrows. Therefore the pull back of $\cO(-1)$ is the tensor product 
\[
\LL_{1,2}\ot\LL_{2,3}\ot\ldots\ot \LL_{p-1,p}.
\]
By Lemma \ref{poincare-moduli-lem}, for $i=0,\ldots, p-2$
\[
\LL_{i+1,i+2}\cong \PP(v_{i+1},v_{i+2})\ot p_{i+1}^*\Lambda^\vee_{i+1},
\]
Moreover $\LL_{p,p+1}\cong \PP(v_{p-1},v_{p})\ot p_p^*\Lambda^\vee_p$ for a line bundle $\Lambda_p^\vee$ of degree $-n_1$ since
by Lemma \ref{univ-bundle-lem1}, $\cU^\prime_{v_{p+1}}(v_{p})\cong \cU_{v_{p-1}}(v_{p})\ot p_p^*\Lambda^\vee_p$. Then we prove the desired statement for $p>1$.

When $p=1$, the pullback of $\cO(-1)$ is 
\[
(p_1)_*\underline\Hom(\cU_{v_{0}}(v_{1}), \cU^\prime_{v_{2}}(v_{1})).
\] Since $v_1=(0,1), v_0=(1,n), v_2=(-1,0)$, $\cU^\prime_{v_2}(v_1)\cong \underline\Ext^1(\cO_\Delta,\cO_{X^2})\cong\cO_\Delta$. By Lemma \ref{univ-bundle-lem1}, $\cU_{v_{0}}(v_{1})\cong \cO_\Delta\ot p_1^*\Lambda$ for some line bundle $\Lambda$ of degree $n$.
\end{proof}


\begin{thebibliography}{XXXX}
\bibitem{Alper09}
Alper, Jarod. {\it Good moduli spaces for Artin stacks}, Annales de l'Institut Fourier, vol. 63, no. 6, pp. 2349-2402. 2013.
\bibitem{Ballard09}
Ballard, Matthew Robert. {\it Derived categories of sheaves on singular schemes with an application to reconstruction}, Advances in Mathematics 227, no. 2 (2011): 895-919.
\bibitem{Ca14} Calaque, Damien. {\it Lagrangian structures on mapping stacks and semi-classical TFTs}, Stacks and categories in geometry, topology, and algebra 643 (2015): 1-23.
87 (1968), 305--320.
\bibitem{CKS19} Chirvasitu, Alex, Ryo Kanda, and S. Paul Smith. \emph{The characteristic variety for Feigin and Odesskii's elliptic algebras}, arXiv preprint arXiv:1903.11798 (2019).
\bibitem{CPTVV} D.~Calaque, T.~Pantev, B.~Toen, M.~Vaqui\'e, G.~Vezzosi, \emph{Shifted Poisson structures and deformation quantization}, Journal of Topology 10.2 (2017): 483-584.
\bibitem{FO87} B.~L.~Feigin, A.~V.~Odesskii, {\it Sklyanin's elliptic algebras}, 
Funct. Anal. Appl. 23 (1989), no. 3, 207--214.
\bibitem{FO95} B.~L.~Feigin, A.~V.~Odesskii, {\it Vector bundles on an elliptic curve and Sklyanin algebras}, in 
{\it Topics in quantum groups and finite-type invariants}, 65--84, Amer. Math. Soc., Providence, RI, 1998.
\bibitem{HAGII} B. ~Toen, G.~ Vezzosi, \emph{Homotopical algebraic geometry II: Geometric stacks and
applications}, Mem.~ Amer.~ Math.~ Soc. 193 (2008), no. 902.
\bibitem{HP1} Z.~Hua, A.~Polishchuk, {\it Shifted Poisson structures and moduli spaces of complexes}, Adv. Math. 338 (2018), 991--1037.
\bibitem{HP2} Z.~Hua, A.~Polishchuk, {\it Shifted Poisson geometry and meromorphic matrix algebras over an elliptic curve}, 
Selecta Math. 25 (2019), no. 3, Paper No. 42.
\bibitem{HP3} Hua, Zheng, and Alexander Polishchuk. {\it Elliptic bihamiltonian structures from relative shifted Poisson structures}, arXiv preprint arXiv:2007.12351 (2020).
\bibitem{HL-P19}
Halpern-Leistner, Daniel, and Anatoly Preygel. {\it Mapping stacks and categorical notions of properness}, arXiv preprint arXiv:1402.3204 (2014).
\bibitem{Lu}
Lurie, Jacob. \emph{Derived algebraic geometry}, PhD diss., Massachusetts Institute of Technology, 2004.
\bibitem{MS2} Melani, Valerio, and Pavel Safronov. \emph{Derived coisotropic structures II: stacks and quantization.} Selecta Mathematica 24, no. 4 (2018): 3119-3173.
\bibitem{NS06} Nevins, T. A., and J. T. Stafford. \emph{Sklyanin algebras and Hilbert schemes of points}, Advances in Mathematics 210, no. 2 (2007): 405-478.
\bibitem{Ode02} Odesskii, Alexander Vladimirovich. \emph{Elliptic algebras}, Russian Mathematical Surveys 57, no. 6 (2002): 1127.
\bibitem{OL10} Lunts, Valery, and Dmitri Orlov. {\it Uniqueness of enhancement for triangulated categories}, Journal of the American Mathematical Society 23, no. 3 (2010): 853-908.
\bibitem{PTVV} T.~Pantev, B.~To\"en, M.~Vaqui\'e, G.~Vezzosi, \emph{Shifted Symplectic Structures}, Publications Math. IHES 117 no. 1 (2013), 271--328.
\bibitem{Pym-Schedler} B.~Pym, T.~Schedler, {\it Holonomic Poisson manifolds and deformations of elliptic algebras}, in {\it Geometry and physics}, Vol. II, 681--703, Oxford Univ. Press, 
Oxford, 2018.
\bibitem{Polishchuk} A.~Polishchuk, {\it Algebraic geometry of Poisson brackets},
J. Math. Sci. (New York) 84 (1997), no. 5, 1413--1444.
\bibitem{Pol98} A.~Polishchuk, {\it Poisson structures and birational morphisms associated with bundles on elliptic curves}, IMRN 13 (1998), 683--703.
\bibitem{Pol-AV} A.~Polishchuk, {\it Abelian Varieties, Theta Functions and the Fourier Transform}, Cambridge University Press, Cambridge, 2003.
\bibitem{TV96} Tate, John, and Michel Van den Bergh. \emph{Homological properties of Sklyanin algebras}, Inventiones mathematicae 124, no. 1 (1996): 619-648.
\bibitem{TV07} B.~To\"en, M.~Vaqui\'e, \emph{Moduli of objects in dg-categories}, Ann. ~Sci.~ ENS~ 40 (2007), 387--444.
\bibitem{Weibel97} Weibel, Charles. {\it The Hodge filtration and cyclic homology}, K-theory 12, no. 2 (1997): 145-164.
\end{thebibliography}
\end{document}